\documentclass[11pt]{article}

\usepackage{titling}
\usepackage{amssymb,enumerate}
\usepackage{amsmath,amsfonts}
\usepackage{amsthm}
\usepackage{latexsym}
\usepackage{mathrsfs}
\usepackage{color}
\usepackage{microtype}
\usepackage{geometry}
\usepackage{algorithm}
\usepackage{algorithmicx}
\usepackage{algpseudocode}
\usepackage{xcolor}
\usepackage{booktabs}
\usepackage{multirow}
\usepackage{pifont}
\def\1{\bm{1}}

\DeclareMathAlphabet{\mathsfit}{\encodingdefault}{\sfdefault}{m}{sl}
\SetMathAlphabet{\mathsfit}{bold}{\encodingdefault}{\sfdefault}{bx}{n}

\newcommand{\E}{\mathbb{E}}

\newcommand{\R}{\mathbb{R}}

\theoremstyle{plain}
\newtheorem{theorem}{Theorem}

\newtheorem{lemma}{Lemma}

\newtheorem{assumption}{Assumption}

\newtheorem{remark}{Remark}

\usepackage{graphicx, color}
\usepackage{algorithm, algpseudocode} % use algorithm and algorithmicx for typesetting algorithms
\usepackage{mathrsfs} % for \mathscr command
\allowdisplaybreaks[4]

\usepackage{lipsum}

\def\Dp{{\mathcal D}^p}
\def\Rp{{\mathcal R}_p}

\title{Complexity of normalized stochastic first-order methods with momentum under heavy-tailed noise}

\author{Chuan He\thanks{Department of Mathematics, Link\"oping University, Sweden (email: {\tt chuan.he@liu.se}). The work of Chuan He was partially supported by the Wallenberg AI, Autonomous Systems and Software Program (WASP) funded by the Knut and Alice Wallenberg Foundation.}
\and
Zhaosong Lu\thanks{Department of Industrial and Systems Engineering, University of Minnesota, USA (email: {\tt zhaosong@umn.edu}).}
\and
Defeng Sun\thanks{Department of Applied Mathematics, The Hong Kong Polytechnic University, Hong Kong, People’s Republic of China (email: {\tt defeng.sun@polyu.edu.hk}). The research of Defeng Sun   was partially supported by the Research Center
for Intelligent Operations Research at The Hong Kong Polytechnic University (P0051214).}
\and
Zhanwang Deng\thanks{Academy for Advanced Interdisciplinary Studies, Peking University, Beijing, People’s Republic of China (email: {\tt dzw\_opt2022@stu.pku.edu.cn}).}
}

\date{
June 12, 2025 (Revised: February 11, 2026)
}

\begin{document}
\maketitle

\begin{abstract} 
In this paper, we propose practical normalized stochastic first-order methods with Polyak momentum, multi-extrapolated momentum, and recursive momentum for solving unconstrained optimization problems. These methods employ dynamically updated algorithmic parameters and do not require explicit knowledge of problem-dependent quantities such as the Lipschitz constant or noise bound. We establish first-order oracle complexity results for finding approximate stochastic stationary points under heavy-tailed noise and weakly average smoothness conditions---both of which are weaker than the commonly used bounded variance and mean-squared smoothness assumptions. Our complexity bounds either improve upon or match the best-known results in the literature. Numerical experiments are presented to demonstrate the practical effectiveness of the proposed methods.
\end{abstract}

\noindent{\small \textbf{Keywords:} Stochastic first-order methods, momentum, heavy-tailed noise, first-order oracle complexity}

\medskip

\noindent{\small {\bf Mathematics Subject Classification:} 49M05, 49M37, 90C25, 90C30}

\section{Introduction}

In this paper, we consider the smooth unconstrained optimization problem:
\begin{align}\label{ucpb}
\min_{x\in\R^n}f(x),
\end{align}
where $f:\mathbb{R}^n\to\mathbb{R}$ is continuously differentiable. We assume that problem \eqref{ucpb} has at least one optimal solution. In many emerging applications---particularly in machine learning and related fields---instances of \eqref{ucpb} are often large- or even huge-scale, which poses significant challenges to classical first-order methods due to the high cost of computing the exact gradient of $f$. To address this issue, stochastic first-order methods (SFOMs) have been extensively studied, as they employ stochastic estimators of the gradient that are typically much cheaper to compute. The goal of this paper is to propose practical \textit{normalized} stochastic first-order methods with momentum for solving problem \eqref{ucpb}, and to analyze their complexity under \textit{heavy-tailed noise}.

Recently, a variety of SFOMs (e.g., \cite{bottou2018optimization,cutkosky2020momentum,cutkosky2019momentum,fang2018spider,gao2024non,ghadimi2013stochastic,ghadimi2016accelerated,khaled2022better,lei2017non,nguyen2017sarah}) have been developed for solving problem \eqref{ucpb} where the stochastic gradient $G(\cdot;\xi)$ is an unbiased estimator of $\nabla f(\cdot)$ with bounded variance, i.e., $G(\cdot;\xi)$ satisfies the conditions:
\begin{align}\label{asp:ub-bd}
\E[G(x;\xi)]=\nabla f(x),\quad \E[\|G(x;\xi)-\nabla f(x)\|^2]\le\sigma^2\qquad\forall x\in\mathbb{R}^n
\end{align}
for some $\sigma>0$. The complexity of these SFOMs has been extensively studied under various smoothness conditions. In particular, under the assumption that 
$\nabla f$ is Lipschitz continuous, the methods proposed in
\cite{cutkosky2020momentum,ghadimi2013stochastic,ghadimi2016accelerated} achieve a first-order oracle complexity of $\mathcal{O}(\epsilon^{-4})$ for finding an \textit{$\epsilon$-stochastic stationary point} (SSP) of \eqref{ucpb}, i.e., a point $x$ satisfying
\begin{align*}
\E[\|\nabla f(x)\|] \le \epsilon.   
\end{align*}
In addition, assuming that $\nabla^2 f$ is Lipschitz continuous, the method in \cite{cutkosky2020momentum} achieves a first-order oracle complexity of $\mathcal{O}(\epsilon^{-7/2})$ for finding an $\epsilon$-SSP of \eqref{ucpb}. Furthermore, several variance-reduced methods \cite{cutkosky2019momentum,fang2018spider,li2021page} have been shown to achieve a first-order oracle complexity of $\mathcal{O}(\epsilon^{-3})$ for finding an $\epsilon$-SSP of \eqref{ucpb}, under the assumption that the stochastic gradient estimator $G(\cdot;\xi)$ satisfies a mean-squared smoothness condition.

Despite extensive studies on SFOMs under the bounded variance assumption \eqref{asp:ub-bd}, recent empirical findings in machine learning (e.g., \cite{gurbuzbalaban2021heavy,simsekli2019heavy,simsekli2019tail,zhang2020adaptive}) suggest that in practice the stochastic gradient estimator $G(\cdot;\xi)$ is typically unbiased but exhibits a bounded $\alpha$th central moment for some $\alpha \in (1, 2]$, rather than bounded variance. Specifically, $G(\cdot;\xi)$ satisfies 
\begin{align*}
\E[G(x;\xi)]=\nabla f(x),\quad\E[\|G(x;\xi)-\nabla f(x)\|^\alpha]\le\sigma^\alpha\qquad\forall x\in\mathbb{R}^n    
\end{align*}
for some $\sigma>0$, which is commonly referred to as the \textit{heavy-tailed noise regime}. The convergence behavior and analysis of SFOMs under this condition differ significantly from those based on the bounded variance assumption \eqref{asp:ub-bd}, as vanilla stochastic gradient methods can diverge when 
$\alpha \in (1, 2)$; see \cite[Remark 1]{zhang2020adaptive}. To address this divergence, gradient clipping has been widely adopted in algorithm design (e.g., \cite{cutkosky2021high,liu2024high,liu2023breaking,NEURIPS2023_4c454d34,sadiev2023high,zhang2020adaptive}). This approach replaces the stochastic gradient estimator with
\begin{align}\label{clip}
\overline{G}_\tau(x;\xi):=\min\Big\{1,\frac{\tau}{\|G(x;\xi)\|}\Big\} G(x;\xi)    
\end{align}
for some suitable clipping threshold $\tau>0$. Numerous recent works have analyzed the complexity of SFOMs with gradient clipping for solving \eqref{ucpb} under heavy-tailed noise (e.g., \cite{cutkosky2021high,liu2023breaking,NEURIPS2023_4c454d34,zhang2020adaptive}). Specifically, assuming that $\nabla f$ is Lipschitz continuous, \cite{zhang2020adaptive} established a first-order oracle complexity of $\mathcal{\mathcal{O}}(\epsilon^{-(3\alpha-2)/(\alpha-1)})$ for finding an $\epsilon$-SSP of \eqref{ucpb}, while \cite{cutkosky2021high,NEURIPS2023_4c454d34} established a first-order oracle complexity of $\widetilde{\mathcal{\mathcal{O}}}(\epsilon^{-(3\alpha-2)/(\alpha-1)})$ for finding a point $x$ satisfying $\|\nabla f(x)\|\le\epsilon$ with high probability. Furthermore, under the additional assumption that $\nabla^2 f$ is Lipschitz continuous, \cite{cutkosky2021high} has shown that the complexity bound can be improved to $\widetilde{\mathcal{\mathcal{O}}}(\epsilon^{-(5\alpha-3)/(2\alpha-2)})$. In addition, assuming that $G(\cdot;\xi)$ is almost surely Lipschitz continuous for all $\xi\in\Xi$ with the same Lipschitz constant, namely, 
\begin{align}\label{as-Lip}
\|G(y;\xi)-G(x;\xi)\|\le L\|y-x\|\qquad \forall x,y\in\mathbb{R}^n,\ \mathrm{a.s.}\ \forall \xi\in\Xi    
\end{align}
for some $L>0$, where $\Xi$ is the sample space of $\xi$, \cite{liu2023breaking} established a first-order oracle complexity of $\widetilde{\mathcal{O}}(\epsilon^{-(2\alpha-1)/(\alpha-1)})$ for finding a point $x$ satisfying $\|\nabla f(x)\|\le\epsilon$ with high probability. 

While the aforementioned SFOMs with gradient clipping provide provable convergence guarantees, they suffer from several practical limitations. In particular, these methods require sufficiently large clipping thresholds 
$\tau$ to ensure theoretical convergence, whereas relatively small thresholds are typically used in practice, especially in the training of deep neural networks (see, e.g., \cite{touvron2023llama}). Furthermore, 
setting an appropriate value for 
$\tau$ often depends on explicit knowledge of problem-dependent parameters, such as Lipschitz constants and noise bounds, which are generally unavailable or difficult to estimate in real-world applications. %{\color{blue} These mismatches between theory and practice for SFOMs with gradient clipping might stem from the over-generality of the theoretical works, and the requirement of parameter knowledge might due to a lack of proof techniques.
%Lastly, there has been numerical evidence showing that clipping thresholds commonly used in practice lead to an increasing probability of clipping gradients, eventually resulting in gradients being clipped at every iteration, which turns the gradient clipping into gradient noralmization in in every iteration.} 
A more detailed discussion of these drawbacks can be found in \cite{hubler2024gradient}.

In addition, it is often observed in practice that the commonly used clipped gradient in clipped SFOMs becomes a normalized gradient as the iterates progress; that is, \eqref{clip} reduces to $\overline{G}_\tau(x;\xi) = \tau G(x;\xi)/\|G(x;\xi)\|$ (e.g., see \cite{hubler2024gradient}). This empirical observation motivates replacing gradient clipping in SFOMs with gradient normalization.%In addition, numerical evidence has shown that commonly used clipping thresholds often gradually increase the probability of gradients being clipped, ultimately turning gradient clipping into gradient normalization (e.g., see \cite{hubler2024gradient}); that is, \eqref{clip} effectively becomes $\overline{G}_\tau(x;\xi) := \tau G(x;\xi)/\|G(x;\xi)\|$. This motivates studies on replacing the gradient clipping used in SFOMs with gradient normalization.
Several recent studies %have sought to avoid gradient clipping and 
have shown that, assuming $\nabla f$ is Lipschitz continuous, normalized SFOMs without gradient clipping can achieve a first-order oracle complexity of 
$\mathcal{O}(\epsilon^{-(3\alpha-2)/(\alpha-1)})$ for finding an $\epsilon$-SSP of \eqref{ucpb} under heavy-tailed noise \cite{hubler2024gradient, liu2025nonconvex, sun2024gradient}. Moreover, it has been shown in \cite{hubler2024gradient, liu2025nonconvex} that this complexity bound becomes 
$\mathcal{O}(\epsilon^{-2\alpha/(\alpha-1)})$ if the tail exponent 
$\alpha$ is unknown. In addition, under the assumption that 
$G(\cdot;\xi)$ is almost surely Lipschitz continuous with a common Lipschitz constant, as described in \eqref{as-Lip}, \cite{sun2024gradient} demonstrated that the complexity bound of normalized SFOMs can be further improved to 
$\mathcal{O}(\epsilon^{-(2\alpha-1)/(\alpha-1)})$. While these methods successfully avoid gradient clipping and achieve complexity bounds comparable to those of SFOMs that rely on it, they still suffer from several practical limitations. In particular, the methods in \cite{liu2025nonconvex,sun2024gradient} require explicit knowledge of problem-specific quantities, such as Lipschitz constant and noise bound, in order to properly set algorithmic parameters like step sizes and momentum coefficients. In contrast, \cite{hubler2024gradient} proposed a fully parameter-free SFOM and established a first-order complexity result under the assumption that $\nabla f$ is Lipschitz continuous. However, the achieved complexity is significantly worse than the best-known results when the tail exponent $\alpha$ is known. In addition, the almost sure Lipschitz continuity assumption \eqref{as-Lip} may be restrictive in practice, as it is stronger than the commonly used assumption of Lipschitz continuity in expectation, such as the mean-squared smoothness condition often imposed in the bounded variance setting (see \cite{cutkosky2019momentum, fang2018spider, li2021page}). 

Despite significant recent advances, existing SFOMs for solving problem \eqref{ucpb} under the heavy-tailed noise regime still face the practical limitations discussed above. These limitations motivate the following open questions:
\begin{itemize}
\item Can we develop practical SFOMs that do not require explicit knowledge of problem-specific quantities, while still achieving the best-known complexity?
\item Can we design practical SFOMs that achieve improved complexity under higher-order smoothness condition on $f$?
\item Can we develop practical SFOMs under weaker conditions than the commonly used mean-squared smoothness condition on $G(\cdot;\xi)$?
\end{itemize}

In this paper, we address these questions by proposing three practical SFOMs that use dynamically updated algorithmic parameters, without requiring explicit knowledge of the Lipschitz constant or noise bounds. We establish first-order oracle complexity results for these methods in finding an $\epsilon$-SSP of \eqref{ucpb} under the heavy-tailed noise regime. 
% {\color{blue}For ease of comparison, we summarize in Table \ref{table:sum-ic} the first-order oracle complexity results of our proposed SFOMs and several existing methods, together with their associated smoothness assumptions and noise conditions.} 
For ease of comparison, we summarize in Table \ref{table:sum-ic} the first-order oracle complexity results of our proposed SFOMs and several existing methods, together with their smoothness assumptions and parameter knowledge requirements.
Moreover, we show that two of our proposed methods achieve improved complexity under higher-order smoothness and a weakly average smoothness condition, respectively. Our main contributions are summarized below.

\begin{table}[htbp!]
\centering
\caption{First-order oracle complexity of our SFOMs and several existing methods under heavy-tailed noise. %The extra logarithmic factor in the complexity results from using dynamically decreasing step sizes instead of a fixed step size. 
$\dagger$ denotes that clipping has been applied to the method.}
\smallskip
%\resizebox{\textwidth}{!}{
\begin{tabular}{llcl}
\hline
method & oracle complexity & parameter-free &  smoothness \\
\hline
GClip$^\dagger$ \cite{zhang2020adaptive} & $\mathcal{O}(\epsilon^{-(3\alpha-2)/(\alpha-1)})$ & \ding{55}  & Lipschitz $\nabla f$ \\
NSGD-M \cite{hubler2024gradient} & $\mathcal{O}(\epsilon^{-(3\alpha-2)/(\alpha-1)})$ & \ding{55}  & Lipschitz $\nabla f$  \\
Algorithm \ref{alg:unf-sfom-pm} (ours) & $\widetilde{\mathcal{O}}(\epsilon^{-(3\alpha-2)/(\alpha-1)})$ & \ding{55}  & Lipschitz $\nabla f$ \\
NIGT$^\dagger$ \cite{cutkosky2021high} & $\widetilde{\mathcal{O}}(\epsilon^{-(5\alpha-3)/(2\alpha-2)})$ & \ding{55} & Lipschitz $\nabla^2 f$ \\
%A-NSGDC$^\dagger$ \cite{sun2024gradient} & $\widetilde{\mathcal{O}}(\epsilon^{-(4\alpha-1)/(2\alpha-2)})$ & \ding{55} & Lipschitz $\nabla^2 f$ \\
{Algorithm \ref{alg:unf-sfom}} (ours) & $\widetilde{\mathcal{O}}(\epsilon^{-(p(2\alpha-1)+\alpha-1)/(p(\alpha-1))})$ & \ding{55} & Lipschitz $\mathcal{D}^p f$, $p\ge2$ \\
NSGD-VR \cite{sun2024gradient} & ${\mathcal{O}}(\epsilon^{-(2\alpha-1)/(\alpha-1)})$ & \ding{55} & Lipschitz $G$ a.s. \\
Algorithm \ref{alg:unf-sfom-rm} (ours) & $\widetilde{\mathcal{O}}(\epsilon^{-(2\alpha-1)/(\alpha-1)})$ & \ding{55} & average Lipschitz $G$ \\
\hline
NSGD-M \cite{hubler2024gradient} & $\widetilde{\mathcal{O}}(\epsilon^{-2\alpha/(\alpha-1)})$ & \ding{51} & Lipschitz $\nabla f$  \\
Algorithm \ref{alg:unf-sfom-pm} (ours) & $\widetilde{\mathcal{O}}(\epsilon^{-2\alpha/(\alpha-1)})$ & \ding{51} & Lipschitz $\nabla f$  \\
Algorithm \ref{alg:unf-sfom} (ours) & $\widetilde{\mathcal{O}}(\epsilon^{-(3p\alpha+\alpha)/(2p(\alpha-1))})$ & \ding{51} & Lipschitz $\mathcal{D}^p f$, $p\ge2$  \\
Algorithm \ref{alg:unf-sfom-rm} (ours) & $\widetilde{\mathcal{O}}(\epsilon^{-3\alpha/(2(\alpha-1))})$ & \ding{51} &  average Lipschitz $G$ \\
\hline
\end{tabular}
%}
\label{table:sum-ic}
\end{table}

\begin{itemize}
\item We propose a practical \textit{normalized SFOM with Polyak momentum} for solving problem \eqref{ucpb}, and show that it achieves the best-known complexity for finding an $\epsilon$-SSP under heavy-tailed noise and the Lipschitz smoothness condition on $f$.

\item We propose a practical \textit{normalized SFOM with multi-extrapolated momentum} for solving problem \eqref{ucpb}, and establish a new complexity for finding an $\epsilon$-SSP under heavy-tailed noise and a higher-order smoothness condition on $f$. To the best of our knowledge, this is the first SFOM that leverages higher-order smoothness of $f$ to achieve acceleration. The resulting complexity significantly improves upon the best-known results under the standard Lipschitz smoothness condition.

\item We develop a practical \textit{normalized SFOM with recursive momentum} for solving problem \eqref{ucpb}, and show that it achieves a new complexity for finding an $\epsilon$-SSP under heavy-tailed noise and a weakly average smoothness condition on $G(\cdot;\xi)$. This complexity generalizes existing complexity results under the mean-squared smoothness condition.
\end{itemize}

The rest of this paper is organized as follows. In Section \ref{sec:nap}, we introduce the notation and assumptions used throughout the paper. In Section \ref{sec:nsgm}, we propose normalized SFOMs with momentum and establish complexity bounds for them. Section \ref{sec:ne} presents preliminary numerical results. In Section \ref{sec:pf-man}, we provide the proofs of the main results. 

\section{Notation and assumptions}\label{sec:nap}
Throughout this paper, we use $\R^n$ to denote the $n$-dimensional Euclidean space and $\langle\cdot,\cdot\rangle$ to represent the standard inner product. We use $\|\cdot\|$ to denote the Euclidean norm for vectors and the spectral norm for matrices. For any positive integer $p$ and a $p$th-order continuously differentiable function $\varphi$, we denote by $\Dp\varphi(x)[h_1,\ldots,h_p]$ the $p$th-order directional derivative of $\varphi$ at $x$ along $h_i\in\R^n$, $1\le i\le p$, and 
by $\Dp \varphi(x)[\cdot]$ the associated symmetric $p$-linear form. For any symmetric $p$-linear form $\mathcal{T}[\cdot]$, we define its norm as
\begin{align}\label{def:pnorm}
\|\mathcal{T}\|:=\max_{h_1,\ldots,h_p}\{\mathcal{T}[h_1,\ldots,h_p]:\|h_i\|\le1,1\le i\le p\}. 
\end{align}
For any $x\in\R^n$ and $h_i\in\R^n$ with $1\le i\le p-1$, we define $\nabla^p\varphi(x)(h_1,\ldots,h_{p-1})\in\R^n$ by
\[
\langle \nabla^p\varphi(x)(h_1,\ldots,h_{p-1}), h_p\rangle := \Dp\varphi(x)[h_1,\ldots,h_p]\qquad \forall h_p\in\R^n.
\]
For any $x,h\in\R^n$, we denote $\Dp \varphi(x)[h]^{p}:= \Dp \varphi(x)[h,\ldots,h]$ and $\nabla^p \varphi(x)(h)^{p-1}:=\nabla^p \varphi(x)(h,\ldots,h)$. For any $s\in\R$, we let $\mathrm{sgn}(s)$ be $1$ if $s\ge0$ and $-1$ otherwise. For any positive integer $p$, we define the residual of the $p$th-order Taylor expansion of $\nabla f$ as:
\begin{align}\label{def:taylor-res}
\Rp(y,x) := \nabla f(y) - \sum_{r=1}^p\frac{1}{(r-1)!}\nabla^r f(x)(y-x)^{r-1}\qquad\forall x,y\in\R^n.
\end{align}
In addition, we use $\widetilde{\mathcal{O}}(\cdot)$ to denote $\mathcal{O}(\cdot)$ with logarithmic terms omitted.

We now make the following assumption throughout this paper.
\begin{assumption}\label{asp:basic}
\begin{enumerate}[{\rm (a)}]
\item There exists a finite $f_{\mathrm{low}}$ such that $f(x)\ge f_{\mathrm{low}}$ for all $x\in\R^n$.
\item There exists $L_1>0$ such that $\|\nabla f(x) - \nabla f(y)\|\le L_1\|x-y\|$ for all $x,y\in\R^n$.
\item The stochastic gradient estimator $G:\mathbb{R}^n\times\Xi\to\R^n$ satisfies 
\begin{align*}
\mathbb{E}[G(x;\xi)] = \nabla f(x),\quad\mathbb{E}[\|G(x;\xi) - \nabla f(x)\|^\alpha]\le \sigma^\alpha\qquad\forall x\in\mathbb{R}^n    
\end{align*}
for some $\sigma>0$ and $\alpha\in(1,2]$.
\end{enumerate}    
\end{assumption}

We next make some remarks on Assumption \ref{asp:basic}.

\begin{remark} 
Assumptions \ref{asp:basic}(a) and \ref{asp:basic}(b) are standard. In particular, Assumption \ref{asp:basic}(b) implies that
\begin{align}\label{ineq:1st-desc}
f(y) \le f(x) + \nabla f(x)^T(y-x) + \frac{L_1}{2}\|y-x\|^2\qquad \forall x,y\in\R^n.    
\end{align}
Assumption \ref{asp:basic}(c) states that $G(\cdot;\xi)$ is an unbiased estimator of $\nabla f(\cdot)$, and its $\alpha$th central moment is uniformly bounded. This is weaker than the commonly used variance bounded assumption corresponding to the case $\alpha = 2$. When $\alpha \in (1,2)$, the stochastic gradient noise is said to be heavy-tailed (see, e.g., \cite{zhang2020adaptive}), a phenomenon frequently observed in modern machine learning applications.
\end{remark}

\section{Normalized stochastic first-order methods with
momentum}\label{sec:nsgm}

In this section, we propose practical normalized SFOMs with Polyak momentum, multi-extrapolated momentum, and recursive momentum for solving problem \eqref{ucpb}. We also establish their first-order oracle complexities for finding an $\epsilon$-SSP of \eqref{ucpb} under heavy-tailed noise.

\subsection{A normalized SFOM with Polyak momentum}\label{subsec:nsgd-pm}

In this subsection, we propose a practical normalized SFOM with Polyak momentum for solving problem \eqref{ucpb} and establish its first-order oracle complexity for finding an $\epsilon$-SSP of \eqref{ucpb} under heavy-tailed noise.

Specifically, our practical normalized SFOM with Polyak momentum generates two sequences, $\{m^k\}$ and $\{x^k\}$. At each iteration $k \geq 0$, the direction $m^k$ is computed as a weighted average of stochastic gradients of $f$ evaluated at the iterates $x^0, \ldots, x^k$. The next iterate $x^{k+1}$ is obtained by performing a line search update from $x^k$ along the normalized direction $-m^k/\|m^k\|$, using a suitable step size. The detailed procedure is presented in Algorithm \ref{alg:unf-sfom-pm}, with the specific momentum weights and step sizes defined in Theorems \ref{thm:s-rate-pm} and \ref{thm:un-s-rate-pm}.

\begin{algorithm}[!htbp]
\caption{A normalized SFOM with Polyak momentum}
\label{alg:unf-sfom-pm}
\begin{algorithmic}[0]
\State \textbf{Input:} starting point $x^0\in\R^n$, step sizes $\{\eta_k\}\subset(0,+\infty)$, weighting parameters $\{\theta_k\}\subset(0,1]$.
\State \textbf{Initialize:} $m^{-1}=0$ and $\theta_{-1}=1$.
\For{$k=0,1,2,\ldots$}
\State Compute the search direction:\footnotemark{}
\begin{align}\label{update-mk-pm}
m^k = (1 - \theta_{k-1}) m^{k-1} + \theta_{k-1} G(x^k;\xi^{k}).%\footnotemark{}
\end{align}
\State Update the next iterate:
\begin{align*}
x^{k+1} = x^k - \eta_k\frac{m^k}{\|m^k\|}. 
\end{align*}
\EndFor
\end{algorithmic}
\end{algorithm}
\footnotetext{$\{\xi^k\}$ is a sequence of independently drawn samples.}

It is worth noting that Algorithm \ref{alg:unf-sfom-pm} shares a similar framework with the normalized SGD with momentum proposed in \cite{cutkosky2020momentum}, but they differ in the choice of parameters. Specifically, the normalized SGD with momentum in \cite{cutkosky2020momentum} adopts static parameters that require explicit knowledge of the Lipschitz constant and the noise bound, whereas Algorithm \ref{alg:unf-sfom-pm} employs two types of dynamically updated parameters. One type leverages the knowledge of $\alpha$ when it is available and achieves an oracle complexity with optimal dependence on the accuracy parameter $\epsilon$ (see Theorem~\ref{thm:s-rate-pm}), while the other is fully parameter-free and applies when $\alpha$ is unknown (see Theorem~\ref{thm:un-s-rate-pm}).

The following theorem establishes a complexity bound for Algorithm \ref{alg:unf-sfom-pm} to compute an $\epsilon$-SSP of problem \eqref{ucpb} under the assumption that the tail exponent $\alpha$ is known. Its proof is deferred to Subsection \ref{subsec:pf-pm}.

\begin{theorem}[{{\bf complexity with known $\alpha$}}]\label{thm:s-rate-pm}
Suppose that Assumption \ref{asp:basic} holds. Let $f_{\mathrm{low}}$, $L_1$, $\sigma$, and $\alpha$ be given in Assumption \ref{asp:basic}, and define
\begin{align}\label{M1a}
M_{1,\alpha} := 2(f(x^0) - f_{\mathrm{low}} + 3\sigma^\alpha + L_1 + (\alpha-1)(2/\alpha)^{\alpha/(\alpha-1)} + 3L_1^\alpha).  
\end{align}
Let $\{x^k\}$ be generated by Algorithm \ref{alg:unf-sfom-pm} with input parameters $\{(\eta_k,\theta_k)\}$ given by
\begin{align}\label{pm-eta-theta}
\eta_k=\frac{1}{(k+1)^{(2\alpha-1)/(3\alpha-2)}},\quad \theta_k=\frac{1}{(k+1)^{\alpha/(3\alpha-2)}}\qquad\forall k\ge0.
\end{align}
Then, for any $\epsilon \in (0,1)$, it holds that $\mathbb{E}[\|\nabla f(x^{\iota_K})\|]\le \epsilon$ for all $K$ satisfying
\begin{align*}
K\ge\max\Big\{\Big(\frac{2(3\alpha-2)M_{1,\alpha}}{(\alpha-1)\epsilon}\ln\Big(\frac{2(3\alpha-2)M_{1,\alpha}}{(\alpha-1)\epsilon}\Big)\Big)^{(3\alpha-2)/(\alpha-1)},3\Big\},     
\end{align*}
where $\iota_K$ is uniformly drawn from $\{0,\ldots,K-1\}$.
\end{theorem}

The next theorem establishes a complexity bound for Algorithm \ref{alg:unf-sfom-pm} to compute an $\epsilon$-SSP of problem \eqref{ucpb} without requiring prior knowledge of the tail exponent $\alpha$. Its proof is deferred to Subsection \ref{subsec:pf-pm}.

\begin{theorem}[{{\bf complexity with unknown $\alpha$}}]\label{thm:un-s-rate-pm}
Suppose that Assumption \ref{asp:basic} holds. Let $f_{\mathrm{low}}$, $L_1$, $\sigma$, and $\alpha$ be given in Assumption \ref{asp:basic}, and define
\begin{align}\label{M1a-o}
\widetilde{M}_{1,\alpha}:= 2(f(x^0) - f_{\mathrm{low}} + \sigma^\alpha + L_1/2 + 3L_1^\alpha),\quad\widehat{M}_{1,\alpha} := 2((\alpha-1)(2/\alpha)^{\alpha/(\alpha-1)} + 2\sigma^\alpha).
\end{align}
Let $\{x^k\}$ be generated by Algorithm \ref{alg:unf-sfom-pm} with input parameters $\{(\eta_k,\theta_k)\}$ given by
\begin{align}\label{pm-eta-theta-o}
\eta_k=\frac{1}{(k+1)^{3/4}},\quad \theta_k=\frac{1}{(k+1)^{1/2}}\qquad\forall k\ge0. 
\end{align}
Then, for any $\epsilon \in (0,1)$, it holds that $\mathbb{E}[\|\nabla f(x^{\iota_K})\|]\le \epsilon$ for all $K$ satisfying
\[
K\ge \max\Big\{\Big(\frac{16\widetilde{M}_{1,\alpha}}{\epsilon}\ln\Big(\frac{16\widetilde{M}_{1,\alpha}}{\epsilon}\Big)\Big)^{4},\Big(\frac{8\alpha\widehat{M}_{1,\alpha}}{(\alpha-1)\epsilon}\ln\Big(\frac{8\alpha\widehat{M}_{1,\alpha}}{(\alpha-1)\epsilon}\Big)\Big)^{2\alpha/(\alpha-1)},3\Big\},
\]
where $\iota_K$ is uniformly drawn from $\{0,\ldots,K-1\}$.
\end{theorem}

\begin{remark}
(i) From Theorems~\ref{thm:s-rate-pm} and \ref{thm:un-s-rate-pm}, we observe that under Assumption~\ref{asp:basic}, Algorithm~\ref{alg:unf-sfom-pm} achieves a first-order oracle complexity of $\widetilde{\mathcal{O}}(\epsilon^{-(3\alpha-2)/(\alpha-1)})$ for finding an $\epsilon$-SSP of problem~\eqref{ucpb} when the tail exponent $\alpha$ is known, and $\widetilde{\mathcal{O}}(\epsilon^{-2\alpha/(\alpha-1)})$ when $\alpha$ is unknown. These results match, up to logarithmic factors, the best-known complexities for SFOMs with gradient clipping~\cite{cutkosky2021high,NEURIPS2023_4c454d34,zhang2020adaptive} and normalized SFOMs without gradient clipping~\cite{hubler2024gradient,liu2025nonconvex,sun2024gradient}.
% {\color{blue}In addition, NSGD-M in \cite[Appendix D.2]{hubler2024gradient} achieves a complexity result with better dependence on the Lipschitz constant and the noise bound than Algorithm~\ref{alg:unf-sfom-pm}. However, their method employs static parameters that require explicit knowledge of both the Lipschitz constant and the noise bound, whereas Algorithm~\ref{alg:unf-sfom-pm} uses two types of dynamically updated parameters. One type leverages the knowledge of $\alpha$ when it is available and achieves an oracle complexity with optimal dependence on $\epsilon$ (see Theorem~\ref{thm:s-rate-pm}), while the other is fully parameter-free and applies when $\alpha$ is unknown (see Theorem~\ref{thm:un-s-rate-pm}).}
It is worth mentioning that these complexity results exhibit a worse dependence on the Lipschitz constant and the noise bound than than that achieved by NSGD-M \cite[Appendix D.2]{hubler2024gradient}. However, the latter method employs static parameters that require explicit knowledge of both the Lipschitz constant and the noise bound, whereas Algorithm~\ref{alg:unf-sfom-pm} uses two types of dynamically updated parameters, as specified in Theorems~\ref{thm:s-rate-pm} and~\ref{thm:un-s-rate-pm}.

(ii) As shown in Theorem \ref{thm:un-s-rate-pm}, when $\alpha \in (1,2)$ is unknown, the step-size and momentum choices---which coincide with those used in the case where $\alpha$ is known to be $2$---still guarantee convergence of the algorithm. However, these choices are considerably more conservative than those specifically chosen for the case where $\alpha \in (1,2)$ is known. In particular, they decay faster than the step-size and momentum sequences designed for the known-$\alpha$ setting (see Theorem~1).

(iii) When the tail exponent $\alpha$ is unknown, Algorithm~\ref{alg:unf-sfom-pm} with $\{(\eta_k,\theta_k)\}$ specified by \eqref{pm-eta-theta-o} resembles the parameter-free SFOM with momentum proposed in \cite[Appendix D]{hubler2024gradient}. Nevertheless, our complexity analysis is fundamentally  different from that of \cite{hubler2024gradient} and other existing works, as it is based on descent properties of a novel potential sequence defined in \eqref{def:pot-pm}. One key benefit of our analysis is that it naturally provides a per-iteration descent property for Algorithm \ref{alg:unf-sfom-pm} as a byproduct (see Lemma \ref{lem:rate-pm}). In contrast, prior analyses under heavy-tailed noise often lack such a result, since a descent relation for the gradient error sequence $\{m^k-\nabla f(x^k)\}$ between two consecutive iterations is not established; instead, they bound the accumulated gradient errors over all previous iterations (see, e.g., \cite[Theorem 1]{cutkosky2020momentum}). Building on this novel framework, we further analyze SFOMs with accelerated momentum schemes in Subsections \ref{subsec:nsgd-mem} and \ref{subsec:nsgd-rm}, obtaining the first complexity results for SFOMs under higher-order smoothness and average smoothness, respectively.
\end{remark} 

\subsection{A normalized SFOM with multi-extrapolated momentum}\label{subsec:nsgd-mem}

In this subsection, we propose a practical normalized SFOM with multi-extrapolated momentum for solving problem \eqref{ucpb} and establish its first-order oracle complexity for finding an $\epsilon$-SSP of \eqref{ucpb} under heavy-tailed noise.

% {\color{blue}
% Our proposed method is a significant generalization of the normalized SGD with implicit gradient transport (NIGT) proposed in~\cite{cutkosky2020momentum}. When $q = 1$, this method and NIGT share a similar algorithmic framework, but they differ in their choices of algorithmic parameters. Specifically, NIGT employs static parameters that require explicit knowledge of the Lipschitz constant and the noise bound, whereas our method uses dynamically updated parameters that do not depend on such problem-specific quantities. Moreover, when $q > 1$, our method is entirely new to the literature and, to the best of our knowledge, is the first method capable of exploiting arbitrary high-order smoothness to achieve acceleration.
% }

Specifically, our practical normalized SFOM with multi-extrapolated momentum generates three sequences: $\{z^{k,t}\}$, $\{m^k\}$, and $\{x^k\}$. At each iteration $k \geq 0$, the points $z^{k,1}, \ldots, z^{k,q}$ are computed by extrapolating $x^{k-1}$ and $x^k$ using a set of extrapolation weights. The direction $m^k$ is then formed as a weighted average of stochastic gradients of $f$ evaluated at the extrapolated points $\{z^{i,t}\}_{0 \leq i \leq k, 1 \leq t \leq q}$. Finally, $x^{k+1}$ is computed via a line search update at $x^k$ using a suitable step size and the normalized direction $-m^k/\|m^k\|$. The detailed procedure is described in Algorithm~\ref{alg:unf-sfom}, with the extrapolation weights, momentum weights, and step sizes specified in Theorems~\ref{thm:known-rate-mem} and~\ref{thm:unknown-rate-mem}.

\begin{algorithm}[!htbp]
\caption{A normalized SFOM with multi-extrapolated momentum}
\label{alg:unf-sfom}
\begin{algorithmic}[0]
\State \textbf{Input:} starting point $x^0\in\R^n$, step sizes $\{\eta_k\}\subset(0,+\infty)$, extrapolation count $q\ge1$, extrapolation parameters $\{\gamma_{k,t}\}\subset(0,1)$, weighting parameters $\{\theta_{k,t}\}$ with $\sum_{t=1}^{q}\theta_{k,t}\in(0,1)$ for all $k\ge0$.
\State \textbf{Initialize:} $x^{-1}=x^0$, $m^{-1}=0$, and $(\gamma_{-1,t},\theta_{-1,t})=(1,1/q)$ for all $1\le t\le q$.
\For{$k=0,1,2,\ldots$}
\State Perform $q$ separate extrapolations: 
\begin{align}\label{update-zk}
z^{k,t} = x^k + \frac{1-\gamma_{k-1,t}}{\gamma_{k-1,t}}(x^k - x^{k-1})\qquad \forall 1\le t\le q.
\end{align}
\State Compute the search direction:\footnotemark{}
\begin{align}\label{update-mk}
m^k = \Big(1 - \sum_{t=1}^{q}\theta_{k-1,t}\Big) m^{k-1} + \sum_{t=1}^{q}\theta_{k-1,t}G(z^{k,t};\xi^k).
\end{align}
\State Update the next iterate:
\begin{align*}
x^{k+1} = x^k - \eta_k\frac{m^k}{\|m^k\|}. 
\end{align*}
\EndFor
\end{algorithmic}
\end{algorithm}

It is worth mentioning that Algorithm~\ref{alg:unf-sfom} is a significant generalization of the normalized SGD with implicit gradient transport (NIGT) proposed in~\cite{cutkosky2020momentum}. When $q = 1$, Algorithm~\ref{alg:unf-sfom} and NIGT share a similar algorithmic framework, but they differ in their choices of algorithmic parameters. Specifically, NIGT employs static parameters that require explicit knowledge of the Lipschitz constant and the noise bound, whereas Algorithm~\ref{alg:unf-sfom} uses dynamically updated parameters that do not depend on such problem-specific quantities. Moreover, when $q > 1$, Algorithm~\ref{alg:unf-sfom} is entirely new to the literature and, to the best of our knowledge, is the first method capable of exploiting arbitrary high-order smoothness to achieve acceleration.

Before analyzing the complexity of Algorithm~\ref{alg:unf-sfom} for computing an approximate solution to problem~\eqref{ucpb}, we introduce an additional assumption regarding the high-order smoothness of the objective function~$f$.

\begin{assumption}\label{asp:high-smooth}
The function $f$ is $p$th-order continuously differentiable in $\R^n$ for some $p\ge2$, and moreover, there exists some $L_p>0$ such that $\|\Dp f(x) - \Dp f(y)\| \le L_p\|x-y\|$ for all $x,y\in\R^n$.    
\end{assumption}

The following theorem establishes a complexity bound for Algorithm \ref{alg:unf-sfom} to compute an $\epsilon$-SSP of problem \eqref{ucpb} under the assumption that the tail exponent $\alpha$ is known. Its proof is deferred to Subsection \ref{subsec:pf-mem}.
\footnotetext{$\{\xi^k\}$ is a sequence of independently drawn samples. Alternatively, one may draw $q$ independent samples $\xi^{k,t}$, $1\le t\le q$, at every $k$th iteration for computing $G(z^{k,t};\xi^{k,t})$, $1\le t\le q$.
}
\begin{theorem}[{{\bf complexity with known $\alpha$}}]\label{thm:known-rate-mem}
Suppose that Assumptions \ref{asp:basic} and \ref{asp:high-smooth} hold. Let $f_{\mathrm{low}}$, $L_1$, $\sigma$, $\alpha$, $p$ and $L_p$ be given in Assumptions \ref{asp:basic} and \ref{asp:high-smooth}, and define
\begin{align}
M_{p,\alpha} &:= 4\Big(f(x^0) - f_{\mathrm{low}} + 4^{1/3}\sigma^\alpha + L_1/2 + 30^{1/(\alpha-1)}(\alpha-1)(2/\alpha)^{\alpha/(\alpha-1)} + 306 p^{2\alpha p} L_p^\alpha/(p!)^\alpha \nonumber\\ 
&\quad \ +64(p-1)^\alpha\sigma^\alpha\Big).   \label{def:Mpa-known}
\end{align}
Let $\{x^k\}$ be  generated by Algorithm \ref{alg:unf-sfom} with input parameters $q=p-1$, and $\{\eta_k\}$ and $\{(\gamma_{k,t},\theta_{k,t})\}$ given by
\begin{align}
&\eta_k=\frac{1}{(k+4)^{(p\alpha+\alpha-1)/(p(2\alpha-1)+\alpha-1)}}\qquad\forall k\ge0,\label{def:eta-known-rate-mem}\\
&\gamma_{k,t}=\frac{\gamma_k}{t^2},\quad \theta_{k,t}=\frac{\prod_{1\le s\le p-1,s\neq t}(1-s^2/\gamma_k)}{(t^2/\gamma_k)\prod_{1\le s\le p-1,s\neq t}((t^2-s^2)/\gamma_k)}\qquad\forall 1\le t\le p-1,k\ge0,\label{def:theta-known-rate-mem}
\end{align}
where
\begin{align}\label{def:gmak-mem-know}
\gamma_k=\frac{1}{(k+4)^{p\alpha/(p(2\alpha-1)+\alpha-1)}}\qquad\forall k\ge0.
\end{align}
Then, for any $\epsilon \in (0,1)$, it holds that $\mathbb{E}[\|\nabla f(x^{\iota_K})\|]\le \epsilon$ for all $K$ satisfying
\[
K\ge  \max\Big\{\Big(\frac{2(p(2\alpha-1) + \alpha -1)M_{p,\alpha}}{p(\alpha-1)\epsilon}\ln\Big(\frac{2(p(2\alpha-1) + \alpha -1)M_{p,\alpha}}{p(\alpha-1)\epsilon}\Big)\Big)^{(p(2\alpha-1)+\alpha-1)/(p(\alpha-1))},5\Big\},
\]
where $\iota_K$ is uniformly drawn from $\{0,\ldots,K-1\}$.
\end{theorem}

The next theorem establishes a complexity bound for Algorithm \ref{alg:unf-sfom} to compute an $\epsilon$-SSP of problem \eqref{ucpb} without requiring prior knowledge of the tail exponent $\alpha$. Its proof is deferred to Subsection~\ref{subsec:pf-mem}.

\begin{theorem}[{{\bf complexity with unknown $\alpha$}}]\label{thm:unknown-rate-mem}
Suppose that Assumptions \ref{asp:basic} and \ref{asp:high-smooth} hold. Let $f_{\mathrm{low}}$, $L_1$, $\sigma$, $\alpha$, $p$ and $L_p$ be given in Assumptions \ref{asp:basic} and \ref{asp:high-smooth}, and define
\begin{align}
&\widetilde{M}_{p,\alpha}:= 4(f(x^0) - f_{\mathrm{low}} + 4^{1/3}\sigma^\alpha + L_1/2 + 306 p^{2p\alpha} L_p^\alpha/(p!)^\alpha),\label{def:t-Mpa-unk}\\
&\widehat{M}_{p,\alpha}:= 8(30^{1/(\alpha-1)}(\alpha-1)(2/\alpha)^{\alpha/(\alpha-1)}+ 64(p-1)^\alpha\sigma^\alpha).\label{def:h-Mpa-unk}
\end{align}
 Let $\{x^k\}$ be generated by Algorithm \ref{alg:unf-sfom} with input parameters $q=p-1$, $\{\eta_k\}$, and $\{(\gamma_{k,t},\theta_{k,t})\}$ given by 
\begin{align}
&\eta_k=\frac{1}{(k+4)^{(2p+1)/(3p+1)}}\qquad\forall k\ge0,\label{def:eta-unknown-rate-mem}\\
&\gamma_{k,t}=\frac{\gamma_k}{t^2},\quad \theta_{k,t}=\frac{\prod_{1\le s\le p-1,s\neq t}(1-s^2/\gamma_k)}{(t^2/\gamma_k)\prod_{1\le s\le p-1,s\neq t}((t^2-s^2)/\gamma_k)}\qquad\forall 1\le t\le p-1,k\ge0,\label{def:theta-unknown-rate-mem}
\end{align}
where
\begin{align}\label{def:gmak-mem-unknow}
\gamma_k=\frac{1}{(k+4)^{2p/(3p+1)}}\qquad\forall k\ge0.    
\end{align}
Then, for any $\epsilon \in (0,1)$, it holds that $\mathbb{E}[\|\nabla f(x^{\iota_K})\|]\le \epsilon$ for all $K$ satisfying
\begin{align*}
&K\ge \max\Bigg\{\Big(\frac{4(3p+1)\widetilde{M}_{p,\alpha}}{p\epsilon}\ln\Big(\frac{4(3p+1)\widetilde{M}_{p,\alpha}}{p\epsilon}\Big)\Big)^{(3p+1)/p},\\
&\qquad\qquad\quad \Big(\frac{2(3p\alpha+\alpha)\widehat{M}_{p,\alpha}}{p(\alpha-1)\epsilon}\ln\Big(\frac{2(3p\alpha+\alpha)\widehat{M}_{p,\alpha}}{p(\alpha-1)\epsilon}\Big)\Big)^{(3p\alpha+\alpha)/(2p(\alpha-1))},5\Bigg\},
\end{align*}
where $\iota_K$ is uniformly drawn from $\{0,\ldots,K-1\}$. 
\end{theorem}

\begin{remark}
(i) To achieve acceleration by leveraging the higher-order smoothness of $f$, the extrapolation parameters ${\gamma_{k,t}}$ and the momentum parameters ${\theta_{k,t}}$ must satisfy the following conditions:
\begin{align}
& \begin{bmatrix}
1/\gamma_{k,1} & 1/\gamma_{k,2} & \cdots & 1/\gamma_{k,q}\\ 
1/\gamma_{k,1}^2 & 1/\gamma_{k,2}^2 & \cdots & 1/\gamma_{k,q}^2\\ 
\vdots&\vdots&\ddots & \vdots\\
1/\gamma_{k,1}^q & 1/\gamma_{k,2}^q & \cdots & 1/\gamma_{k,q}^q
\end{bmatrix}\begin{bmatrix}
\theta_{k,1}\\
\theta_{k,2}\\
\vdots\\
\theta_{k,q}
\end{bmatrix}  = \begin{bmatrix}
1\\
1\\
\vdots\\
1
\end{bmatrix}\qquad\forall k\ge0,\label{pth-lss} \\
&\sum_{t=1}^q \theta_{k,t} \in(0,1)\qquad\forall k\ge0, \label{sum-tht-bd} 
\end{align}
where $q=p-1$. Note that the coefficient matrix in \eqref{pth-lss} is a Vandermonde matrix (see, e.g., \cite{horn2012matrix,humpherys2020foundations}). As will be proven in Lemma \ref{lem:ppt-thetakt} in Subsection \ref{subsec:pf-mem}, when choosing $\gamma_{k,t} = \gamma_k/t^2$, the resulting ${\theta_{k,t}}$ that satisfy \eqref{pth-lss} takes the form given in \eqref{def:theta-known-rate-mem} or \eqref{def:theta-unknown-rate-mem}.

(ii) From Theorems~\ref{thm:known-rate-mem} and~\ref{thm:unknown-rate-mem}, we observe that under Assumptions~\ref{asp:basic} and~\ref{asp:high-smooth}, Algorithm~\ref{alg:unf-sfom} achieves a first-order oracle complexity of $\widetilde{\mathcal{O}}(\epsilon^{-(p(2\alpha - 1) + \alpha - 1)/(p(\alpha - 1))})$ for finding an $\epsilon$-SSP of~\eqref{ucpb} when the tail exponent $\alpha$ is known, and $\widetilde{\mathcal{O}}(\epsilon^{-(3p\alpha + \alpha)/(2p(\alpha - 1))})$ when $\alpha$ is unknown. For $p = 2$, the complexity result for known $\alpha$ matches, up to a logarithmic factor, the bound established in~\cite{cutkosky2021high}, while the complexity result for unknown $\alpha$ is new.
%For $p = 2$, these complexity results match, up to a logarithmic factor, the best-known bound established in~\cite{cutkosky2021high}. 
Moreover, for $p \ge 3$, our results are entirely new and provide significantly improved complexities over the case $p = 2$.
\end{remark}

\subsection{A normalized SFOM with recursive momentum}\label{subsec:nsgd-rm}

In this subsection, we propose a practical normalized SFOM with recursive momentum for solving problem \eqref{ucpb} and establish its first-order oracle complexity for finding an $\epsilon$-SSP of \eqref{ucpb} under heavy-tailed noise.

% {\color{blue}Our proposed method adopts the recursive momentum technique developed in STORM \cite{cutkosky2019momentum}, but they differ in both the search direction and the choice of algorithmic parameters. Specifically, this method employs a normalized search direction, whereas STORM uses the unnormalized direction. Moreover, STORM relies on static parameters that requires explicit knowledge of the Lipschitz constant and the noise bound, while our method uses dynamically updated parameters that do not rely on such problem-specific quantities.}

%Our proposed method {\color{blue}adopts the recursive momentum technique developed in \cite{cutkosky2019momentum},}
%follows the same general framework as that in \cite{cutkosky2019momentum}, 
%but differs in both the search direction and the choice of algorithmic parameters---particularly the momentum weights and step sizes. Specifically, our method employs a normalized search direction, whereas \cite{cutkosky2019momentum} uses the unnormalized (full) direction. Moreover, unlike \cite{cutkosky2019momentum}, which {\color{blue}fixes algorithmic parameters and} requires explicit knowledge of the Lipschitz constant and the noise bound, our method uses dynamically updated parameters that do not rely on such problem-specific quantities. This feature makes our method more practical, especially in scenarios where these constants are unknown or hard to estimate.

Specifically, our practical normalized SFOM with recursive momentum generates two sequences, $\{m^k\}$ and $\{x^k\}$. At each iteration $k \geq 0$, the direction $m^k$ is computed as a weighted average of stochastic gradients of $f$ evaluated at the iterates $x^0, \ldots, x^k$. The next iterate $x^{k+1}$ is obtained by performing a line search update from $x^k$ along the normalized direction $-m^k/\|m^k\|$, using a suitable step size. The detailed procedure is presented in Algorithm \ref{alg:unf-sfom-rm}, with the specific momentum weights and step sizes defined in Theorems \ref{thm:known-rate-rm} and \ref{thm:unknown-rate-rm}.

\begin{algorithm}[!htbp]
\caption{A normalized SFOM with recursive momentum}
\label{alg:unf-sfom-rm}
\begin{algorithmic}[0]
\State \textbf{Input:} starting point $x^0\in\R^n$, step sizes $\{\eta_k\}\subset(0,+\infty)$, weighting parameters $\{\theta_k\}\subset(0,1]$.
\State \textbf{Initialize:} $x^{-1}=x^0$, $m^{-1}=0$, and $\theta_{-1}=1$.
\For{$k=0,1,2,\ldots$}
\State Compute the search direction:\footnotemark{}
\begin{align}\label{update-mk-rm}
m^k = (1 - \theta_{k-1}) m^{k-1} + G(x^k;\xi^{k}) - (1-\theta_{k-1})G(x^{k-1};\xi^{k}).
\end{align}
\State Update the next iterate:
\begin{align*}%\label{update-xk-rm}
x^{k+1} = x^k - \eta_k\frac{m^k}{\|m^k\|}. 
\end{align*}
\EndFor
\end{algorithmic}
\end{algorithm}

\footnotetext{$\{\xi^k\}$ is a sequence of independently drawn samples.}

It shall be mentioned that Algorithm~\ref{alg:unf-sfom-rm} adopts the recursive momentum technique developed in STORM \cite{cutkosky2019momentum}, but they differ in both the search direction and the choice of algorithmic parameters. Specifically, this method employs a normalized search direction, whereas STORM uses the unnormalized direction. Moreover, STORM relies on static parameters that requires explicit knowledge of the Lipschitz constant and the noise bound, while Algorithm~\ref{alg:unf-sfom-rm} uses dynamically updated parameters that do not rely on such problem-specific quantities.

Before analyzing the complexity of Algorithm~\ref{alg:unf-sfom-rm} for computing an approximate solution to problem~\eqref{ucpb}, we introduce an additional assumption regarding the {\it weakly average smoothness} of the stochastic gradient estimator $G(\cdot;\xi)$.

\begin{assumption}\label{asp:gen-ave-smth}
There exists some $L>0$ such that $\mathbb{E}[\|G(x;\xi) - G(y;\xi)\|^\alpha]\le L^\alpha\|x-y\|^\alpha$ holds for all $x,y\in\mathbb{R}^n$, where $\alpha\in(1,2]$ is given in Assumption \ref{asp:basic}(c).
\end{assumption}

\begin{remark}
(i) When $\alpha = 2$, Assumption~\ref{asp:gen-ave-smth} reduces to the standard mean-squared smoothness condition commonly used in the literature (see, e.g., \cite{arjevani2023lower,fang2018spider,li2021page}). For $\alpha \in (1, 2)$, it is strictly weaker than the mean-squared smoothness assumption, thereby holding for a broader class of stochastic gradient estimators $G(\cdot; \xi)$. In addition, if $G(\cdot; \xi)$ is almost surely Lipschitz continuous for all $\xi \in \Xi$ with a uniform Lipschitz constant $L$, one can verify that it satisfies Assumption~\ref{asp:gen-ave-smth}. Consequently, Assumption~\ref{asp:gen-ave-smth} is strictly weaker than the almost sure Lipschitz condition stated in~\eqref{as-Lip}, which is adopted in~\cite{cutkosky2019momentum,liu2023breaking,sun2024gradient}.

(ii) It is reasonable to assume that the exponent $\alpha$ in Assumption~\ref{asp:gen-ave-smth} is the same as that in Assumption~\ref{asp:basic}(c). Indeed, if Assumptions~\ref{asp:basic}(c) and~\ref{asp:gen-ave-smth} hold with different exponents $\alpha_1, \alpha_2 \in (1,2]$, then both assumptions also hold with $\alpha = \min\{\alpha_1, \alpha_2\}$.
\end{remark}

The following theorem establishes a complexity bound for Algorithm~\ref{alg:unf-sfom-rm} to compute an $\epsilon$-SSP of problem \eqref{ucpb} under the assumption that the tail exponent $\alpha$ is known. Its proof is deferred to Subsection \ref{subsec:pf-rm}.

\begin{theorem}[{{\bf complexity with known $\alpha$}}]\label{thm:known-rate-rm}
Suppose that Assumptions \ref{asp:basic} and \ref{asp:gen-ave-smth} hold. Let $f_{\mathrm{low}}$, $L_1$, $\sigma$, $\alpha$, and $L$ be given in Assumption \ref{asp:basic} and \ref{asp:gen-ave-smth}, and define
\begin{align}\label{Ma-rm}
M_\alpha := 2(f(x^0) - f_{\mathrm{low}}+\sigma^\alpha+ L_1/2 + 2^{1/(\alpha-1)}(\alpha-1)(2/\alpha)^{\alpha/(\alpha-1)} + 12(L_1^\alpha+L^\alpha) + 12\sigma^\alpha).
\end{align}
 Let $\{x^k\}$ be generated by Algorithm \ref{alg:unf-sfom-rm} with input parameters $\{(\eta_k,\theta_k)\}$ given by
\begin{align}\label{rm-eta-theta}
\eta_k= \theta_k=\frac{1}{(k+1)^{\alpha/(2\alpha-1)}}\qquad\forall k\ge0.    
\end{align}
Then, for any $\epsilon \in (0,1)$, it holds that $\mathbb{E}[\|\nabla f(x^{\iota_K})\|]\le \epsilon$ for all $K$ satisfying
\begin{align*}
\forall K\ge\max\Big\{\Big(\frac{2(2\alpha-1)M_{\alpha}}{(\alpha-1)\epsilon}\ln\Big(\frac{2(2\alpha-1)M_{\alpha}}{(\alpha-1)\epsilon}\Big)\Big)^{(2\alpha-1)/(\alpha-1)},3\Big\},     
\end{align*}   
where $\iota_K$ is uniformly drawn from $\{0,\ldots,K-1\}$.
\end{theorem}

The following theorem establishes a complexity bound for Algorithm \ref{alg:unf-sfom-rm} to compute an $\epsilon$-SSP of problem \eqref{ucpb} without requiring prior knowledge of the tail exponent $\alpha$. Its proof is deferred to Subsection \ref{subsec:pf-rm}.

\begin{theorem}[{{\bf complexity with unknown $\alpha$}}]\label{thm:unknown-rate-rm}
Suppose that Assumptions \ref{asp:basic} and \ref{asp:gen-ave-smth} hold. Let $f_{\mathrm{low}}$, $L_1$, $\sigma$, $\alpha$, and $L$ be given in Assumption \ref{asp:basic} and \ref{asp:gen-ave-smth}, and define
\begin{align}
&\widetilde{M}_\alpha := 2(f(x^0) - f_{\mathrm{low}} + \sigma^\alpha + L_1/2),\label{Ma-tilde-rm}\\
&\widehat{M}_\alpha := 2(2^{1/(\alpha-1)}(\alpha-1)(2/\alpha)^{\alpha/(\alpha-1)}+ 12(L_1^\alpha+L^\alpha) + 12\sigma^\alpha).\label{Ma-hat-rm}
\end{align}
 Let $\{x^k\}$ be generated by Algorithm \ref{alg:unf-sfom-rm} with input parameters $\{(\eta_k,\theta_k)\}$ given by
\begin{align}\label{rm-eta-theta-un}
\eta_k= \theta_k=\frac{1}{(k+1)^{2/3}}\qquad\forall k\ge0.    
\end{align}
Then, for any $\epsilon \in (0,1)$, it holds that $\mathbb{E}[\|\nabla f(x^{\iota_K})\|]\le \epsilon$ for all $K$ satisfying
\[
K\ge\max\Big\{\Big(\frac{12\widetilde{M}_\alpha}{\epsilon}\ln\Big(\frac{12\widetilde{M}_{\alpha}}{\epsilon}\Big)\Big)^{3},\Big(\frac{6\alpha\widehat{M}_\alpha}{(\alpha-1)\epsilon}\ln\Big(\frac{6\alpha \widehat{M}_\alpha}{(\alpha-1)\epsilon}\Big)\Big)^{3\alpha/(2(\alpha-1))},3\Big\},
\]
where $\iota_K$ is uniformly drawn from $\{0,\ldots,K-1\}$.    
\end{theorem}

\begin{remark}
From Theorems \ref{thm:known-rate-rm} and \ref{thm:unknown-rate-rm}, we observe that under Assumptions \ref{asp:basic} and \ref{asp:gen-ave-smth}, Algorithm \ref{alg:unf-sfom-rm} achieves a first-order oracle complexity of $\widetilde{\mathcal{O}}(\epsilon^{-(2\alpha-1)/(\alpha-1)})$ for finding an $\epsilon$-SSP of problem \eqref{ucpb} when the tail exponent $\alpha$ is known, and $\widetilde{\mathcal{O}}(\epsilon^{-3\alpha/(2(\alpha-1))})$ when $\alpha$ is unknown. The complexity bound for the known-$\alpha$ case matches, up to logarithmic factors, the best-known results in \cite{liu2023breaking,sun2024gradient}, while the bound for the unknown-$\alpha$ case is, to the best of our knowledge, new to the literature.
Moreover, these are the first complexity results established under the weakly average smoothness condition stated in Assumption \ref{asp:gen-ave-smth}. This condition generalizes and relaxes both the almost sure Lipschitz condition (see \eqref{as-Lip}) used in prior works on SFOMs under heavy-tailed noise \cite{liu2023breaking,sun2024gradient}, and the commonly adopted mean-squared smoothness assumption \cite{cutkosky2019momentum,fang2018spider,li2021page}.
\end{remark}

\section{Numerical experiments}\label{sec:ne}

In this section, we present preliminary numerical experiments to evaluate the performance of Algorithm~\ref{alg:unf-sfom-pm}, Algorithm~\ref{alg:unf-sfom} with $q = 1$, and Algorithm~\ref{alg:unf-sfom-rm}, abbreviated as NSFOM-PM, NSFOM-EM, and NSFOM-RM, respectively. %We compare these methods against their counterparts without normalization---SFOM-PM, SFOM-EM, and SFOM-RM---in the presence of heavy-tailed noise. 
We compare these methods with clipped SGD \cite{sadiev2023high, zhang2020adaptive} and the adaptive coordinate-wise clipping algorithm proposed in \cite{zhang2020adaptive}, abbreviated as GClip and ACClip, respectively. The experiments are conducted on three problem classes: a data fitting problem (Subsection~\ref{subsec:df}), a robust regression problem (Subsection~\ref{subsec:rrp}), and a multimodal contrastive learning problem (Subsection~\ref{subsec:mcl}). The first two are run on a standard PC equipped with a 3.20 GHz AMD R7 5800H processor and 16 GB of memory, while the last is executed on a server with an NVIDIA A100 GPU (80 GB). The code for reproducing our numerical results is publicly available at:
\texttt{https://github.com/ChuanH6/SFOM-HT}.

\subsection{Data fitting problem}\label{subsec:df}
In this subsection, we consider the data fitting problem:
\begin{align}\label{df}
\min_{x\in\R^n}  \Big\{f(x) = \sum_{i=1}^m \big(s(a_i^Tx)-b_i\big)^2\Big\},
\end{align}
where $s(t)=e^t/(1+e^t)$ is the sigmoid function, and $\{(a_i,b_i)\}_{1\le i\le n}\subset\R^n\times\R$ denotes the given dataset. We simulate the noisy gradient evaluations by setting the stochastic gradient estimator as $G(x;\xi)=\nabla f(x) + \xi e$, where $e\in\R^n$ is the all-ones vector and $\xi\in\R$ is drawn from a heavy-tailed distribution with density function $p(t)=3/(4(1+|t|)^{5/2})$. One can verify that such $G(\cdot;\xi)$ satisfies Assumption \ref{asp:basic}(c) for every $\alpha\in(1,1.5)$, and that the $\alpha$th central moment of $G(\cdot;\xi)$ is unbounded for all $\alpha\ge1.5$.

For each pair $(n,m)$, we randomly generate $a_i$, $1\le i\le m$, with all entries independently drawn from the standard normal distribution. We also generate a ground truth solution $x^*$ in the same manner and set $b_i=s(a_i^Tx^*)+ 10^{-4}e_i$ for each $1\le i\le m$, where $e_i$'s are independently drawn from the standard normal distribution. 

We apply NSFOM-PM, NSFOM-EM,NSFOM-RM, GClip, and ACClip 
%their unnormalized variants, 
to solve problem \eqref{df}. All methods are initialized at the zero vector. We compare them based on two metrics: the relative objective value gap $(f(x^k)-f^*)/(f(x^0)-f^*)$ and the relative gradient norm $\|\nabla f(x^k)\|/\|\nabla f(x^0)\|$, computed over the first $500$ stochastic gradient evaluations. Here, $f^*$ denotes the minimum objective value found during the first $600$ stochastic gradient evaluations across all methods. %The algorithmic parameters are selected to suit each method well in terms of computational performance.
The algorithmic parameters are tuned according to the following strategy:
% \begin{itemize}
%     \item For NSFOM-PM, NSFOM-EM, and NSFOM-RM, we set the step sizes and momentum parameters according to Theorems \ref{thm:s-rate-pm}, \ref{thm:known-rate-mem}, and \ref{thm:known-rate-rm}, with $\alpha = 1.5$. For NSFOM-EM, we set the extrapolation parameters to be the same as the momentum parameters according to \eqref{pth-lss} with $q=1$.
%     %\item For NSFOM-EM, we set the step sizes $\{\eta_k\}$ and the extrapolation parameters $\{\gamma_{k,t}\}$ as $\{1/(k+1)^{\beta_1}\}$ and $\{\{1/(t^2(k+1)^{\beta_2})\}\}$, with $\beta_1$ and $\beta_2$ set based on Theorem \ref{thm:known-rate-mem} and $\alpha=1.5$. We set the weighting parameters $\{\theta_{k,t}\}$ by solving \eqref{pth-lss}.
%     \item For GClip and ACClip, we set the step sizes and clipping thresholds as $\{1/(k+1)^{\beta_1}\}$ and $\{(k+1)^{\beta_2}\}$, with $\beta_1>0$ and $\beta_2\in\R$ tuned to best suit each method's computation performance. For ACClip, we set the momentum parameters as $\{1/(k+1)^{\beta_3}\}$, with $\beta_3$ tuned to best suit its computational performance.
% \end{itemize}
% }
\begin{itemize}
\item For NSFOM-PM, NSFOM-EM, and NSFOM-RM, the step sizes and momentum parameters are chosen according to Theorems~\ref{thm:s-rate-pm}, \ref{thm:known-rate-mem}, and~\ref{thm:known-rate-rm}, respectively, with $\alpha = 1.5$. For NSFOM-EM, the extrapolation parameters are set equal to the momentum parameters in accordance with~\eqref{pth-lss} with $q = 1$.

\item  For GClip and ACClip, the step sizes and clipping thresholds are set as $\{(k+1)^{-\beta_1}\}$ and $\{(k+1)^{-\beta_2}\}$, respectively, where $\beta_1 > 0$ and $\beta_2 \in \mathbb{R}$ are tuned to optimize the empirical performance of each method. For ACClip, the momentum parameters are set as $\{(k+1)^{-\beta_3}\}$, where $\beta_3$ is likewise tuned to optimize its empirical performance.
\end{itemize}

\begin{figure}[htbp]
\centering
\includegraphics[width=.9\linewidth]{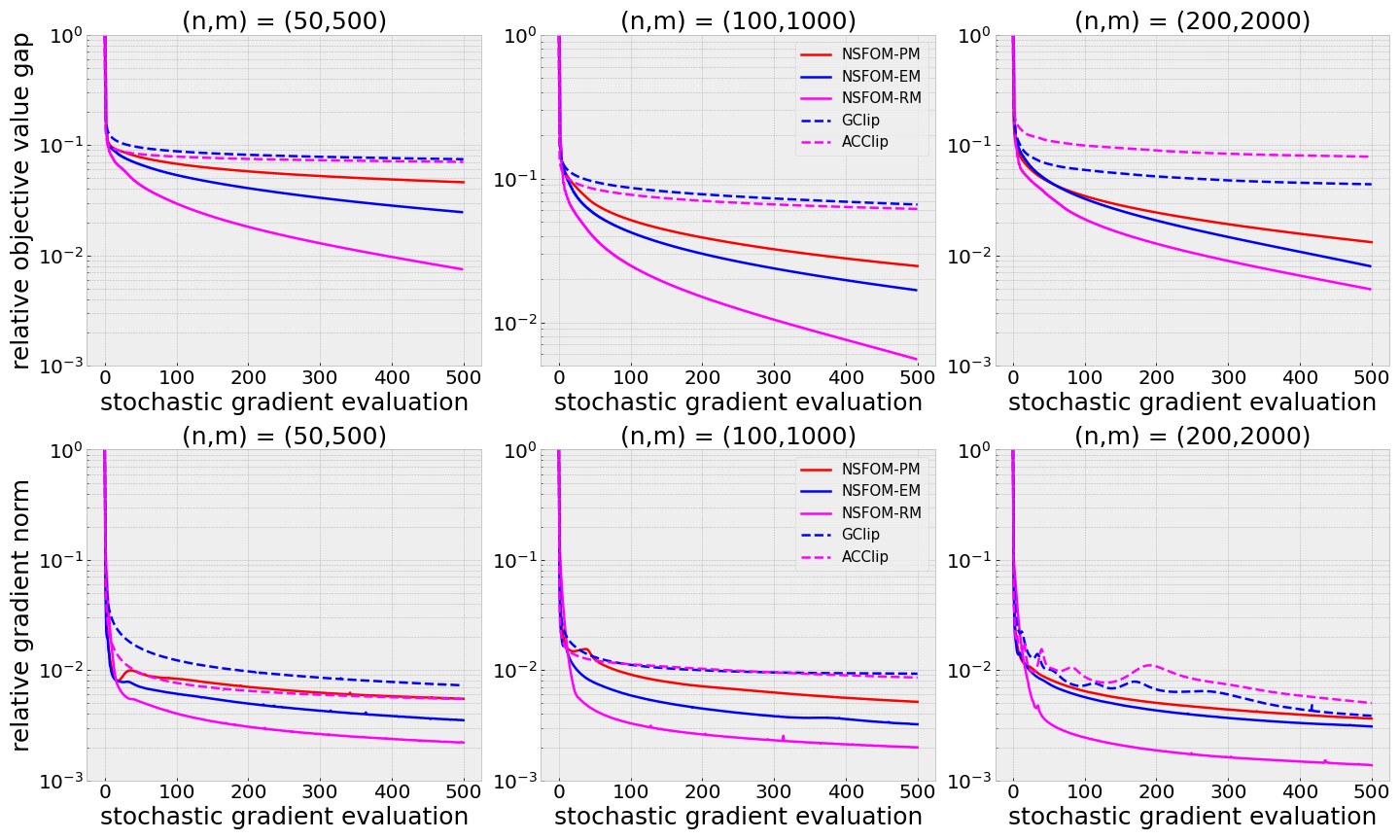}
\caption{Convergence behavior of the relative objective value gap (first row) and relative gradient norm (second row) for all the methods when solving problem \eqref{df}.}
\label{fig:df-loss}
\end{figure}

For each pair $(n,m)$, we plot the relative objective value gap and the relative gradient norm in Figure \ref{fig:df-loss} to illustrate the convergence behavior of all the SFOMs. As shown in Figure~\ref{fig:df-loss}, SFOMs with normalization (namely, NSFOM-PM, NSFOM-EM, NSFOM-RM) consistently outperform SFOMs with clipping (namely, GClip and ACClip) after an early stage of the iterations. This may be because gradient normalization allows a larger step size than gradient clipping when the stochastic gradient norm is below the clipping threshold. %their unnormalized counterparts. 
Furthermore, among the normalized variants, those using extrapolated momentum and recursive momentum converge faster than the one using Polyak momentum. Notably, the normalized SFOM with recursive momentum achieves the best performance, even outperforming the extrapolated momentum variant. These observations are consistent with our theoretical results.

\subsection{Robust regression problem}\label{subsec:rrp}
In this subsection, we consider the robust regression problem:
\begin{align}\label{robust-reg}
\min_{x\in\R^n}  \Big\{f(x) = \sum_{k=1}^m\sum_{i=1}^N \phi(a_{ik}^Tx-b_{ik})\Big\},
\end{align}
where $\phi(t)=t^2/(1+t^2)$ is a robust loss function \cite{carmon2017convex}, and $\{(a_{ik},b_{ik})\}_{1\le i\le N}\subset\R^n\times\R$, $1\le k\le m$, is the $k$th batch of the training set. We consider this problem on three real datasets, `red wine quality', `white wine quality', and `covtype' from 
the UCI repository.\footnote{see {\tt archive.ics.uci.edu/datasets}} For each dataset, we rescale both the features and predictions to lie in $[0,1]$, and set the batch size $N=100$.

We apply NSFOM-PM, NSFOM-EM, NSFOM-RM, GClip, and ACClip to solve \eqref{robust-reg}. All methods are initialized at the zero vector. Similar to Subsection \ref{subsec:df}, we compare these methods in terms of the relative objective value gap and the relative gradient norm, defined respectively as $(f(x^k)-f^*)/(f(x^0)-f^*)$ and $\|\nabla f(x^k)\|/\|\nabla f(x^0)\|$, over the first $500$ stochastic gradient evaluations, where $f^*$ is the minimum objective value found during the first $600$ stochastic gradient evaluations across all the SFOMs. %The algorithmic parameters are selected to suit each method well in terms of computational performance.
The algorithmic parameters are chosen according to the following strategy:
\begin{itemize}
    % \item For NSFOM-PM, NSFOM-EM, and NSFOM-RM, we choose the step sizes and the momentum parameters as $\{1/(k+1)^{\beta_1}\}$ and $\{1/(k+1)^{\beta_2}\}$, respectively, with $\beta_1,\beta_2>0$ tuned to best suit each method's computational performance. For NSFOM-EM, we set the extrapolation parameters to be the same as the momentum parameters according to \eqref{pth-lss} with $q=1$.
    \item For NSFOM-PM, NSFOM-EM, and NSFOM-RM, we choose the step sizes and the momentum parameters as $\{(k+1)^{-\beta_1}\}$ and $\{(k+1)^{-\beta_2}\}$, respectively, with $\beta_1,\beta_2>0$ tuned to optimize the empirical performance of each method. For NSFOM-EM, the extrapolation parameters are set equal to the momentum parameters, in accordance with~\eqref{pth-lss} with $q = 1$.
    %For NSFOM-EM, we set the step sizes $\{\eta_k\}$ and the extrapolation parameters $\{\gamma_{k}\}$ as $\{1/(k+1)^{\beta_1}\}$ and $\{\{1/(k+1)^{\beta_2}\}\}$, with $\beta_1$ and $\beta_2$ tuned to best suit each method in terms of computational performance. We set the weighting parameters $\theta_k=\gamma_k$ for all $k\ge0$ according to \eqref{pth-lss}.
    \item For GClip and ACClip, the algorithmic parameters are chosen  following the strategy described in Subsection~\ref{subsec:df}.
\end{itemize}

\begin{figure}[htbp]
\centering
\includegraphics[width=.9\linewidth]{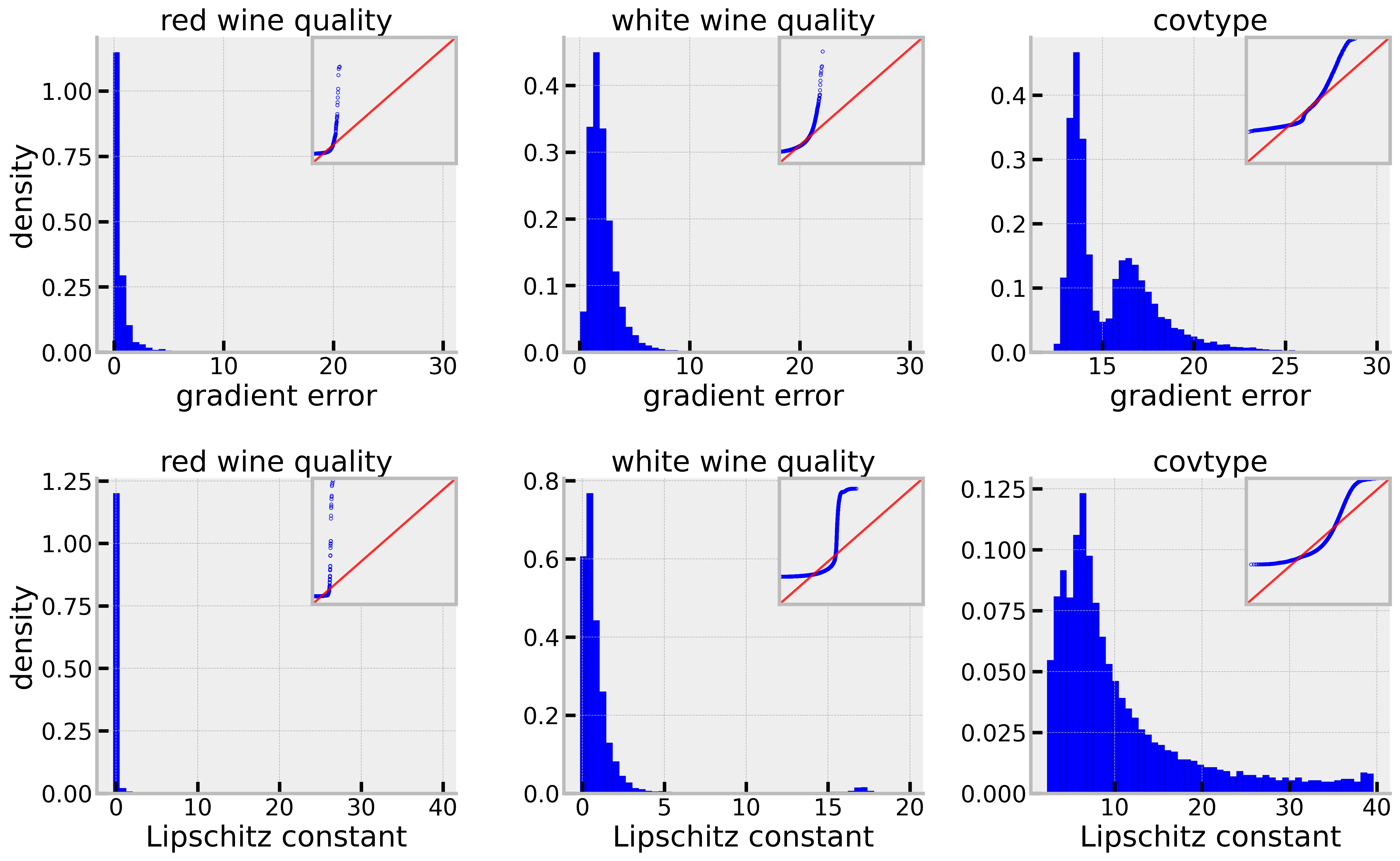}
\caption{Distributions of gradient errors $\|G(x;\xi)-\nabla f(x)\|$ (first row) and Lipschitz constant estimates $\|G(y;\xi)-G(x;\xi)\|/\|y-x\|$ (second row) compared against a normal distribution (QQ-plot), when solving \eqref{robust-reg}. Here, the gradient errors are calculated for the first epoch of optimization, and the Lipschitz constant estimates are taken over every two consecutive iterates within the first epoch of optimization for all methods.}
\label{fig:ht-gn-Lc}
\end{figure}

For each dataset, we visualize the distributions of gradient errors and Lipschitz constant estimates, compared against a normal distribution (QQ-plot) in Figure \ref{fig:ht-gn-Lc}, to illustrate their heavy-tailed behavior. This visualizations partly justify the heavy-tailed noise condition in Assumption \ref{asp:basic}(c) and the weakly average smoothness condition in Assumption \ref{asp:gen-ave-smth} when solving the regression problem \eqref{robust-reg}. 

\begin{figure}[htbp]
\centering
\includegraphics[width=.9\linewidth]{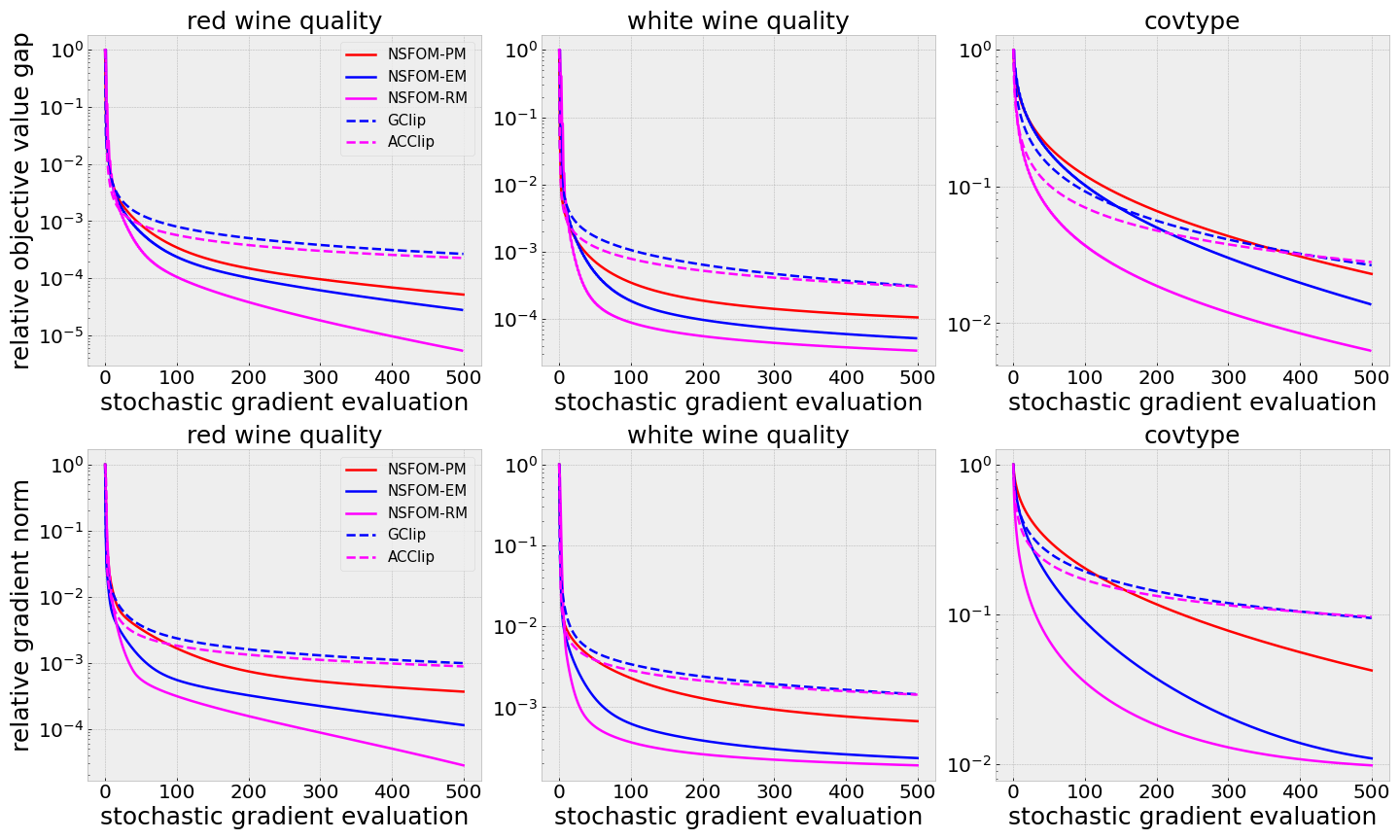}
\caption{Convergence behavior of the relative objective value gap (first row) and relative gradient norm (second row) for all the methods when solving problem \eqref{robust-reg}.}
\label{fig:real-reg}
\end{figure}

In addition, for each dataset, we plot the relative objective value gap and the relative gradient norm in Figure \ref{fig:real-reg} to illustrate the convergence behavior of all the SFOMs. From Figure \ref{fig:real-reg}, we observe that SFOMs with normalization (namely, NSFOM-PM, NSFOM-EM, NSFOM-RM) consistently outperform SFOMs with clipping (namely, GClip and ACClip) after an early stage of the iterations. This may be because gradient normalization leads to a larger step size than gradient clipping when the stochastic gradient norm is below the clipping threshold. We also observe that the normalized SFOMs with extrapolated and recursive momentum are faster than the SFOM with Polyak momentum, while the SFOM with recursive momentum outperforms the SFOM with extrapolated momentum. These observations align with our theoretical results.

\subsection{Multimodal contrastive learning problem}\label{subsec:mcl}
In this subsection, we consider the multimodal contrastive learning problem (see \cite{radford2021learning}):
{\small
\begin{align}
\min_{x_I\in\R^{n_I},x_T\in\R^{n_T}}  -\sum_{k=1}^m\sum_{i=1}^N\bigg(\ln\bigg(\frac{\exp(f_{x_I}(a_{ik})^Tf_{x_T}(b_{ik})/\tau)}{\sum_{j=1}^N\exp(f_{x_I}(a_{ik})^Tf_{x_T}(b_{jk})/\tau)}\bigg) + \ln\bigg(\frac{\exp(f_{x_I}(a_{ik})^Tf_{x_T}(b_{ik})/\tau)}{\sum_{j=1}^N\exp(f_{x_I}(a_{jk})^Tf_{x_T}(b_{ik})/\tau)}\bigg)\bigg),\label{multi-modal} 
\end{align}
}

\noindent where $\{(a_{ik},b_{ik})\}_{1\le i\le N}$, $1\le k\le m$, denotes the image-caption pairs for the $k$th batch of training dataset, $f_{x_I}$ and $f_{x_T}$ are the image and text encoders, respectively, and $\tau>0$ is a temperature parameter. Here, we consider problem \eqref{multi-modal} on three real text-image datasets, Flickr \cite{plummer2015flickr30k}, MSCOCO \cite{lin2014microsoft}, and CC3M \cite{sharma2018conceptual}, and choose the network structure for image encoder $f_{x_I}$ and text encoder $f_{x_T}$ as ResNet50 \cite{he2016deep} and DistilBERT \cite{sanh2019distilbert}, respectively.  

We apply NSFOM-PM, NSFOM-EM, NSFOM-RM, GClip, and ACClip to solve problem \eqref{multi-modal}. We use the same initial weights for all methods as those of the pretrained models ResNet50 and DistilBERT. Similar to Subsection \ref{subsec:df}, we compare these methods in terms of the relative objective value gap and the relative gradient norm, defined respectively as $(f(x^k)-f^*)/(f(x^0)-f^*)$ and $\|\nabla f(x^k)\|/\|\nabla f(x^0)\|$, over the first $50,000$ stochastic gradient evaluations, where $f$ denotes the objective function of \eqref{multi-modal} and $f^*$ is the minimum objective value found during the first $60,000$ stochastic gradient evaluations across all the SFOMs. The other algorithmic parameters are chosen according to the strategy described in Subsection \ref{subsec:rrp}.%to suit each method well in terms of computational performance.

\begin{figure}[htbp]
\centering
\includegraphics[width=.9\linewidth]{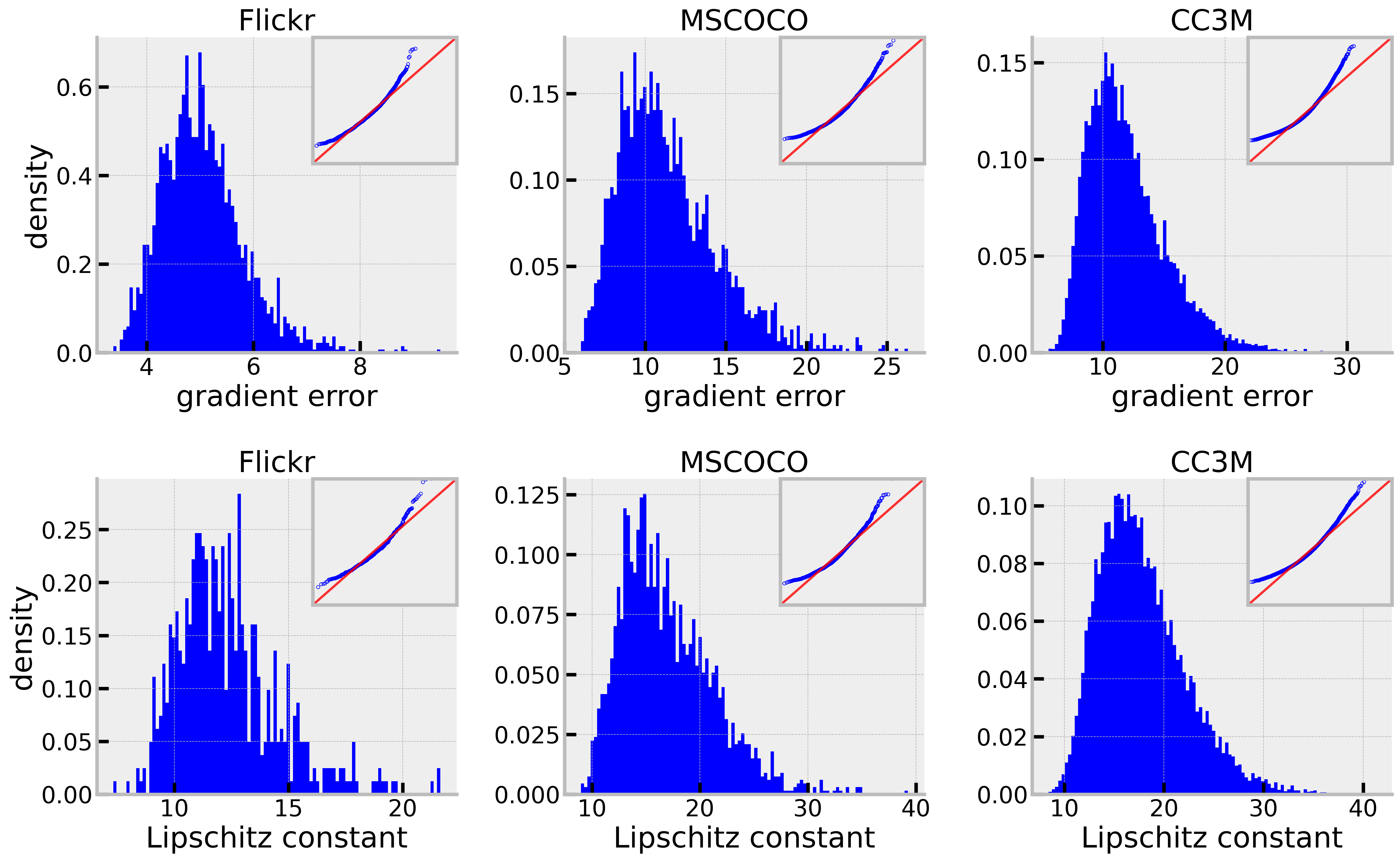}
\caption{Distributions of gradient errors $\|G(x;\xi)-\nabla f(x)\|$ (first row) and Lipschitz constant estimates $\|G(x;\xi)-G(y;\xi)\|/\|x-y\|$ (second row) compared against a normal distribution (QQ-plot), when solving \eqref{multi-modal}. Here, the gradient errors are calculated for the first epoch of training, and the Lipschitz constant estimates are taken over every two consecutive iterates within the first epoch of training for all methods.}
\label{fig:ht-gn-Lc-mm}
\end{figure}

For each dataset, we visualize the distributions of gradient errors and Lipschitz constant estimates, compared against a normal distribution (QQ-plot) in Figure \ref{fig:ht-gn-Lc-mm}, to illustrate their heavy-tailed behavior. These visualizations provide partial justification for the heavy-tailed noise condition in Assumption \ref{asp:basic}(c) and the weakly average smoothness condition in Assumption \ref{asp:gen-ave-smth} when solving the multimodal contrastive learning problem \eqref{multi-modal}. 

\begin{figure}[htbp]
\centering
\includegraphics[width=.9\linewidth]{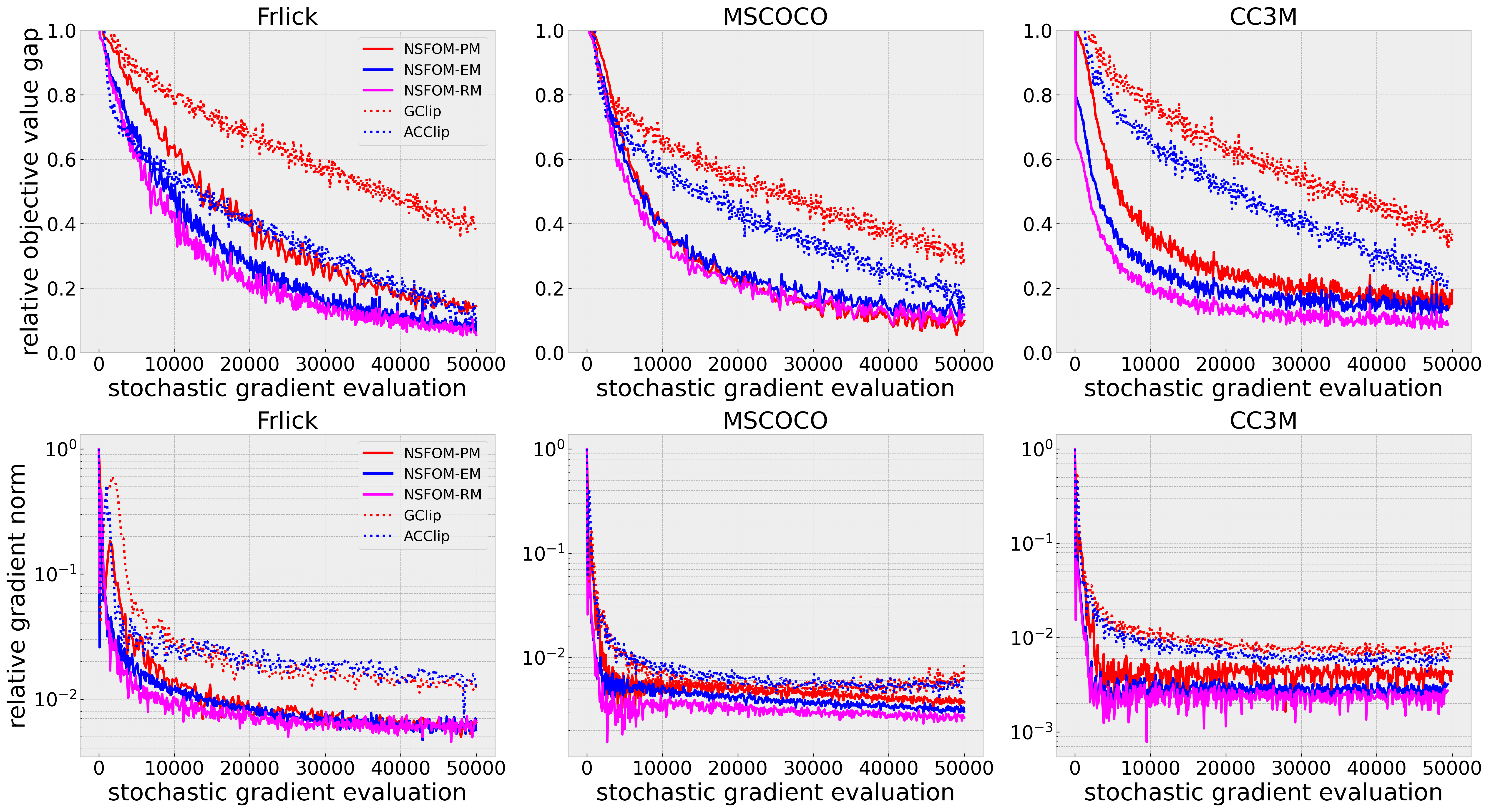}
\caption{Convergence behavior of the relative objective value gap (first row) and relative gradient norm (second row) for all  the methods when solving problem \eqref{multi-modal}.}
\label{fig:mm}
\end{figure}

In addition, for each dataset, we plot the relative objective value gap and the relative gradient norm in Figure \ref{fig:mm} to illustrate the convergence behavior of the SFOMs. %From Figure \ref{fig:mm}, we can observe that SFOMs with normalization tend to outperform their unnormalized counterparts. 
From Figure \ref{fig:mm}, we observe that the normalized SFOMs (namely, NSFOM-PM, NSFOM-EM, NSFOM-RM) outperform the clipped SFOMs  (namely, GClip and ACClip) after an early stage of iterations. This is likely because gradient normalization allows a larger step size than gradient clipping when the stochastic gradient norm falls below the clipping threshold. We also observe that the normalized SFOMs with extrapolated and recursive momentum converge faster than the SFOM with Polyak momentum, while the SFOM with recursive momentum slightly outperforms the one with extrapolated momentum. These phenomena are generally consistent with our theoretical results.

\subsection{Comparison over different algorithmic parameters of the normalized SFOMs}
In this subsection, we investigate the performance of the normalized SFOMs (namely, NSFOM-PM, NSFOM-EM, and NSFOM-RM) under different choices of algorithmic parameters.

\begin{figure}[htbp]
\centering
\includegraphics[width=.9\linewidth]{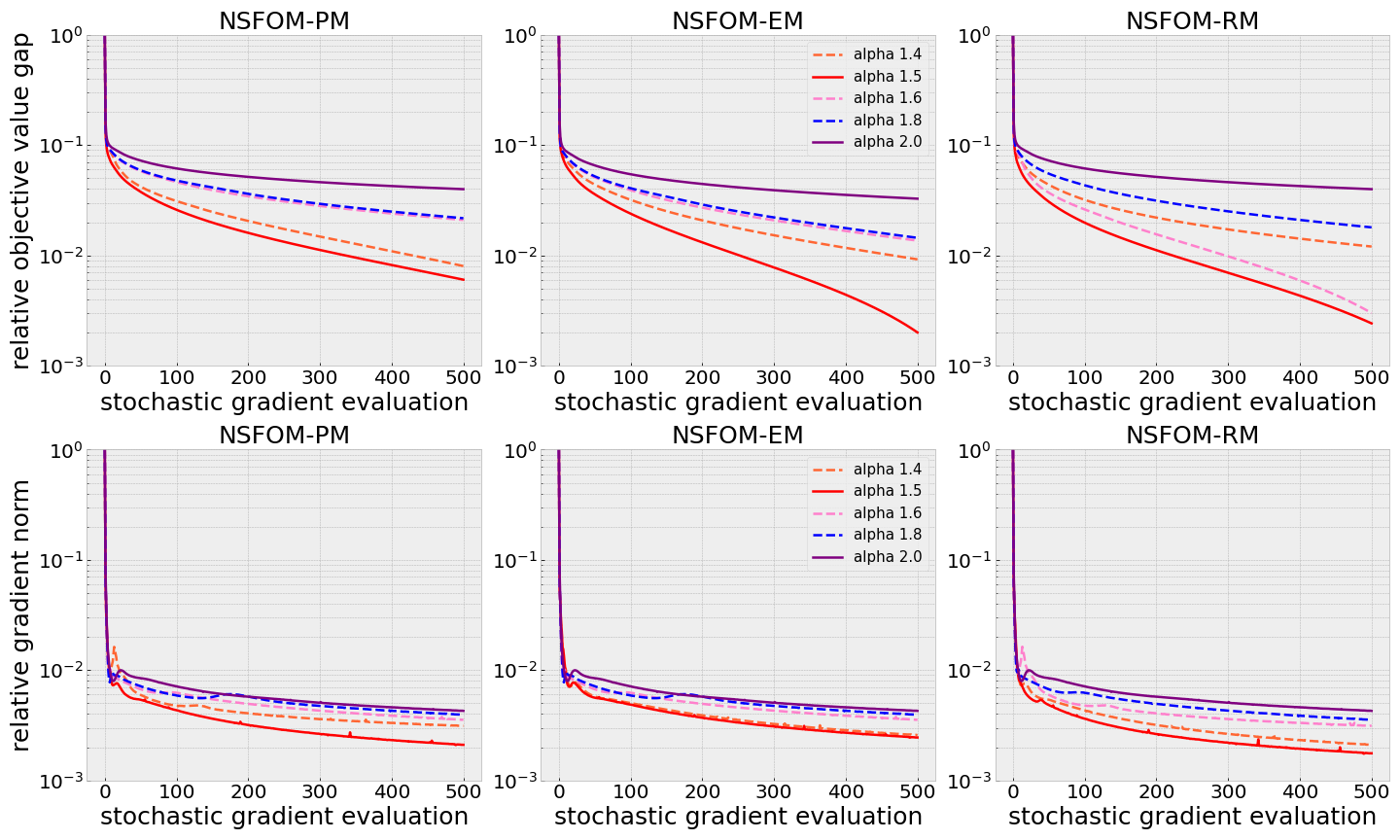}
\caption{Convergence behavior of the relative objective value gap (first row) and relative gradient norm (second row) for the normalized SFOMs when solving problem \eqref{df} with $(n,m)=(200,2000)$.}
\label{fig:df-tune}
\end{figure}

% \begin{figure}[htbp]
% \centering
% \includegraphics[width=.9\linewidth]{figs/RR_tuned.png}
% \caption{{\color{blue}Convergence behavior of the relative objective value gap (first row) and relative gradient norm (second row) for untuned normalized methods and tuned methods in solving problem \eqref{robust-reg} on the ``white wine quality" dataset.}}
% \label{fig:rr-tune}
% \end{figure}

\begin{figure}[htbp]
\centering
\includegraphics[width=.9\linewidth]{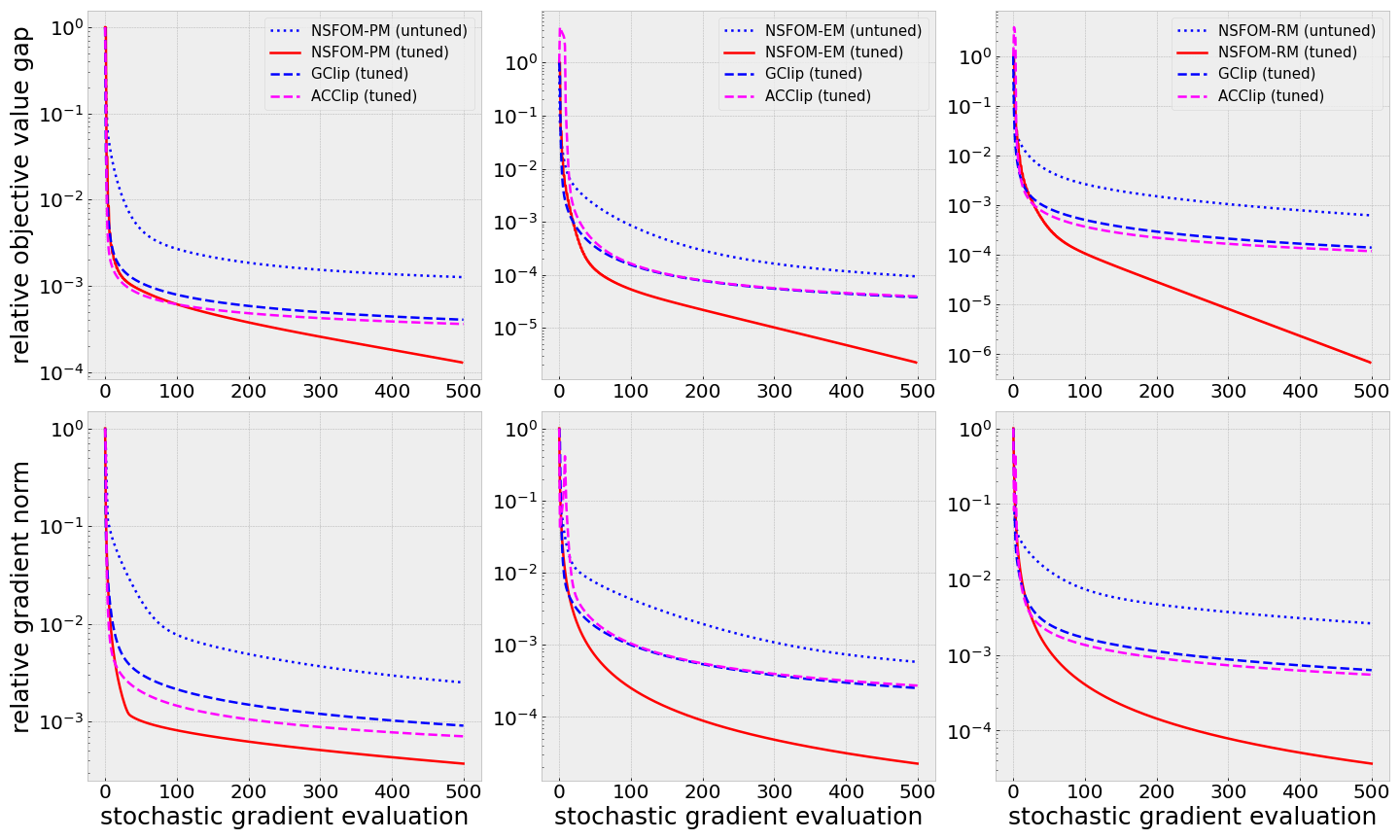}
\caption{Convergence behavior of the relative objective value gap (first row) and relative gradient norm (second row) for the untuned normalized SFOMs and tuned SFOMs when solving problem \eqref{robust-reg} on the White Wine Quality dataset.}
\label{fig:rr-tune}
\end{figure}

% \begin{figure}[htbp]
% \centering
% \includegraphics[width=.9\linewidth]{figs/MM_tuned.png}
% \caption{{\color{blue}Convergence behavior of the relative objective value gap (first row) and relative gradient norm (second row) for untuned normalized methods and tuned methods in solving problem \eqref{multi-modal} on the ``Frlick" dataset.}}
% \label{fig:mm-tune}
% \end{figure}

\begin{figure}[htbp]
\centering
\includegraphics[width=.9\linewidth]{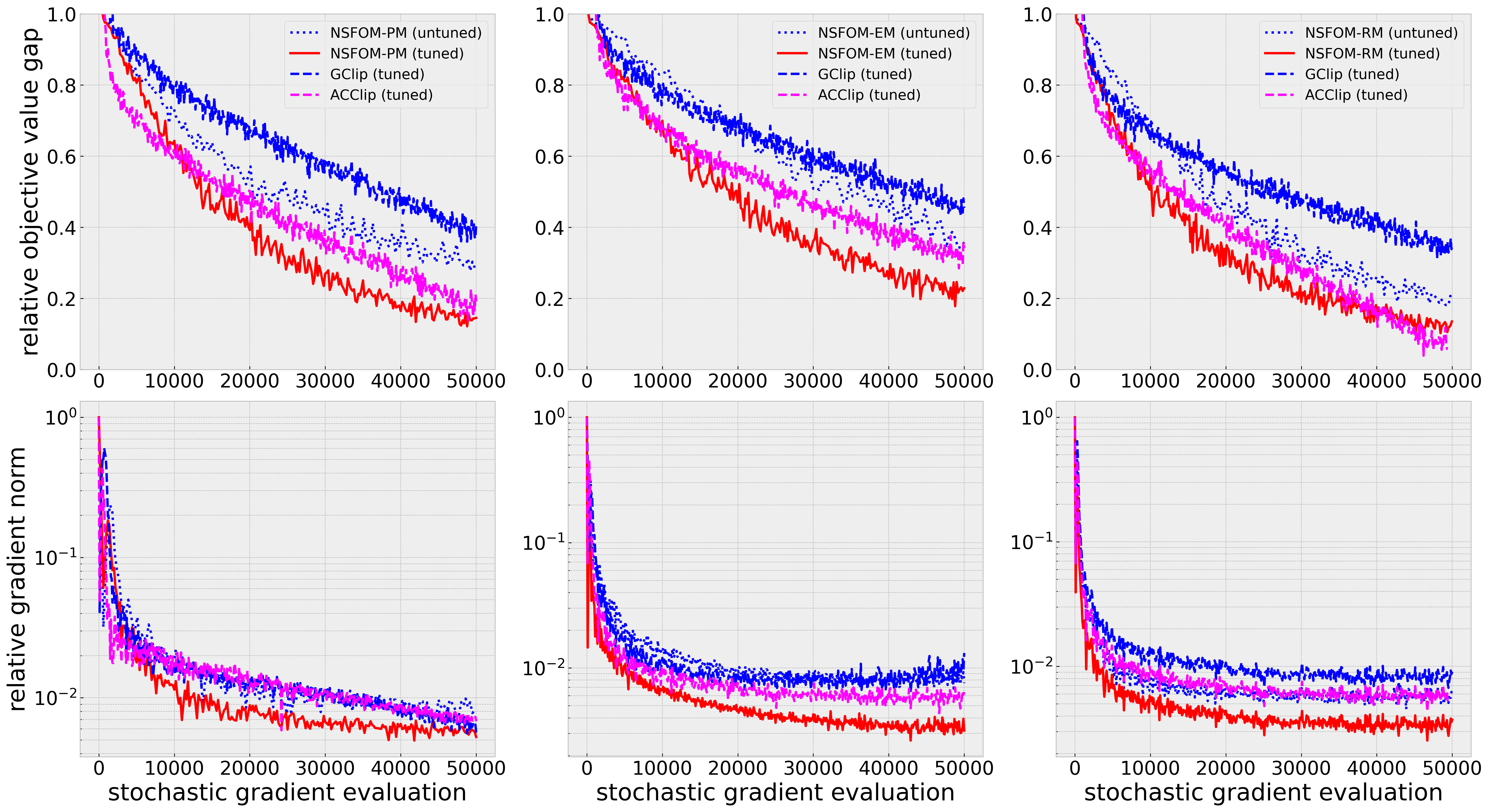}
\caption{Convergence behavior of the relative objective value gap (first row) and relative gradient norm (second row) for the untuned normalized SFOMs and tuned SFOMs when solving problem \eqref{multi-modal} on the Frlick dataset.}
\label{fig:mm-tune}
\end{figure}

% For the data fitting problem \eqref{df}, we plot the relative objective value gap and the relative gradient norm in Figure~\ref{fig:df-tune} to illustrate the convergence behavior of normalized SFOMs with different choice of algorithmic parameters. Specifically, we first apply normalized SFOMs to this problem with $\alpha=1.5$, using the parameter settings specified in Theorems \ref{thm:s-rate-pm}, \ref{thm:known-rate-mem}, and \ref{thm:known-rate-rm}. Second, we artificially treat $\alpha$ as unknown and apply normalized SFOMs with the parameter
% settings specified in Theorems~\ref{thm:un-s-rate-pm}, \ref{thm:unknown-rate-mem}, and \ref{thm:unknown-rate-rm}. Third, again treating $\alpha$ as unknown, we apply normalized SFOMs with the parameter settings specified in Theorems~\ref{thm:s-rate-pm}, \ref{thm:known-rate-mem}, and \ref{thm:known-rate-rm}, but using a grid of trial values for $\alpha$. From Figure~\ref{fig:df-tune}, we observe that choosing the algorithmic parameters specified in Theorems~\ref{thm:s-rate-pm}, \ref{thm:known-rate-mem}, and \ref{thm:known-rate-rm}, based on $\alpha = 1.5$, results in the best convergence compared with other choices of $\alpha$. This observation indicates that these parameter choices, with $\alpha$ matched to the gradient noise level, outperform alternative selections.

We begin by comparing the performance of the normalized SFOMs under different algorithmic parameter choices for solving the data fitting problem \eqref{df}. Specifically, we consider three parameterization strategies. First, we apply the normalized SFOMs with $\alpha = 1.5$, using the parameter settings specified in Theorems \ref{thm:s-rate-pm}, \ref{thm:known-rate-mem}, and \ref{thm:known-rate-rm}. Second, we artificially treat $\alpha$ as unknown and apply the normalized SFOMs with the parameter
settings given in Theorems~\ref{thm:un-s-rate-pm}, \ref{thm:unknown-rate-mem}, and \ref{thm:unknown-rate-rm}. Third, again treating $\alpha$ as unknown, we apply the normalized SFOMs with the parameter settings specified in Theorems~\ref{thm:s-rate-pm}, \ref{thm:known-rate-mem}, and \ref{thm:known-rate-rm}, but with a grid of trial values for $\alpha$. We plot the relative objective value gap and the relative gradient norm in Figure~\ref{fig:df-tune} to illustrate the convergence behavior of the normalized SFOMs under different parameter choices. As shown in Figure~\ref{fig:df-tune}, the parameter settings based on $\alpha = 1.5$, as prescribed in Theorems~\ref{thm:s-rate-pm}, \ref{thm:known-rate-mem}, and \ref{thm:known-rate-rm}, achieve the fastest convergence among all tested configurations. This result demonstrates that aligning $\alpha$ with the gradient noise level leads to superior practical performance compared with alternative parameter selections.

% We present the performance of the tuned SFOMs and the untuned normalized SFOMs for solving the robust regression problem~\eqref{robust-reg} and the multimodal contrastive learning problem~\eqref{multi-modal} in Figures~\ref{fig:rr-tune} and~\ref{fig:mm-tune}, respectively. For the tuned SFOMs, we choose the algorithmic parameters as indicated in Section~\ref{subsec:rrp}. For the untuned variants, we choose the algorithmic parameters as indicated in Theorems~\ref{thm:un-s-rate-pm}, \ref{thm:unknown-rate-mem}, and \ref{thm:unknown-rate-rm}. From Figure~\ref{fig:rr-tune}, we see that when solving \eqref{robust-reg} on the ``white wine quality" dataset, the performance of normalized SFOMs degrades if the algorithmic parameters are untuned, and it can be worse than that of the tuned GClip and ACClip. From Figure~\ref{fig:mm-tune}, we see that when solving \eqref{multi-modal} on the ``Frlick'' dataset, the performance of normalized SFOMs worsens when the algorithmic parameters are untuned. In addition, the untuned normalized SFOMs perform worse than the tuned ACClip, while still slightly outperforming the tuned GClip.

We next compare the performance of the tuned SFOMs (namely, tuned NSFOM-PM, NSFOM-EM, NSFOM-RM, GClip, and ACClip) and the untuned normalized SFOMs (namely, untuned NSFOM-PM, NSFOM-EM, and NSFOM-RM) for solving the robust regression problem~\eqref{robust-reg} and the multimodal contrastive learning problem~\eqref{multi-modal}. For the tuned SFOMs, algorithmic parameters are selected as described in Subsection~\ref{subsec:rrp}. For the untuned SFOMs, the parameters are chosen according to Theorems~\ref{thm:un-s-rate-pm}, \ref{thm:unknown-rate-mem}, and \ref{thm:unknown-rate-rm}. The performance of the tuned and untuned methods on problems~\eqref{robust-reg} and~\eqref{multi-modal} is reported in Figures~\ref{fig:rr-tune} and~\ref{fig:mm-tune}, respectively. From Figure~\ref{fig:rr-tune}, we observe that when solving problem \eqref{robust-reg} on the White Wine Quality dataset, the performance of the normalized SFOMs deteriorates when the algorithmic parameters are untuned, and can even be inferior to that of the tuned GClip and ACClip methods. Similarly, Figure~\ref{fig:mm-tune} shows that for problem~\eqref{multi-modal} on the Flickr dataset, the performance of the normalized SFOMs also degrades when the parameters are untuned. In this setting, the untuned normalized SFOMs perform worse than the tuned ACClip, while still slightly outperforming the tuned GClip.

\begin{figure}[htbp]
\centering
\includegraphics[width=.9\linewidth]{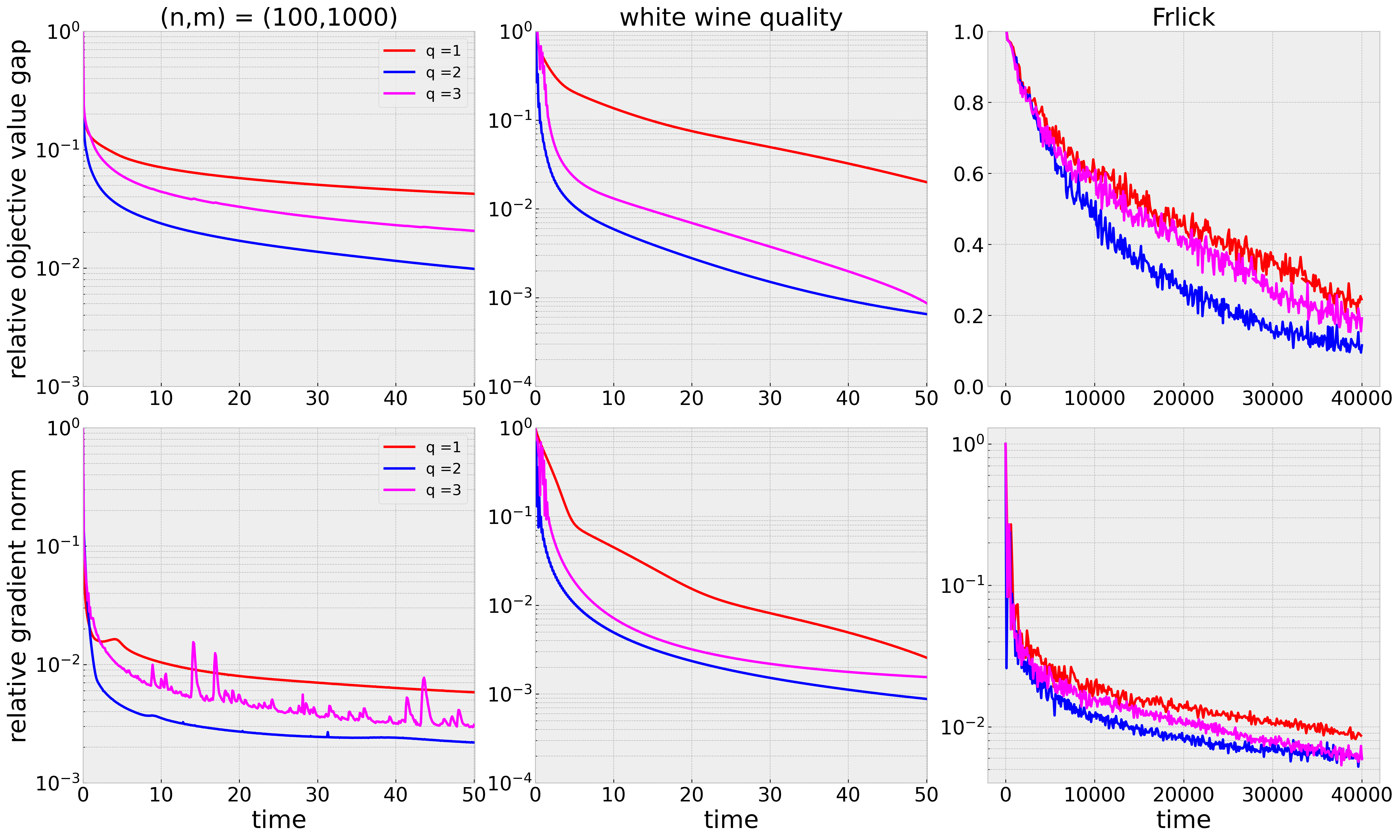}
\caption{Convergence behavior of the relative objective value gap (first row) and the relative gradient norm (second row) for Algorithm~\ref{alg:unf-sfom} with different values of $q$ when solving problems~\eqref{df}, \eqref{robust-reg}, and \eqref{multi-modal}.}
\label{fig:ec}
\end{figure}

% Finally, we compare Algorithm~\ref{alg:unf-sfom} with $q=1,2$, and $3$ for solving problems~\eqref{df}, \eqref{robust-reg}, and \eqref{multi-modal}, respectively, under a fixed maximum computation time (in seconds). For problem~\eqref{df}, we set $(n,m) = (100,1000)$, and for problems~\eqref{robust-reg} and~\eqref{multi-modal}, we consider the `white wine quality' and `Flickr' datasets.} For Algorithm~\ref{alg:unf-sfom} with each value of $q$, we set the step sizes $\{\eta_k\}$ and the extrapolation parameters $\{\gamma_{k,t}\}$ as $\{1/(k+1)^{\beta_1}\}$ and $\{\{1/(t^2(k+1)^{\beta_2})\}\}$, with $\beta_1,\beta_2>0$ tuned individually. We choose the weighting parameters $\{\theta_{k,t}\}$ by solving \eqref{pth-lss}. The performance of Algorithm~\ref{alg:unf-sfom} is reported in Figure~\ref{fig:ec}. As shown in Figure~\ref{fig:ec}, Algorithm~\ref{alg:unf-sfom} with $q = 2$ or $3$ generally outperforms the case $q = 1$. Moreover, Algorithm~\ref{alg:unf-sfom} with $q = 3$ often exhibits worse performance than that with $q = 2$. This phenomenon may be due to the fact that the per-iteration computational cost appears to increase more rapidly than the acceleration gain provided by the extrapolation scheme when increasing $q$ from $2$ to $3$.
Finally, we compare Algorithm~\ref{alg:unf-sfom} with $q = 1, 2,$ and $3$ for solving problems~\eqref{df}, \eqref{robust-reg}, and \eqref{multi-modal}, respectively, under a fixed maximum computation time (in seconds). Specifically, for problem~\eqref{df}, we set $(n,m) = (100,1000)$ and randomly generate instances in the same manner as described in Subsection~\ref{subsec:df}. For problems~\eqref{robust-reg} and~\eqref{multi-modal}, we use the White Wine Quality and Flickr datasets, respectively. For each value of $q$, we set the step sizes $\{\eta_k\}$ and the extrapolation parameters $\{\gamma_{k,t}\}$ of Algorithm~\ref{alg:unf-sfom} to $\{(k+1)^{-\beta_1}\}$ and $\{t^{-2}(k+1)^{-\beta_2}\}$, respectively, where $\beta_1, \beta_2 > 0$ are tuned individually to optimize empirical performance. In addition, the weighting parameters ${\theta_{k,t}}$ are chosen by solving~\eqref{pth-lss}. The performance of Algorithm~\ref{alg:unf-sfom} is reported in Figure~\ref{fig:ec}. As shown in Figure~\ref{fig:ec}, Algorithm~\ref{alg:unf-sfom} with $q = 2$ or $3$ generally outperforms the case $q = 1$. Moreover, Algorithm~\ref{alg:unf-sfom} with $q = 3$ often exhibits worse performance compared with $q = 2$. This phenomenon may be due to the fact that the per-iteration computational cost appears to increase more rapidly than the acceleration gain provided by the extrapolation scheme when increasing $q$ from $2$ to $3$.

\section{Proof of the main results}\label{sec:pf-man}
In this section, we present the proofs of the main results stated in Section~\ref{sec:nsgm}, namely, Theorems~\ref{thm:s-rate-pm} through~\ref{thm:unknown-rate-rm}.

\subsection{Proof overview and technical lemmas}

% In this subsection, we provide an overview of the proofs of the main results and the technical lemmas that will be used.
This subsection provides an overview of the proof strategy for the main results and introduces the technical lemmas that will be used throughout the analysis.

% We will derive the complexity bounds of Algorithms~\ref{alg:unf-sfom-pm}, \ref{alg:unf-sfom}, and \ref{alg:unf-sfom-rm} in a unified manner by establishing the descent of a sequence of potential functions defined as
% \begin{align}\label{def:pot-pm}
% {\mathcal P}_k := f(x^k) + p_k\|m^k - \nabla f(x^k)\|^\alpha \qquad \forall k \ge 0,
% \end{align}
% where the sequence $\{(x^k, m^k)\}$ is generated by the corresponding algorithm, and $\{p_k\}$ is a sequence of positive scalars specified separately for each case. Specifically, the descent of $\{\mathcal{P}_k\}$ will be established by combining the descent properties of the sequence of function values $\{f(x^k)\}$ and the sequence of gradient estimation errors $\{\|m^k - \nabla f(x^k)\|^\alpha\}$. The descent property for $\{f(x^k)\}$ will be derived in Lemma \ref{lem:ppt-desc} by analyzing the update along a normalized direction. In addition, the descent property for $\{\|m^k - \nabla f(x^k)\|^\alpha\}$ depends on the momentum scheme, as shown in Lemmas \ref{lem:rec-pm}, \ref{lem:rec-em}, and \ref{lem:rec-rm}. Under stronger smoothness assumptions than gradient Lipschitz continuity, the accelerated schemes, i.e., extrapolated momentum and recursive momentum, exhibit faster descent compared with the Polyak momentum.
We derive the complexity bounds of Algorithms~\ref{alg:unf-sfom-pm}, \ref{alg:unf-sfom}, and \ref{alg:unf-sfom-rm} in a unified manner by establishing the descent of a potential sequence defined as
\begin{align}\label{def:pot-pm}
{\mathcal P}_k := f(x^k) + p_k\|m^k - \nabla f(x^k)\|^\alpha \qquad \forall k \ge 0,
\end{align}
where the sequence $\{(x^k, m^k)\}$ is generated by the corresponding algorithm, and $\{p_k\}$ is a sequence of positive scalars specified separately for each case. The descent of the potential sequence $\{\mathcal P_k\}$ is established by combining the descent properties of the function value sequence $\{f(x^k)\}$ and the gradient estimation error sequence $\{\|m^k - \nabla f(x^k)\|^\alpha\}$. Specifically, the descent property of $\{f(x^k)\}$ is derived in Lemma~\ref{lem:ppt-desc} by analyzing the update along a normalized direction. The descent property of $\{\|m^k - \nabla f(x^k)\|^\alpha\}$ is derived based on the choice of momentum scheme (see Lemmas~\ref{lem:rec-pm}, \ref{lem:rec-em}, and \ref{lem:rec-rm}).

% The following lemma provides an expansion for the $\alpha$-power of the Euclidean norm, generalizing the well-known identity $\|u + v\|^2 = \|u\|^2 + 2u^T v + \|v\|^2$ and inequality $\|u + v\|^2\le (1+c)\|u\|^2+(1+1/c)\|v\|^2$ for all $u,v\in\mathbb{R}^n$ and $c>0$. {\color{blue}This lemma will be used frequently for deriving the descent properties for $\{\|m^k - \nabla f(x^k)\|^\alpha\}$.}

The following lemma provides an expansion for the $\alpha$-power of the Euclidean norm, generalizing the well-known identity $\|u + v\|^2 = \|u\|^2 + 2u^T v + \|v\|^2$ and inequality $\|u + v\|^2\le (1+c)\|u\|^2+(1+1/c)\|v\|^2$ for all $u,v\in\mathbb{R}^n$ and $c>0$. It will be used repeatedly to derive the descent properties of the gradient estimation error sequence $\{\|m^k - \nabla f(x^k)\|^\alpha\}$.

\begin{lemma}\label{lem:expan-alpha}
For any $\alpha\in(1,2]$, it holds that
\begin{align}
&\|u+v\|^\alpha \le \|u\|^\alpha + \alpha\|u\|^{\alpha-2} u^Tv + 2\|v\|^\alpha\qquad\forall u,v\in\mathbb{R}^n,\label{open-alpha}\\
&\|u+v\|^\alpha\le (1+c)\|u\|^\alpha + (2 + (\alpha-1)^{\alpha-1}c^{1-\alpha})\|v\|^\alpha\qquad\forall u,v\in\mathbb{R}^n, c>0. \label{open-alpha-2}   
\end{align}
\end{lemma}

\begin{proof}
Let $\alpha\in(1,2]$ be arbitrarily chosen, and $\psi(w)=\|w\|^\alpha$ for all $w\in\R^n$. It follows from \cite[Theorem 6.3]{rodomanov2020smoothness} that 
\[
\|\nabla\psi(w^1) - \nabla\psi(w^2)\|\le 2^{2-\alpha}\alpha\|w^1-w^2\|^{\alpha-1} \quad \forall w^1, w^2\in\mathbb{R}^n.
\]
By this, one has that for any $u,v\in\R^n$, 
\begin{align*}
&|\psi(u+v) - \psi(u) - \nabla\psi(u)^Tv|=\Big|\int^1_0(\nabla \psi(u+tv) - \nabla \psi(u))^Tv\mathrm{d}t\Big|\\
&\le  \int^1_0\|\nabla \psi(u+tv) - \nabla \psi(u)\|\mathrm{d}t\cdot \|v\|\le 2^{2-\alpha}\alpha \int^1_0\|tv\|^{\alpha-1}\mathrm{d}t\cdot \|v\| = 2^{2-\alpha}\|v\|^\alpha,
\end{align*}
 This along with $\alpha\in(1,2]$ and the fact that $\psi(w)=\|w\|^\alpha$ and $\nabla\psi(w)=\alpha\|w\|^{\alpha-2}w$ implies that \eqref{open-alpha} holds. We next prove \eqref{open-alpha-2}. Let $\alpha^\prime=\alpha/(\alpha-1)$. By the Young's inequality, one has that for all $c>0$,
\[
\alpha\|u\|^{\alpha-2} u^Tv \le \frac{\big((c\alpha^\prime)^{1/\alpha^\prime}\big\|\|u\|^{\alpha-2} u\big\|\big)^{\alpha^\prime}}{\alpha^\prime} + \frac{\big(\alpha\|v\|/(c\alpha^\prime)^{1/\alpha^\prime}\big )^\alpha}{\alpha} =c\|u\|^\alpha + \frac{(\alpha-1)^{\alpha-1}\|v\|^\alpha}{c^{\alpha-1}},
\]
which together with \eqref{open-alpha} implies that \eqref{open-alpha-2} holds.
\end{proof}

The following lemma provides an estimation of the partial sums of series.

\begin{lemma}
Let $\zeta(\cdot)$ be a convex univariate function. Then it holds that $\sum_{r=a}^b\zeta(r)\le \int^{b+1/2}_{a-1/2}\zeta(\tau)\mathrm{d}\tau$ for any integers $a,b$ satisfying $[a-1/2,b+1/2]\subset\mathrm{dom}\,\zeta$. Consequently,  one has 
\begin{align}
\sum_{r=a}^b \frac{1}{r^\beta} \le \left\{\begin{array}{ll}
\ln\big(b+\frac{1}{2}\big) - \ln\big(a-\frac{1}{2}\big)&\text{if }\beta=1,  \\[4pt]
\frac{1}{1-\beta}\big(\big(b+\frac{1}{2}\big)^{1-\beta} - \big(a-\frac{1}{2}\big)^{1-\beta}\big)&\text{if }\beta\in(0,1)\cup(1,+\infty).
\end{array}\right.
\label{upbd:series-ka} 
\end{align}
\end{lemma}

\begin{proof}
Let $a,b$ be integers satisfying $[a-1/2,b+1/2]\subset\mathrm{dom}\,\zeta$. Since $\zeta$ is convex, one has 
$\zeta(\tau) \geq \zeta(r)+s^T (\tau-r)$ for all $s\in\partial \zeta(r)$ and $r\in[a,b]$. It then follows that 
\[
 \int_{r-1/2}^{r+1/2}\zeta(\tau)\mathrm{d}\tau \geq  \int_{r-1/2}^{r+1/2} \big(\zeta(r)+s^T (\tau-r)\big) \mathrm{d}\tau = \zeta(r),
\]
which implies that $\sum_{r=a}^b\zeta(r)\le \sum^b_{r=a}\int_{r-1/2}^{r+1/2}\zeta(\tau)\mathrm{d}\tau=\int_{a-1/2}^{b+1/2}\zeta(\tau)\mathrm{d}\tau$.  By this and  $\zeta(\tau)=1/\tau^\beta$, one can see that \eqref{upbd:series-ka} holds. 
\end{proof}

We next provide a lemma that will be used to derive complexity bounds subsequently.

\begin{lemma}\label{lem:rate-complexity}
Let $\beta\in(0,1)$ and $u\in(0,1/e)$ be given. Then $v^{-\beta}\ln v\le 2u/\beta$ holds for all $v\ge(u^{-1}\ln(1/u))^{1/\beta}$.
\end{lemma}

\begin{proof}
Fix any $v$ satisfying $v\ge(u^{-1}\ln(1/u))^{1/\beta}$. It then follows from $u\in(0,1/e)$ that
\begin{align}\label{aux-uv-ln}
v\ge(u^{-1}\ln(1/u))^{1/\beta}> e^{1/\beta}.    
\end{align}
Let $\phi(\tau)= \tau^{-\beta}\ln \tau$. It can be verified that $\phi$ is decreasing on $(e^{1/\beta},\infty)$. By this and  \eqref{aux-uv-ln}, one has that 
\begin{align*}
v^{-\beta}\ln v = \phi(v) \le  \phi((u^{-1}\ln(1/u))^{1/\beta}) = \frac{u}{\beta} \Big(1+\frac{\ln\ln(1/u)}{\ln(1/u)}\Big)\le \frac{2u}{\beta},
\end{align*}
where the last inequality follows from $\ln\ln(1/u)\le\ln(1/u)$ due to $u\in(0,1/e)$. Hence, the conclusion of this lemma holds.
\end{proof}

We next establish a descent property for $f$ along a normalized direction.

\begin{lemma}\label{lem:ppt-desc}
Suppose that Assumption \ref{asp:basic} holds. Let $x,m\in\mathbb{R}^n$ and $\eta>0$ be given, and let $x^+=x-\eta m/\|m\|$. Then we have 
\begin{align*}
f(x^+) \le f(x) - \eta\|\nabla f(x)\| + 2\eta\|\nabla f(x) - m\| + \frac{L_1}{2}\eta^2,    
\end{align*}
where $L_1$ is given in Assumption \ref{asp:basic}(b).
\end{lemma}

\begin{proof}
Using \eqref{ineq:1st-desc} with $y=x^+$, we obtain that 
\begin{align*}
&f(x^+) \overset{\eqref{ineq:1st-desc}}{\le} f(x) +  \nabla f(x)^T(x^+ - x) + \frac{L_1}{2}\|x^+ - x\|^2\\
&= f(x) +  m^T(x^+ - x) + (\nabla f(x) - m)^T(x^+ - x) + \frac{L_1}{2}\|x^+ - x\|^2\\
&= f(x) -\eta \|m\| -\frac{\eta}{\|m\|}(\nabla f(x) - m)^Tm + \frac{L_1}{2}\eta^2 \le f(x) - \eta\|m\| + \eta \|\nabla f(x) - m\| + \frac{L_1}{2}\eta^2\\
&\le f(x) - \eta \|\nabla f(x)\| + 2\eta\|\nabla f(x) - m\| + \frac{L_1}{2}\eta^2 ,
\end{align*}
where the second equality follows from $x^+=x-\eta m/\|m\|$, the second inequality is due to the Cauchy-Schwarz inequality, and the last inequality follows from the triangular inequality $\|m\| \geq \|\nabla f(x)\| - \|\nabla f(x) - m\| $. Hence, this lemma holds as desired.
\end{proof}

The following lemma provides an upper bound on the residual of the $p$th-order Taylor expansion of $\nabla f$. 

\begin{lemma}\label{lem:grad-desc-near-high}
Suppose that Assumption \ref{asp:high-smooth} holds. Let $\Rp(\cdot,\cdot)$ and $L_p$ be given in \eqref{def:taylor-res} and Assumption \ref{asp:high-smooth}, respectively. Then it holds that $\|\Rp(y,x)\|\le L_p\|y-x\|^p/p!$ for all $x,y\in\mathbb{R}^n$.
\end{lemma}

\begin{proof}
Fix any $u\in\R^n$. Let $\phi(x)=\langle \nabla f(x),u\rangle$. By this and the definition of $\nabla^{r+1} f(x)(h)^{r}$, one has
\begin{align}\label{pth-der-phi-f}
\mathcal{D}^r\phi(x)[v]^r = \langle \nabla^{r+1} f(x)(v)^r, u\rangle\qquad\forall 1\le r\le p-1, v\in\R^n.    
\end{align}
Using this and \eqref{def:pnorm}, we have
\begin{align}\label{p-p-1th-der-phi-f}
\|\mathcal{D}^{p-1}\phi(y)-\mathcal{D}^{p-1}\phi(x)\| \leq \|u\|\|\mathcal{D}^p f(y)-\mathcal{D}^p f(x)\|\qquad \forall x,y\in\R^n.
\end{align}
Fix any $x,y\in\R^n$. By Taylor's expansion, one has
\begin{align*}
&\phi(y) = \phi(x) + \sum_{r=1}^{p-2}\frac{1}{r!} \mathcal{D}^r\phi(x)[y-x]^r  + \frac{1}{(p-2)!} \int_0^1(1-t)^{p-2}\mathcal{D}^{p-1} \phi(x+t(y-x))[y-x]^{p-1}\mathrm{d}t \\
& = \phi(x) + \sum_{r=1}^{p-1}\frac{1}{r!} D^r\phi(x)[y-x]^r + \frac{1}{(p-2)!} \int_0^1(1-t)^{p-2}(\mathcal{D}^{p-1} \phi(x+t(y-x)) - \mathcal{D}^{p-1} \phi(x))[y-x]^{p-1}\mathrm{d}t.
\end{align*}
Using this, \eqref{def:pnorm}, \eqref{pth-der-phi-f}, and \eqref{p-p-1th-der-phi-f}, we obtain that 
\begin{align*}
&\Big|\Big\langle \nabla f(y) - \nabla f(x) - \sum_{r=1}^{p-1}\frac{1}{r!} \nabla^{r+1}f(x)(y-x)^r, u\Big\rangle \Big| \overset{\eqref{pth-der-phi-f}}{=} \Big|\phi(y) - \phi(x) - \sum_{r=1}^{p-1}\frac{1}{r!} \mathcal{D}^r\phi(x)[y-x]^r \Big|\\
&=\Big|\frac{1}{(p-2)!} \int_0^1(1-t)^{p-2}(\mathcal{D}^{p-1} \phi(x+t(y-x)) - \mathcal{D}^{p-1} \phi(x))[y-x]^{p-1}\mathrm{d}t\Big|\\
&\overset{\eqref{def:pnorm}}{\le} \frac{1}{(p-2)!}\|y-x\|^{p-1}\int^1_0(1-t)^{p-2}\|\mathcal{D}^{p-1} \phi(x+t(y-x)) - \mathcal{D}^{p-1} \phi(x)\|\mathrm{d}t\\
&\overset{\eqref{p-p-1th-der-phi-f}}{\leq} \frac{1}{(p-2)!}\|y-x\|^{p-1}\|u\|\int^1_0(1-t)^{p-2}\|\mathcal{D}^p f(x+t(y-x)) - \mathcal{D}^pf(x)\|\mathrm{d}t\\
&{\le} \frac{1}{(p-2)!}L_p\|y-x\|^p\|u\|\int^1_0(1-t)^{p-2}t\mathrm{d}t = \frac{1}{p!}L_p\|y-x\|^p\|u\|,
\end{align*}
where the last inequality follows from Assumption \ref{asp:high-smooth}, and the last equality is due to $\int^1_0(1-t)^{p-2}t\mathrm{d}t=1/(p(p-1))$. Taking the maximum of this inequality over all $u$ with $\|u\|\le 1$, we conclude that this lemma holds.    
\end{proof}

\subsection{Proof of the main results in Section \ref{subsec:nsgd-pm}}\label{subsec:pf-pm}

In this subsection, we first establish several technical lemmas and then use them to prove Theorems~\ref{thm:s-rate-pm} and \ref{thm:un-s-rate-pm}. The following lemma presents a recurrence relation for the estimation error of the gradient estimators $\{m^k\}$ generated by Algorithm \ref{alg:unf-sfom-pm}.

\begin{lemma}\label{lem:rec-pm}
Suppose that Assumption \ref{asp:basic} holds. Let $\{(x^k,m^k)\}$ be the sequence generated by Algorithm \ref{alg:unf-sfom-pm} with input parameters $\{(\eta_k,\theta_k)\}$. Then we have
\begin{align}
&\mathbb{E}_{\xi^{k+1}}[\|m^{k+1}-\nabla f(x^{k+1})\|^\alpha] \le (1-\theta_k)\|m^k-\nabla f(x^k)\|^\alpha + 3L_1^\alpha\eta_k^\alpha \theta_k^{1-\alpha} + 2\sigma^\alpha\theta_k^\alpha \qquad \forall k\ge0,\label{ineq:vr-pm}
\end{align}
where $L_1$, $\sigma$, and $\alpha$ are given in Assumption \ref{asp:basic}.
\end{lemma}

\begin{proof}
Fix any $k\ge0$. It follows from \eqref{update-mk-pm} that
\begin{align}
&m^{k+1} - \nabla f(x^{k+1}) \overset{\eqref{update-mk-pm}}{=} (1- \theta_k)m^k + \theta_k G(x^{k+1};\xi^{k+1}) - \nabla f(x^{k+1}) \nonumber\\
&= (1- \theta_k) (m^k-\nabla f(x^k)) + (1-\theta_k)(\nabla f(x^k)-\nabla f(x^{k+1})) + \theta_k (G(x^{k+1};\xi^{k+1}) - \nabla f(x^{k+1})). \label{rela-mk+1-pm}
\end{align}
Observe from Algorithm \ref{alg:unf-sfom-pm} and Assumption \ref{asp:basic} that $\|x^{k+1}-x^k\|=\eta_k$, $\E_{\xi^{k+1}}[G(x^{k+1};\xi^{k+1})-\nabla f(x^{k+1})]=0$, $\E_{\xi^{k+1}}[\|G(x^{k+1};\xi^{k+1})-\nabla f(x^{k+1})\|^\alpha]\le\sigma^\alpha$, and $\|\nabla f(x^k)-\nabla f(x^{k+1})\|\leq L_1 \eta_k$.
Using these, \eqref{open-alpha}, \eqref{open-alpha-2}, and \eqref{rela-mk+1-pm}, we obtain that for all $c>0$,
\begin{align}
&\mathbb{E}_{\xi^{k+1}}[\|m^{k+1}-\nabla f(x^{k+1})\|^\alpha]\nonumber \\
&\overset{\eqref{rela-mk+1-pm}}{=}\mathbb{E}_{\xi^{k+1}}[\|(1- \theta_k) (m^k-\nabla f(x^k)) + (1-\theta_k)(\nabla f(x^k)-\nabla f(x^{k+1})) + \theta_k (G(x^{k+1};\xi^{k+1}) - \nabla f(x^{k+1}))\|^\alpha]\nonumber\\
&\overset{ \eqref{open-alpha}}{\le} \|(1- \theta_k) (m^k-\nabla f(x^k))+ (1-\theta_k)(\nabla f(x^k)-\nabla f(x^{k+1}))\|^\alpha \nonumber\\
&\qquad + 2\mathbb{E}_{\xi^{k+1}}[\|\theta_k (G(x^{k+1};\xi^{k+1}) - \nabla f(x^{k+1}))\|^\alpha]\nonumber\\
&\overset{\eqref{open-alpha-2}}{\le} (1+c)(1-\theta_k)^\alpha\|m^k-\nabla f(x^k)\|^\alpha  + (2+(\alpha-1)^{\alpha-1}c^{1-\alpha})(1-\theta_k)^\alpha \|\nabla f(x^k)-\nabla f(x^{k+1})\|^\alpha + 2\sigma^\alpha\theta_k^\alpha\nonumber\\
&\le (1+c)(1-\theta_k)^\alpha\|m^k-\nabla f(x^k)\|^\alpha + L_1^\alpha(2+(\alpha-1)^{\alpha-1}c^{1-\alpha})(1-\theta_k)^\alpha\eta_k^\alpha + 2\sigma^\alpha\theta_k^\alpha,\label{vr-inter}
\end{align}
where the first inequality is due to \eqref{open-alpha} and $\E_{\xi^{k+1}}[G(x^{k+1};\xi^{k+1})-\nabla f(x^{k+1})]=0$, the second inequality is due to \eqref{open-alpha-2} and $\E_{\xi^{k+1}}[\|G(x^{k+1};\xi^{k+1})-\nabla f(x^{k+1})\|^\alpha]\le\sigma^\alpha$, and the last inequality follows from  $\|x^{k+1}-x^k\|=\eta_k$ and $\|\nabla f(x^k)-\nabla f(x^{k+1})\|\leq L_1 \eta_k$.

When $\theta_k = 1$, \eqref{ineq:vr-pm} clearly holds. For $\theta_k \in (0, 1)$, letting $c = (1 - \theta_k)^{1 - \alpha} - 1$ in \eqref{vr-inter}, and using the fact that $\alpha \in (1, 2]$, we have
\begin{align*}
c^{1-\alpha} & = ((1-\theta_k)^{1-\alpha} - 1)^{1-\alpha} = \Big(\frac{1}{(1-\theta_k)^{\alpha-1}} - 1\Big)^{1-\alpha} \le \Big(\frac{1}{1-(\alpha-1)\theta_k} - 1\Big)^{1-\alpha}\\
& =\Big(\frac{1-(\alpha-1)\theta_k}{(\alpha-1)\theta_k} \Big)^{\alpha-1} \le ((\alpha-1)\theta_k)^{1-\alpha},
\end{align*}
where the first inequality follows from $(1-\tau)^\beta\le 1-\beta\tau$ for all $\tau\in(-\infty,1)$ and $\beta\in[0,1]$. Combining this inequality with \eqref{vr-inter}, one has 
\begin{align*}
\mathbb{E}_{\xi^{k+1}}[\|m^{k+1}-\nabla f(x^{k+1})\|^\alpha]\le (1-\theta_k)\|m^k-\nabla f(x^k)\|^\alpha + L_1^\alpha(2+\theta_k^{1-\alpha})(1-\theta_k)^\alpha\eta_k^\alpha + 2\sigma^\alpha\theta_k^\alpha,
\end{align*}
which together with $\theta_k\in(0,1]$ and $\alpha\in(1,2]$ implies that \eqref{ineq:vr-pm} holds.
\end{proof}

The following lemma establishes a descent property for the potential sequence $\{\mathcal{P}_k\}$ defined below.

\begin{lemma}\label{lem:rate-pm}
Suppose that Assumption \ref{asp:basic} holds. Let $\{(x^k,m^k)\}$ be the sequence generated by Algorithm \ref{alg:unf-sfom-pm} with input parameters $\{(\eta_k,\theta_k)\}$, and $L_1$, $\sigma$, and $\alpha$ be given in Assumption \ref{asp:basic}, and let $\{{\mathcal P}_k\}$ be defined in \eqref{def:pot-pm} for $\{(x^k,m^k)\}$ and any nonincreasing positive sequence $\{p_k\}$. Then it holds that for all $k\ge0$, 
\begin{align}\label{stat-bd-pm}
\mathbb{E}_{\xi^{k+1}}[{\mathcal P}_{k+1}] \le {\mathcal P}_k- \eta_k\|\nabla f(x^k)\| + \frac{L_1}{2}\eta_k^2
+ \frac{(\alpha-1)(2\eta_k/\alpha)^{\alpha/(\alpha-1)}}{(\theta_kp_k)^{1/(\alpha-1)}}  
+  3L_1^\alpha \theta_k^{1-\alpha}\eta_k^\alpha p_k + 2\sigma^\alpha\theta_k^\alpha p_k.
\end{align}
\end{lemma}

\begin{proof}
Fix any $k\ge0$. By Lemma \ref{lem:ppt-desc} with $(x^+,x,m,\eta)=(x^{k+1},x^k,m^k,\eta_k)$, one has
\begin{align}\label{upbd-fxk+1-pm}
f(x^{k+1}) \le f(x^k) - \eta_k\|\nabla f(x^k)\| + 2\eta_k\|\nabla f(x^k) - m^k\| + \frac{L_1}{2}\eta_k^2.    
\end{align}
Combining this with \eqref{def:pot-pm} and \eqref{ineq:vr-pm}, we obtain that
\begin{align}
&\mathbb{E}_{\xi^{k+1}}[{\mathcal P}_{k+1}] \overset{\eqref{def:pot-pm}}{=} \mathbb{E}_{\xi^{k+1}}[f(x^{k+1}) + p_{k+1}\|m^{k+1} - \nabla f(x^{k+1})\|^\alpha]\nonumber\\
&\overset{\eqref{ineq:vr-pm}\eqref{upbd-fxk+1-pm}}{\le} f(x^k) - \eta_k\|\nabla f(x^k)\| + 2\eta_k\|\nabla f(x^k) - m^k\| + \frac{L_1}{2}\eta_k^2\nonumber \\
&\qquad + (1-\theta_k)p_{k+1}\|m^k-\nabla f(x^k)\|^\alpha + 3L_1^\alpha \theta_k^{1-\alpha}\eta_k^\alpha p_{k+1} + 2\sigma^\alpha\theta_k^\alpha p_{k+1}\nonumber\\
&\le f(x^k) - \eta_k\|\nabla f(x^k)\| + 2\eta_k\|\nabla f(x^k) - m^k\| + \frac{L_1}{2}\eta_k^2\nonumber\\
&\quad + (1-\theta_k)p_k\|m^k-\nabla f(x^k)\|^\alpha + 3L_1^\alpha \theta_k^{1-\alpha}\eta_k^\alpha p_k + 2\sigma^\alpha\theta_k^\alpha p_k,\label{pot-desc-ineq-pm}
\end{align}
where the last inequality follows from the fact that $\{p_k\}$ is nonincreasing. In addition, letting $\alpha^\prime=\alpha/(\alpha-1)$ and using the Young's inequality, we have
\begin{align*}
2\eta_k\|\nabla f(x^k) - m^k\| & \le \frac{\big((\alpha\theta_k p_k)^{1/\alpha} \|\nabla f(x^k) - m^k\|\big)^\alpha}{\alpha} + \frac{\big(2\eta_k/(\alpha\theta_k p_k)^{1/\alpha}\big)^{\alpha^\prime}}{\alpha^\prime}\\
&= \theta_k p_k \|\nabla f(x^k) - m^k\|^\alpha + \frac{(\alpha-1)(2\eta_k)^{\alpha/(\alpha-1)}}{\alpha^{\alpha/(\alpha-1)}(\theta_kp_k)^{1/(\alpha-1)}}.
\end{align*}
This together with \eqref{pot-desc-ineq-pm} implies that
\begin{align*}
&\mathbb{E}_{\xi^{k+1}}[{\mathcal P}_{k+1}]  \le f(x^k) + p_k\|m^k-\nabla f(x^k)\|^\alpha - \eta_k\|\nabla f(x^k)\| + \frac{L_1}{2}\eta_k^2 \\
&\qquad\qquad\qquad + \frac{(\alpha-1)(2\eta_k)^{\alpha/(\alpha-1)}}{\alpha^{\alpha/(\alpha-1)}(\theta_kp_k)^{1/(\alpha-1)}} + 3L_1^\alpha \theta_k^{1-\alpha}\eta_k^\alpha p_k + 2\sigma^\alpha\theta_k^\alpha p_k.
\end{align*}
The conclusion  \eqref{stat-bd-pm} then follows from this and  \eqref{def:pot-pm}.
\end{proof}

We are now ready to prove Theorem \ref{thm:s-rate-pm}.

\begin{proof}[\textbf{Proof of Theorem \ref{thm:s-rate-pm}}]
Let $\{(x^k,m^k)\}$ be generated by Algorithm \ref{alg:unf-sfom-pm} with $\{(\eta_k,\theta_k)\}$ given in \eqref{pm-eta-theta}, and 
$\{{\mathcal P}_k\}$ be defined in \eqref{def:pot-pm} with such $\{(x^k,m^k)\}$ and the following $\{p_k\}$:
\begin{align}\label{def:pk-pm}
p_k= (k+1)^{(\alpha^2-3\alpha+2)/(3\alpha-2)} \qquad\forall k\ge0.  
\end{align}
Since $\alpha\in(1,2]$, one can see that $\{p_k\}$ is nonincreasing. Also, observe from \eqref{pm-eta-theta} that $\{\eta_k\}\subset(0,+\infty)$ and $\{\theta_k\}\subset(0,1]$. Hence, $\{(\eta_k,\theta_k,p_k)\}$ satisfies the assumptions in Lemma \ref{lem:rate-pm} and Algorithm \ref{alg:unf-sfom-pm}. In addition, by \eqref{def:pot-pm} and \eqref{def:pk-pm}, one has that
\begin{align}
&\mathbb{E}[{\mathcal P}_0]= f(x^0)+p_0\mathbb{E}[\|m^0-\nabla f(x^0)\|^\alpha]= f(x^0) + \mathbb{E}[\|G(x^0;\xi^0)-\nabla f(x^0)\|^\alpha]\le f(x^0)+\sigma^\alpha,\label{upbd-exp-P0-pm}\\    
&\mathbb{E}[{\mathcal P}_K]= \mathbb{E}[f(x^K)+p_K\|m^K-\nabla f(x^K)\|^\alpha]\ge \mathbb{E}[f(x^K)] \ge f_{\mathrm{low}}.\label{lwbd-exp-Pk-pm}
\end{align}
Taking expectation on both sides of \eqref{stat-bd-pm} with respect to $\{\xi^i\}_{i=0}^{k+1}$, we have
\begin{align*}
\mathbb{E}[{\mathcal P}_{k+1}] \le \mathbb{E}[{\mathcal P}_k]- \eta_k\mathbb{E}[\|\nabla f(x^k)\|] + \frac{L_1}{2}\eta_k^2
+ \frac{(\alpha-1)(2\eta_k/\alpha)^{\alpha/(\alpha-1)}}{(\theta_kp_k)^{1/(\alpha-1)}}  
+  3L_1^\alpha \theta_k^{1-\alpha}\eta_k^\alpha p_k + 2\sigma^\alpha\theta_k^\alpha p_k\quad\forall k\ge0.
\end{align*}
Summing up this inequality over $k=0,\ldots,K-1$, and using \eqref{upbd-exp-P0-pm} and \eqref{lwbd-exp-Pk-pm}, we obtain that for all $K\ge1$,
\begin{align}
&f_{\mathrm{low}}\overset{\eqref{lwbd-exp-Pk-pm}}{\le} \mathbb{E}[{\mathcal P}_K]\nonumber \\
&\le \E[{\mathcal P}_0] - \sum_{k=0}^{K-1}\eta_k\mathbb{E}[\|\nabla f(x^k)\|] + \sum_{k=0}^{K-1}\Big(\frac{L_1}{2}\eta_k^2 + \frac{(\alpha-1)(2\eta_k)^{\alpha/(\alpha-1)}}{\alpha^{\alpha/(\alpha-1)}(\theta_kp_k)^{1/(\alpha-1)}} + 3L_1^\alpha \theta_k^{1-\alpha}\eta_k^\alpha p_k + 2\sigma^\alpha\theta_k^\alpha p_k\Big)\nonumber\\
&\overset{\eqref{upbd-exp-P0-pm}}{\le} f(x^0) + \sigma^\alpha - \eta_{K-1}\sum_{k=0}^{K-1}\mathbb{E}[\|\nabla f(x^k)\|] \nonumber\\
&\qquad + \sum_{k=0}^{K-1}\Big(\frac{L_1}{2}\eta_k^2 + \frac{(\alpha-1)(2\eta_k)^{\alpha/(\alpha-1)}}{\alpha^{\alpha/(\alpha-1)}(\theta_kp_k)^{1/(\alpha-1)}} + 3L_1^\alpha \theta_k^{1-\alpha}\eta_k^\alpha p_k + 2\sigma^\alpha\theta_k^\alpha p_k\Big),\label{rear-pot-desc-pm}
\end{align}
where the last inequality follows from \eqref{upbd-exp-P0-pm} and the fact that $\{\eta_k\}$ is nonincreasing. Rearranging the terms in 
\eqref{rear-pot-desc-pm}, and using \eqref{M1a}, \eqref{pm-eta-theta}, and \eqref{def:pk-pm}, we obtain that for all $K\ge3$, 
\begin{align*}
&\frac{1}{K}\sum_{k=0}^{K-1}\mathbb{E}[\|\nabla f(x^k)\|]\\
&\overset{\eqref{rear-pot-desc-pm}}{\le}\frac{f(x^0) - f_{\mathrm{low}} + \sigma^\alpha}{K\eta_{K-1}} + \frac{1}{K\eta_{K-1}}\sum_{k=0}^{K-1}\Big(\frac{L_1}{2}\eta_k^2+\frac{(\alpha-1)(2\eta_k)^{\alpha/(\alpha-1)}}{\alpha^{\alpha/(\alpha-1)}(\theta_kp_k)^{1/(\alpha-1)}} + 3L_1^\alpha \theta_k^{1-\alpha}\eta_k^\alpha p_k + 2\sigma^\alpha\theta_k^\alpha p_k\Big)  \\
&\overset{\eqref{pm-eta-theta}\eqref{def:pk-pm}}{=} \frac{f(x^0) - f_{\mathrm{low}} + \sigma^\alpha}{K^{(\alpha-1)/(3\alpha-2)}} \\
&\qquad +\frac{1}{K^{(\alpha-1)/(3\alpha-2)}}\sum_{k=0}^{K-1}\Big(\frac{L_1}{2(k+1)^{2(2\alpha-1)/(3\alpha-2)}} + \frac{(\alpha-1)(2/\alpha)^{\alpha/(\alpha-1)} + 3L_1^\alpha+2\sigma^\alpha}{k+1}\Big) \\
% &\le \frac{f(x^0) - f_{\mathrm{low}} + \sigma^\alpha + 2L_1}{K^{(\alpha-1)/(3\alpha-2)}} + \frac{((\alpha-1)(2/\alpha)^{\alpha/(\alpha-1)} + 3L_1^\alpha+2\sigma^\alpha)\ln(2K+1)}{K^{(\alpha-1)/(3\alpha-2)}}\\
&\le\frac{2(f(x^0) - f_{\mathrm{low}} + \sigma^\alpha + L_1 + (\alpha-1)(2/\alpha)^{\alpha/(\alpha-1)} + 3L_1^\alpha+2\sigma^\alpha)\ln K}{K^{(\alpha-1)/(3\alpha-2)}}\overset{\eqref{M1a}}{=}\frac{M_{1,\alpha}\ln K}{K^{(\alpha-1)/(3\alpha-2)}},
\end{align*}
% where the third inequality follows from $\sum_{k=0}^{K-1}1/(k+1)\le \ln(2K+1)$ and $\sum_{k=0}^{K-1}1/(k+1)^{2(2\alpha-1)/(3\alpha-2)}\le(3\alpha-2) 2^{\alpha/(3\alpha-2)}/\alpha < 4$
% due to \eqref{upbd:series-ka} with $(a,b)=(1,K)$, $\beta=1$ or $2(2\alpha-1)/(3\alpha-2)$, and  $(3\alpha-2)/\alpha \in (1,2]$, and the last inequality is due to $\ln(2K+1)\le 2\ln K$ for all $K\ge 3$. 
where the second inequality follows from $\sum_{k=0}^{K-1}1/(k+1)\le 2\ln K$ due to \eqref{upbd:series-ka} and $K\ge 3$, and $\sum_{k=0}^{K-1}1/(k+1)^{2(2\alpha-1)/(3\alpha-2)}\le(3\alpha-2) 2^{\alpha/(3\alpha-2)}/\alpha < 4$ due to \eqref{upbd:series-ka} and $(3\alpha-2)/\alpha \in (1,2]$. Recall that $\iota_K$ is uniformly selected from $\{0,\ldots,K-1\}$. It then follows from this and the above inequality that
\begin{align}\label{pre-upbd-pm}
\E[\|\nabla f(x^{\iota_K})\|] = \frac{1}{K}\sum_{k=0}^{K-1}\mathbb{E}[\|\nabla f(x^k)\|] \le \frac{M_{1,\alpha}\ln K}{K^{(\alpha-1)/(3\alpha-2)}}\qquad\forall K\ge 3.
\end{align}
In addition, by Lemma \ref{lem:rate-complexity} with $(\beta,u,v)=((\alpha-1)/(3\alpha-2),(\alpha-1)\epsilon/(2(3\alpha-2)M_{1,\alpha}),K)$, one can see that 
\begin{align*}
K^{-(\alpha-1)/(3\alpha-2)}\ln K\le \frac{\epsilon}{M_{1,\alpha}}\qquad\forall K\ge\Big(\frac{2(3\alpha-2)M_{1,\alpha}}{(\alpha-1)\epsilon}\ln\Big(\frac{2(3\alpha-2)M_{1,\alpha}}{(\alpha-1)\epsilon}\Big)\Big)^{(3\alpha-2)/(\alpha-1)},    
\end{align*}
which together with \eqref{pre-upbd-pm} implies that Theorem \ref{thm:s-rate-pm} holds.
\end{proof}

We next prove Theorem \ref{thm:un-s-rate-pm}.

\begin{proof}[\textbf{Proof of Theorem \ref{thm:un-s-rate-pm}}]
Let $\{(x^k,m^k)\}$ be generated by Algorithm \ref{alg:unf-sfom-pm} with $\{(\eta_k,\theta_k)\}$ given in \eqref{pm-eta-theta-o}, and 
$\{{\mathcal P}_k\}$ be defined in \eqref{def:pot-pm} with such $\{(x^k,m^k)\}$ and the following $\{p_k\}$:
\begin{align}\label{pk-pm-unk}
p_k= (k+1)^{(2\alpha^2-5\alpha+2)/(4\alpha)}\qquad\forall k\ge0.    
\end{align}
Since $\alpha\in(1,2]$, one can see that $\{p_k\}$ is nonincreasing. Also, observe from \eqref{pm-eta-theta-o} that $\{\eta_k\}\subset(0,+\infty)$ and $\{\theta_k\}\subset(0,1]$. Hence, $\{(\eta_k,\theta_k,p_k)\}$ defined in \eqref{pm-eta-theta-o} and \eqref{pk-pm-unk} satisfies the assumptions in Lemma \ref{lem:rate-pm} and Algorithm \ref{alg:unf-sfom-pm}. By \eqref{pm-eta-theta-o}, \eqref{pk-pm-unk}, and similar arguments as those for deriving \eqref{rear-pot-desc-pm}, one has that for all $K\ge1$, 
\begin{align*}
f_{\mathrm{low}} & \le f(x^0) + \sigma^\alpha - \eta_{K-1}\sum_{k=0}^{K-1}\mathbb{E}[\|\nabla f(x^k)\|] \\
&\qquad + \sum_{k=0}^{K-1}\Big(\frac{L_1}{2}\eta_k^2 + \frac{(\alpha-1)(2\eta_k)^{\alpha/(\alpha-1)}}{\alpha^{\alpha/(\alpha-1)}(\theta_kp_k)^{1/(\alpha-1)}} + 3L_1^\alpha \theta_k^{1-\alpha}\eta_k^\alpha p_k + 2\sigma^\alpha\theta_k^\alpha p_k\Big).
\end{align*}
Rearranging the terms of this inequality, and using \eqref{M1a-o}, \eqref{pm-eta-theta-o}, and \eqref{pk-pm-unk}, we obtain that for all $K\ge3$,
\begin{align*}
&\frac{1}{K}\sum_{k=0}^{K-1}\mathbb{E}[\|\nabla f(x^k)\|]\\
&\le \frac{f(x^0) - f_{\mathrm{low}} + \sigma^\alpha}{K\eta_{K-1}} + \frac{1}{K\eta_{K-1}}\sum_{k=0}^{K-1}\Big(\frac{L_1}{2}\eta_k^2+\frac{(\alpha-1)(2\eta_k)^{\alpha/(\alpha-1)}}{\alpha^{\alpha/(\alpha-1)}(\theta_kp_k)^{1/(\alpha-1)}} + 3L_1^\alpha \theta_k^{1-\alpha}\eta_k^\alpha p_k + 2\sigma^\alpha\theta_k^\alpha p_k\Big)\\
&\overset{\eqref{pm-eta-theta-o}\eqref{pk-pm-unk}}{=}\frac{f(x^0) - f_{\mathrm{low}} + \sigma^\alpha}{K^{1/4}} \\
& \qquad + \frac{1}{K^{1/4}}\sum_{k=0}^{K-1}\Big(\frac{L_1}{2(k+1)^{3/2}} + \frac{(\alpha-1)(2/\alpha)^{\alpha/(\alpha-1)} + 2\sigma^\alpha}{(k+1)^{(5\alpha-2)/(4\alpha)}}  + \frac{3L_1^\alpha}{(k+1)^{(7\alpha - \alpha^2 - 2)/(4\alpha)}}\Big)\\
&\le\frac{f(x^0) - f_{\mathrm{low}} + \sigma^\alpha}{K^{1/4}} + \frac{1}{K^{1/4}}\sum_{k=0}^{K-1}\Big(\frac{L_1/2 + 3L_1^\alpha + ((\alpha-1)(2/\alpha)^{\alpha/(\alpha-1)} + 2\sigma^\alpha)K^{(2-\alpha)/(4\alpha)}}{k+1} \Big) \\
% &\le \frac{f(x^0) - f_{\mathrm{low}} + \sigma^\alpha + (L_1/2 + 3L_1^\alpha + ((\alpha-1)(2/\alpha)^{\alpha/(\alpha-1)} + 2\sigma^\alpha)K^{(2-\alpha)/(4\alpha)})\ln(2K+1)}{K^{1/4}}\\
&\le \frac{2(f(x^0) - f_{\mathrm{low}} + \sigma^\alpha + L_1/2 + 3L_1^\alpha)\ln K}{K^{1/4}} + \frac{2((\alpha-1)(2/\alpha)^{\alpha/(\alpha-1)} + 2\sigma^\alpha)\ln K}{K^{(\alpha-1)/(2\alpha)}} \\
&\overset{\eqref{M1a-o}}{=}\frac{\widetilde{M}_{1,\alpha}\ln K}{K^{1/4}} + \frac{\widehat{M}_{1,\alpha}\ln K}{K^{(\alpha-1)/(2\alpha)}},
\end{align*}
where the second inequality follows from $(5\alpha-2)/(4\alpha)\le1$ and $(7\alpha - \alpha^2 - 2)/(4\alpha)\ge1$ due to $\alpha\in(1,2]$, and the last inequality follows from $\sum_{k=0}^{K-1}1/(k+1)\le 2\ln K$ due to \eqref{upbd:series-ka} and $K\ge 3$. Recall that $\iota_K$ is uniformly selected from $\{0,\ldots,K-1\}$. It follows from this and the above inequality that
\begin{align}\label{ave-bd-unknow-pm}
\E[\|\nabla f(x^{\iota_K})\|] = \frac{1}{K}\sum_{k=0}^{K-1}\mathbb{E}[\|\nabla f(x^k)\|] \le \frac{\widetilde{M}_{1,\alpha}\ln K}{K^{1/4}} + \frac{\widehat{M}_{1,\alpha}\ln K}{K^{(\alpha-1)/(2\alpha)}}\qquad\forall K\ge 3.
\end{align}
In addition, by Lemma \ref{lem:rate-complexity} with $(\beta,u,v)=(1/4,\epsilon/(16\widetilde{M}_{1,\alpha}),K)$ and $(\beta,u,v)=((\alpha-1)/(2\alpha),(\alpha-1)\epsilon/(8\alpha\widehat{M}_{1,\alpha}),K)$, one can see that
\begin{align*}
&K^{-1/4}\ln K\le \frac{\epsilon}{2\widetilde{M}_{1,\alpha}}\qquad\forall K\ge\Big(\frac{16\widetilde{M}_{1,\alpha}}{\epsilon}\ln\Big(\frac{16\widetilde{M}_{1,\alpha}}{\epsilon}\Big)\Big)^{4},\\
&K^{-(\alpha-1)/(2\alpha)}\ln K\le \frac{\epsilon}{2\widehat{M}_{1,\alpha}}\qquad\forall K\ge\Big(\frac{8\alpha\widehat{M}_{1,\alpha}}{(\alpha-1)\epsilon}\ln\Big(\frac{8\alpha\widehat{M}_{1,\alpha}}{(\alpha-1)\epsilon}\Big)\Big)^{2\alpha/(\alpha-1)},
\end{align*}
which together with \eqref{ave-bd-unknow-pm} imply that Theorem \ref{thm:un-s-rate-pm} holds.
\end{proof}

\subsection{Proof of the main results in Section \ref{subsec:nsgd-mem}}\label{subsec:pf-mem}

In this subsection, we first establish several technical lemmas and then use them to prove Theorems \ref{thm:known-rate-mem} and \ref{thm:unknown-rate-mem}.  

Before proceeding, we present the well-known Weierstrass product inequalities (see, e.g., \cite{klamkin1970extensions}). For any given $\{a_t\}_{t=1}^m\subset(0,1)$, it holds that 
\begin{align}\label{W-ineq}
1 - \sum_{t=1}^m a_t  \le \prod_{t=1}^m(1-a_t) \le \frac{1}{1 + \sum_{t=1}^m a_t}.
\end{align}

We next present an auxiliary lemma that will be used subsequently.

\begin{lemma}\label{lem:inf-prod-tool}
$\prod_{s\ge1,s\neq t}(1- t^2/s^2)=(-1)^{t-1}/2$ holds for any positive integer $t$.
\end{lemma}

\begin{proof}
Fix any positive integer $t$. Let $\phi(a)=\prod_{s\ge1,s\neq t}(1- a^2/s^2)$ for any $a\geq 0$. Observe that the sequence $\{u_r\}$ is decreasing and $u_r\geq 0$ when $r\ge a+1$, where $u_r=\prod_{a+1\le s\le r,s\neq t}(1-a^2/s^2)$. This implies that $\phi(a)$ is well-defined for all $a\geq 0$. In addition, it is well-known that the normalized sinc function 
$\sin(\pi y)/(\pi y)$ can be represented as the infinite product $\prod_{s=1}^\infty(1-y^2/s^2)$ for all $y\in\R$. By this, one can see that
\begin{align*}%\label{func-express}
\phi(a) = \frac{\prod_{s=1}^{\infty}(1-a^2/s^2)}{1-a^2/t^2} = \frac{\sin(\pi a)}{\pi a(1-a^2/t^2)}\qquad\forall a\neq t.  
\end{align*}
It then follows that
\begin{align}\label{lpt-rule}
\lim_{a\to t}\phi(a) = \lim_{a\to t}\frac{\sin(\pi a)}{\pi a(1-a^2/t^2)} = \lim_{a\to t}\frac{\cos(\pi a)}{ 1 - 3a^2/t^2} = \frac{(-1)^{t-1}}{2},
\end{align}
where the second equality is due to L'H\^opital's rule. Thus, to prove this lemma, it suffices to show that $\phi(a)$ is continuous at $a=t$. To this end, we define a sequence of functions $\{\phi_r\}$ as follows: 
\begin{align*}
\phi_r(a)=\prod_{1\le s\le r, s\neq t}(1- a^2/s^2)\qquad\forall a \geq 0 
\end{align*}
for each $r\ge1$.  We next show that $\phi_r$ converges uniformly to $\phi$ on $[0,2t]$. Indeed, let $R \in [4t, \infty)$ be arbitrarily chosen, and fix any $r_1, r_2$ satisfying $r_2>r_1>R$. Observe that
\begin{align}\label{cauchy-seq}
|\phi_{r_1}(a) - \phi_{r_2}(a)| = \Big(1 - \prod_{r_1+1\le s\le r_2} (1-a^2/s^2)\Big)\Big|\prod_{1\le s\le r_1, s\neq t} (1-a^2/s^2)\Big|\qquad\forall a\in[0,2t].   
\end{align}
In addition, one has
\begin{align*}
\prod_{r_1+1\le s\le r_2} (1-a^2/s^2)\overset{\eqref{W-ineq}}{\ge} 1 - \sum_{s = r_1+1}^\infty (a^2/s^2) \ge 1 -  \sum_{s = r_1+1}^\infty \frac{a^2}{s(s-1)} = 1 - \frac{a^2}{r_1} \ge 1 - \frac{a^2}{R}\qquad\forall a\in[0,2t],
\end{align*}
which, together with \eqref{cauchy-seq} and $r_2>r_1>R$, implies that\footnotemark
\begin{align*}
&|\phi_{r_1}(a) - \phi_{r_2}(a)| \le \frac{a^2|\prod_{1\le s\le r_1, s\neq t} (1-a^2/s^2)|}{R} \\
&\le  \frac{a^2\prod_{1\le s\le\lceil a\rceil-1, s\neq t} (a^2/s^2-1)}{R} \le \frac{4t^2\prod_{1\le s\le 2t-1}(4t^2/s^2 - 1)}{R} \qquad\forall a\in[0,2t],%\footnotemark
\end{align*}
\footnotetext{$\prod_{1\le s\le\lceil a\rceil-1, s\neq t} (a^2/s^2-1)$ is set to $1$ if $a\in[0,1]$.}

\noindent where the second inequality follows from $(1-a^2/s^2)\in[0,1]$ for all $s\ge\lceil a\rceil$, and the last inequality is due to $a\le 2t$. By this and the choice of $r_1$, $r_2$, and $R$, one can conclude that $\{\phi_r\}$ converges uniformly to $\phi$ on $[0,2t]$, which together with the continuity of $\phi_r$ for each $r\geq 1$ implies that $\phi$ is continuous on $[0,2t]$. Hence, one has $\phi(t)=
\lim_{a\to t}\phi(a)$, which along with \eqref{lpt-rule} and the definition of $\phi$ implies that this lemma holds.
\end{proof}

The following lemma provides a set of choices for ${(\gamma_{k,t}, \theta_{k,t})}$ that satisfy \eqref{pth-lss} and \eqref{sum-tht-bd}.

\begin{lemma}\label{lem:ppt-thetakt}
Let $\{\gamma_k\}\subset(0,1/2]$ and a positive integer $q$ be given, and 
\begin{align}\label{unk-theatkt}
\gamma_{k,t} = \gamma_k/t^2, \quad \theta_{k,t} = \frac{\prod_{1\le s\le q,s\neq t}(1-s^2/\gamma_k)}{{(t^2/\gamma_k)}\prod_{1\le s\le q,s\neq t}((t^2 - s^2)/\gamma_k)}\qquad\forall 1\le t\le q,k\ge0.
\end{align}
Then $\{(\gamma_{k,t},\theta_{k,t})\}$ satisfies \eqref{pth-lss}.  
Moreover, it holds that
\begin{align}
&\sum_{t=1}^{q}\theta_{k,t}\in\Big(\frac{\gamma_k}{1+\pi^2/6},2\gamma_k\Big)\subset(0,1),\quad|\theta_{k,t}|\le \frac{4\gamma_k}{t^2}\qquad\forall 1\le t\le q,k\ge0.\label{two-ipt-ppt-theta-2}
\end{align}
\end{lemma}

\begin{proof}
Fix any $k\ge0$. We first prove that $\{(\gamma_{k,t},\theta_{k,t})\}$ satisfies \eqref{pth-lss}. For convenience, we denote the coefficient matrix in \eqref{pth-lss} as
\begin{align*}
\Gamma = \begin{bmatrix}
1/\gamma_{k,1} & 1/\gamma_{k,2} & \cdots & 1/\gamma_{k,q}\\ 
1/\gamma_{k,1}^2 & 1/\gamma_{k,2}^2 & \cdots & 1/\gamma_{k,q}^2\\ 
\vdots&\vdots&\vdots & \vdots\\
1/\gamma_{k,1}^q & 1/\gamma_{k,2}^q & \cdots & 1/\gamma_{k,q}^q
\end{bmatrix}\in\R^{q\times q}.
\end{align*} 
In addition, we define a matrix $V\in\R^{q\times q}$, whose $t$-th row $[v_{t1}\ \cdots\ v_{tq}]$ consists of the coefficients of the polynomial
\begin{align*}
h_t(\alpha) = \frac{\alpha\prod_{1\le s\le q, s\neq t}(\alpha-1/\gamma_{k,s})}{(1/\gamma_{k,t})\prod_{1\le s\le q, s\neq t}(1/\gamma_{k,t}-1/\gamma_{k,s})}=v_{t1}\alpha + v_{t2}\alpha^2 + \cdots + v_{tq}\alpha^q\qquad \forall 1\le t\le q,    
\end{align*}
which satisfies $h_t(1/\gamma_{k,t})=1$ and $h_t(1/\gamma_{k,s})=0$ for all $s\neq t$. By the definitions of $h_t$, $\Gamma$, and $V$, one has 
\begin{align*}
V\Gamma & = \begin{bmatrix}
\sum_{s=1}^q v_{1s}/\gamma_{k,1}^s& \sum_{s=1}^q v_{1s}/\gamma_{k,2}^s&\cdots& \sum_{s=1}^q v_{1s}/\gamma_{k,q}^s\\
\sum_{s=1}^q v_{2s}/\gamma_{k,1}^s& \sum_{s=1}^q v_{2s}/\gamma_{k,2}^s&\cdots& \sum_{s=1}^q v_{2s}/\gamma_{k,q}^s\\
\vdots&\vdots&\vdots&\vdots\\
\sum_{s=1}^q v_{qs}/\gamma_{k,1}^s& \sum_{s=1}^q v_{qs}/\gamma_{k,2}^s&\cdots& \sum_{s=1}^q v_{qs}/\gamma_{k,q}^s
\end{bmatrix}\\
& =     \begin{bmatrix}
h_1(1/\gamma_{k,1})& h_1(1/\gamma_{k,2}) &\cdots& h_1(1/\gamma_{k,q})\\
h_2(1/\gamma_{k,1})& h_2(1/\gamma_{k,2}) &\cdots& h_2(1/\gamma_{k,q})\\
\vdots&\vdots&\vdots&\vdots\\
h_q(1/\gamma_{k,1})& h_q(1/\gamma_{k,2}) &\cdots& h_q(1/\gamma_{k,q})
\end{bmatrix} = I,
\end{align*}
where $I$ is the $q\times q$ identity matrix. Hence, we have  $V=\Gamma^{-1}$. In view of this and the definition of $V$, one can see that the solution to \eqref{pth-lss} is unique and can be written as
\begin{align}\label{pf-exp-thetakt-ht}
\begin{bmatrix}
\theta_{k,1}\\
\theta_{k,2}\\
\vdots\\
\theta_{k,q}
\end{bmatrix} = V\begin{bmatrix}
1\\
1\\
\vdots\\
1
\end{bmatrix}  = \begin{bmatrix}
h_1(1)\\
h_2(1)\\
\vdots\\
h_q(1)
\end{bmatrix},
\end{align}
which together with the definition of $h_t$ implies that $\{(\gamma_{k,t},\theta_{k,t})\}$ satisfies \eqref{pth-lss}.

To prove the first relation in \eqref{two-ipt-ppt-theta-2}, we first establish the following equality:
\begin{align}\label{claim:sum-theta}
\sum_{t=1}^q\theta_{k,t} =1- \frac{\prod_{t=1}^q(1/\gamma_{k,t} - 1)}{\prod_{t=1}^q1/\gamma_{k,t}}.
\end{align}
Notice from \eqref{pf-exp-thetakt-ht} that $\theta_{k,t}=h_t(1)$ for each $1\le t\le q$. Let $h(\alpha)=\sum_{t=1}^qh_t(\alpha)$. It then follows that 
\begin{align}\label{sum-theta-kt}
\sum_{t=1}^q\theta_{k,t} = \sum_{t=1}^q h_t(1) = h(1).    
\end{align}
Also, observe that $h_t(0)=0$ for $1\leq t \leq q$. By this,  $h_t(1/\gamma_{k,t})=1$ and $h_t(1/\gamma_{k,s})=0$ for all $1\le s\le q$ and $s\neq t$, one can see that $h$ satisfies $h(0)=0$ and $h(1/\gamma_{k,t})=1$ for all $1\le t\le q$. Using these and the fact that $1/\gamma_{k,t}$, $1\le t\le q$, are distinct, we conclude that $h$ is uniquely given by 
\begin{align*}
h(\alpha) = 1- \frac{\prod_{t=1}^q(1/\gamma_{k,t}-\alpha)}{\prod_{t=1}^q1/\gamma_{k,t}},
\end{align*}
which along with \eqref{sum-theta-kt} implies that \eqref{claim:sum-theta} holds as desired.

We are now ready to prove the first relation in \eqref{two-ipt-ppt-theta-2}. Substituting the definition of $\{\gamma_{k,t}\}$ given in \eqref{unk-theatkt} into \eqref{claim:sum-theta}, we obtain
\begin{align}\label{sum-theatkt}
\sum_{t=1}^q\theta_{k,t} \overset{\eqref{unk-theatkt}}{=} 1 - \frac{\prod_{1\le t\le q}(t^2/\gamma_k - 1)}{\prod_{1\le t\le q}(t^2/\gamma_k)} = 1 - \prod_{1\le t\le q}\Big(1-\frac{\gamma_k}{t^2}\Big).    
\end{align}
It follows from \eqref{W-ineq} that
\begin{align*}
1 - \gamma_k\sum_{t=1}^q\frac{1}{t^2} \le \prod_{1\le t\le q}\Big(1 - \frac{\gamma_k}{t^2}\Big) \le \frac{1}{1 + \gamma_k\sum_{t=1}^{q}(1/t^2)}.   
\end{align*}
Using these, \eqref{sum-theatkt}, and the identity $\sum_{t=1}^\infty(1/t^2)=\pi^2/6$, we obtain that
\begin{align*}
\frac{\gamma_k}{1 + \pi^2/6}< \frac{\gamma_k\sum_{t=1}^q(1/t^2) }{1 + \gamma_k\sum_{t=1}^q(1/t^2) }\le \sum_{t=1}^q\theta_{k,t}\le \gamma_k\sum_{t=1}^q\frac{1}{t^2}\le \frac{\pi^2\gamma_k}{6} < 2\gamma_k,
\end{align*}
where the first inequality is due to $\gamma_k\sum_{t=1}^q(1/t^2) < \pi^2/6$ and $\sum_{t=1}^q(1/t^2)\geq 1$. The above inequalities along with $\gamma_k\in(0,1/2]$ implies that the first relation in \eqref{two-ipt-ppt-theta-2} holds. In addition, by \eqref{unk-theatkt}, $\gamma_k\in(0,1/2]$, and Lemma \ref{lem:inf-prod-tool}, one can see that for all $1\le t\le q$,
\begin{align*}
|\theta_{k,t}| & = \frac{\prod_{1\le s\le q}(s^2/\gamma_k - 1)}{\prod_{1\le s\le q}(s^2/\gamma_k)} \cdot \frac{\prod_{1\le s\le q, s\neq t}(s^2/\gamma_k)}{|\prod_{1\le s\le q, s\neq t}((t^2-s^2)/\gamma_k)|}\cdot \frac{1}{t^2/\gamma_k - 1}\\
&= \frac{\prod_{1\le s\le q}(1 -\gamma_k/s^2)}{|\prod_{1\le s\le q, s\neq t}(t^2/s^2-1)|} \cdot \frac{1}{t^2/\gamma_k - 1}\le \frac{\prod_{1\le s\le q}(1 -\gamma_k/s^2)}{|\prod_{s\ge1, s\neq t}(t^2/s^2-1)|} \cdot \frac{1}{t^2/\gamma_k - 1} \\
&= \frac{2\prod_{1\le s\le q}(1 -\gamma_k/s^2)}{t^2/\gamma_k - 1} \le \frac{2}{t^2/\gamma_k - 1} \le \frac{4\gamma_k}{t^2},
\end{align*}
where the first inequality is due to $|\prod_{s\ge q+1}(t^2/s^2-1)|\le1$, the third equality follows from Lemma \ref{lem:inf-prod-tool}, the second inequality is due to $\prod_{1\le s\le q}(1 -\gamma_k/s^2)\in(0,1)$, and the last inequality is due to $\gamma_k\in(0,1/2]$ and $t\ge1$. Hence, the second relation in \eqref{two-ipt-ppt-theta-2} also holds, which completes the proof of this lemma.
\end{proof}

The next lemma shows that $\{(\gamma_{k,t},\theta_{k,t})\}$ satisfying \eqref{pth-lss} can lead to the following important identity that will be used for the subsequent analysis.

\begin{lemma}
Suppose that Assumptions \ref{asp:basic} and \ref{asp:high-smooth} hold. Let $\Rp(\cdot,\cdot)$ be defined in \eqref{def:taylor-res}, and $\{x^k\}$ and $\{z^{k,t}\}$ be generated by Algorithm \ref{alg:unf-sfom} with input parameters $q=p-1$ and $\{(\gamma_{k,t},\theta_{k,t})\}$ satisfying \eqref{pth-lss}, where $p$ is given in Assumption \ref{asp:high-smooth}. Then it holds that for all $k\geq 0$, 
\begin{align}
\nabla f(x^{k+1})= \Big(1 - \sum_{t=1}^{p-1}\theta_{k,t}\Big)\nabla f(x^k) + \sum_{t=1}^{p-1}\theta_{k,t}\nabla f(z^{k+1,t}) + \Rp(x^{k+1},x^k) - \sum_{t=1}^{p-1}\theta_{k,t}\Rp(z^{k+1,t},x^k).\label{idtt-from-high-smth}
\end{align}
\end{lemma}

\begin{proof}
Fix any $k\ge0$. It follows from \eqref{update-zk} with $q=p-1$ that
\begin{align}\label{extra-rela}
z^{k+1,t}-x^k = \frac{1}{\gamma_{k,t}}(x^{k+1}-x^k)\qquad\forall 1\le t\le p-1.   
\end{align}
By this and \eqref{pth-lss}, one has that 
\begin{align*}
&\nabla f(x^{k+1}) \overset{\eqref{pth-lss}}{=} \nabla f(x^{k+1}) - \sum_{r=2}^p\Big(\Big(1-\sum_{t=1}^{p-1}\frac{\theta_{k,t}}{\gamma_{k,t}^{r-1}}\Big)\frac{1}{(r-1)!}\nabla^r f(x^k)(x^{k+1}-x^k)^{r-1}\Big)\\
&=\nabla f(x^{k+1}) - \sum_{r=2}^p\frac{1}{(r-1)!}\nabla^r f(x^k)(x^{k+1}-x^k)^{r-1} + \sum_{t=1}^{p-1} \sum_{r=2}^p\frac{\theta_{k,t}}{(r-1)!\gamma_{k,t}^{r-1}}\nabla^r f(x^k)(x^{k+1}-x^k)^{r-1}\\
&\overset{\eqref{extra-rela}}{=}\nabla f(x^{k+1}) - \sum_{r=2}^p\frac{1}{(r-1)!}\nabla^r f(x^k)(x^{k+1}-x^k)^{r-1} + \sum_{t=1}^{p-1} \sum_{r=2}^p\frac{\theta_{k,t}}{(r-1)!}\nabla^r f(x^k)(z^{k+1,t}-x^k)^{r-1}\\
&= \Big(1-\sum_{t=1}^{p-1}\theta_{k,t}\Big)\nabla f(x^k) + \sum_{t=1}^{p-1}\theta_{k,t}\nabla f(z^{k+1,t}) + \Big(\nabla f(x^{k+1}) - \sum_{r=1}^p\frac{1}{(r-1)!}\nabla^r f(x^k)(x^{k+1}-x^k)^{r-1}\Big)\\
&\qquad -\sum_{t=1}^{p-1}\theta_{k,t}\Big(\nabla f(z^{k+1,t}) - \sum_{r=1}^p\frac{1}{(r-1)!}\nabla^r f(x^k)(z^{k+1,t}-x^k)^{r-1}\Big),
\end{align*}
which along with the definition of $\Rp$ in \eqref{def:taylor-res} implies that \eqref{idtt-from-high-smth} holds.
\end{proof}

The following lemma presents a recurrence relation for the estimation error of the gradient estimators $\{m^k\}$ generated by Algorithm \ref{alg:unf-sfom}.

\begin{lemma}\label{lem:rec-em}
Suppose that Assumptions \ref{asp:basic} and \ref{asp:high-smooth} hold. Let $\{(x^k,m^k)\}$ be the sequence generated by Algorithm \ref{alg:unf-sfom} with input parameters $q=p-1$, $\{\eta_k\}$, and $\{(\gamma_{k,t},\theta_{k,t})\}$ satisfying \eqref{pth-lss} and \eqref{sum-tht-bd}. Then we have
\begin{align}
&\mathbb{E}_{{\xi^{k+1}}}[\|m^{k+1}-\nabla f(x^{k+1})\|^\alpha] \le \Big(1 - \sum_{t=1}^{p-1}\theta_{k,t}\Big)\|m^k-\nabla f(x^k)\|^\alpha  \nonumber \\ 
& \qquad \qquad\quad + \frac{3p^{\alpha-1} L_p^\alpha}{(p!)^\alpha}\Big(\sum_{t=1}^{p-1}\theta_{k,t}\Big)^{1-\alpha}\eta_k^{\alpha p}\Big(1 + \sum_{t=1}^{p-1}\frac{|\theta_{k,t}|^\alpha}{\gamma_{k,t}^{\alpha p}}\Big) + 2(p-1)^{\alpha-1}\sigma^\alpha\sum_{t=1}^{p-1}|\theta_{k,t}|^\alpha\qquad\forall k\ge0,\label{ineq:vr-mem}
\end{align}
where $L_1$, $\sigma$, $\alpha$, $p$, and $L_p$ are given in Assumptions \ref{asp:basic} and \ref{asp:high-smooth}, respectively.
\end{lemma}

\begin{proof}
Fix any $k\ge0$. It follows from \eqref{update-mk}, \eqref{idtt-from-high-smth}, and  $q=p-1$ that 
\begin{align}
\nabla f(x^{k+1}) - m^{k+1} & = \Big(1 - \sum_{t=1}^{p-1}\theta_{k,t}\Big)(\nabla f(x^k)-m^k) + \sum_{t=1}^{p-1}\theta_{k,t}(\nabla f(z^{k+1,t}) - G(z^{k+1,t};\xi^{k+1}))\nonumber  \\
&\quad \  + \Rp(x^{k+1},x^k) - \sum_{t=1}^{p-1}\theta_{k,t}\Rp(z^{k+1,t},x^k).\label{ipt-idt-hsmth-mem}
\end{align}
Observe from Algorithm \ref{alg:unf-sfom} and Assumption \ref{asp:basic} that $\|x^{k+1}-x^k\|=\eta_k$, $\|z^{k+1,t}-x^k\|=\eta_k/\gamma_{k,t}$, $\E_{\xi^{k+1}}[G(z^{k+1,t};\xi^{k+1})-\nabla f(z^{k+1,t})]=0$, and $\mathbb{E}_{\xi^{k+1}}[\|\nabla f(z^{k+1,t}) - G(z^{k+1,t};\xi^{k+1})\|^\alpha]\le\sigma^\alpha$ for all $1\le t\le p-1$. Using these, \eqref{open-alpha}, \eqref{open-alpha-2}, \eqref{rela-mk+1-pm}, \eqref{ipt-idt-hsmth-mem}, and Lemma \ref{lem:grad-desc-near-high}, we obtain that for all $c>0$, 
\begin{align}
&\mathbb{E}_{{\xi^{k+1}}}[\|\nabla f(x^{k+1})-m^{k+1}\|^\alpha]\nonumber\\
&\overset{\eqref{ipt-idt-hsmth-mem}}{=}\mathbb{E}_{{\xi^{k+1}}}\Big[\Big\|\Big(1 - \sum_{t=1}^{p-1}\theta_{k,t}\Big)(\nabla f(x^k)-m^k) + \sum_{t=1}^{p-1}\theta_{k,t}(\nabla f(z^{k+1,t}) - G(z^{k+1,t};\xi^{k+1}))\nonumber\\
&\qquad + \Rp(x^{k+1},x^k) - \sum_{t=1}^{p-1}\theta_{k,t}\Rp(z^{k+1,t},x^k)\Big\|^\alpha\Big]\nonumber\\
&\overset{\eqref{open-alpha}}{\le} \Big\|\Big(1 - \sum_{t=1}^{p-1}\theta_{k,t}\Big)(\nabla f(x^k)-m^k)+\Rp(x^{k+1},x^k) - \sum_{t=1}^{p-1}\theta_{k,t}\Rp(z^{k+1,t},x^k)\Big\|^\alpha\nonumber\\
&\quad + 2\mathbb{E}_{{\xi^{k+1}}}\Big[\Big\|\sum_{t=1}^{p-1}\theta_{k,t}(\nabla f(z^{k+1,t}) - G(z^{k+1,t};\xi^{k+1}))\Big\|^\alpha\Big]\nonumber\\
&\overset{\eqref{open-alpha-2}}{\le} (1+c)\Big(1 - \sum_{t=1}^{p-1}\theta_{k,t}\Big)^\alpha\|\nabla f(x^k)-m^k\|^\alpha \nonumber\\
&\quad + (2+(\alpha-1)^{\alpha-1}c^{1-\alpha})\Big\|\Rp(x^{k+1},x^k) - \sum_{t=1}^{p-1}\theta_{k,t}\Rp(z^{k+1,t},x^k)\Big\|^\alpha\nonumber\\
&\quad + 2\mathbb{E}_{{\xi^{k+1}}}\Big[\Big\|\sum_{t=1}^{p-1}\theta_{k,t}(\nabla f(z^{k+1,t}) - G(z^{k+1,t};\xi^{k+1}))\Big\|^\alpha\Big]\nonumber\\
&\le (1+c)\Big(1 - \sum_{t=1}^{p-1}\theta_{k,t}\Big)^\alpha\|\nabla f(x^k)-m^k\|^\alpha\nonumber\\
&\quad + p^{\alpha-1}(2+(\alpha-1)^{\alpha-1}c^{1-\alpha}) \Big(\|\Rp(x^{k+1},x^k)\|^\alpha + \sum_{t=1}^{p-1}|\theta_{k,t}|^\alpha\|\Rp(z^{k+1,t},x^k)\|^\alpha\Big)\nonumber\\
&\quad + 2(p-1)^{\alpha-1}\sum_{t=1}^{p-1}|\theta_{k,t}|^\alpha\mathbb{E}_{\xi^{k+1}}[\|\nabla f(z^{k+1,t}) - G(z^{k+1,t};\xi^{k+1})\|^\alpha]\nonumber\\
&\le (1+c)\Big(1 - \sum_{t=1}^{p-1}\theta_{k,t}\Big)^\alpha\|\nabla f(x^k)-m^k\|^\alpha\nonumber\\
&\quad + \frac{p^{\alpha-1} L_p^\alpha}{(p!)^\alpha}(2+(\alpha-1)^{\alpha-1}c^{1-\alpha})\eta_k^{\alpha p}\Big(1 + \sum_{t=1}^{p-1}\frac{|\theta_{k,t}|^\alpha}{\gamma_{k,t}^{\alpha p}}\Big) + 2(p-1)^{\alpha-1}\sigma^\alpha\sum_{t=1}^{p-1}|\theta_{k,t}|^\alpha,\label{vr-inter-mem}
 \end{align}
where the first inequality is due to \eqref{open-alpha}, and the fact that $\E_{\xi^{k+1}}[G(z^{k+1,t};\xi^{k+1})-\nabla f(z^{k+1,t})]=0$ for all $1\le t\le p-1$, the third inequality follows from $\|\sum_{t=1}^{m}w_t\|^\alpha \le m^{\alpha-1}\sum_{t=1}^{m}\|w_t\|^\alpha$ because of the convexity of $\|\cdot\|^\alpha$, and the last inequality follows from Lemma \ref{lem:grad-desc-near-high}, $\|x^{k+1}-x^k\|=\eta_k$, $\|z^{k+1,t}-x^k\|=\eta_k/\gamma_{k,t}$, and 
$\mathbb{E}_{\xi^{k+1}}[\|\nabla f(z^{k+1,t}) - G(z^{k+1,t};\xi^{k+1})\|^\alpha]\le\sigma^\alpha$ for all $1\le t\le p-1$. 

Letting $c=(1-\sum_{t=1}^{p-1}\theta_{k,t})^{1-\alpha}-1$ in \eqref{vr-inter-mem}, and using $\sum_{t=1}^{p-1}\theta_{k,t}\in(0,1)$ (see \eqref{sum-tht-bd}) and $\alpha\in(1,2]$, we obtain that
\begin{align*}
c^{1-\alpha} & = \Big(\frac{1}{(1-\sum_{t=1}^{p-1}\theta_{k,t})^{\alpha-1}} - 1\Big)^{1-\alpha} \le \Big(\frac{1}{1-(\alpha-1)\sum_{t=1}^{p-1}\theta_{k,t}} - 1\Big)^{1-\alpha}\\
& =\Big(\frac{1-(\alpha-1)\sum_{t=1}^{p-1}\theta_{k,t}}{(\alpha-1)\sum_{t=1}^{p-1}\theta_{k,t}} \Big)^{\alpha-1} \le \Big((\alpha-1)\sum_{t=1}^{p-1}\theta_{k,t}\Big)^{1-\alpha},
\end{align*}
where the first inequality follows from $(1-\tau)^\beta\le 1-\beta\tau$ for all $\tau\in(-\infty,1)$ and $\beta\in[0,1]$. Combining the above inequality with \eqref{vr-inter-mem}, one can obtain that
\begin{align*}
&\mathbb{E}_{{\xi^{k+1}}}[\|\nabla f(x^{k+1})-m^{k+1}\|^\alpha] \le \Big(1 - \sum_{t=1}^{p-1}\theta_{k,t}\Big)\|\nabla f(x^k)-m^k\|^\alpha \\
&\qquad\qquad\qquad\qquad + \frac{p^{\alpha-1} L_p^\alpha}{(p!)^\alpha}\Big(2+\Big(\sum_{t=1}^{p-1}\theta_{k,t}\Big)^{1-\alpha}\Big)\eta_k^{\alpha p}\Big(1 + \sum_{t=1}^{p-1}\frac{|\theta_{k,t}|^\alpha}{\gamma_{k,t}^{\alpha p}}\Big) + 2(p-1)^{\alpha-1}\sigma^\alpha\sum_{t=1}^{p-1}|\theta_{k,t}|^\alpha,
\end{align*}
which together with $\sum_{t=1}^{p-1}\theta_{k,t}\in(0,1)$ and $\alpha\in(1,2]$ implies that \eqref{ineq:vr-mem} holds.
\end{proof}

The following lemma establishes a descent property for the potential sequence $\{{\mathcal P}_k\}$ defined below.

\begin{lemma}\label{thm:rate-mem}
Suppose that Assumptions \ref{asp:basic} and \ref{asp:high-smooth} hold. Let $\{(x^k,m^k)\}$ be generated by Algorithm \ref{alg:unf-sfom} with input parameters $q=p-1$, $\{\eta_k\}$, and  $\{(\gamma_{k,t},\theta_{k,t})\}$ satisfying \eqref{pth-lss} and \eqref{sum-tht-bd}. Let $L_1$, $\sigma$, and $\alpha$ be given in Assumption \ref{asp:basic}, $p$ and $L_p$ be given in Assumption \ref{asp:high-smooth}, and $\{\mathcal{P}_k\}$ be defined in \eqref{def:pot-pm} for $\{(x^k,m^k)\}$ and any positive sequence $\{p_k\}$ that satisfies $(1-\sum_{t=1}^{p-1}\theta_{k,t})p_{k+1} \le (1-\sum_{t=1}^{p-1}\theta_{k,t}/10)p_k$ for all $k\ge0$. Then it holds that for all $k\geq 0$,
\begin{align}
\mathbb{E}_{{\xi^{k+1}}}[{\mathcal P}_{k+1}] & \le  {\mathcal P}_k- \eta_k\|\nabla f(x^k)\| + \frac{L_1}{2}\eta_k^2 + \frac{(\alpha-1)(2\eta_k)^{\alpha/(\alpha-1)}}{\alpha^{\alpha/(\alpha-1)}(p_k\sum_{t=1}^{p-1}\theta_{k,t}/10)^{1/(\alpha-1)}}\nonumber\\
&\quad + \frac{3p^{\alpha-1} L_p^\alpha}{(p!)^\alpha} p_{k+1}\Big(\sum_{t=1}^{p-1}\theta_{k,t}\Big)^{1-\alpha}\eta_k^{\alpha p}\Big(1 + \sum_{t=1}^{p-1}\frac{|\theta_{k,t}|^\alpha}{\gamma_{k,t}^{\alpha p}}\Big) + 2(p-1)^{\alpha-1}\sigma^\alpha p_{k+1}\sum_{t=1}^{p-1}|\theta_{k,t}|^\alpha.
\label{stat-bd-mem}
\end{align}
\end{lemma}

\begin{proof}
Fix any $k\ge0$. By Lemma \ref{lem:ppt-desc} with $(x^+,x,m,\eta)=(x^{k+1},x^k,m^k,\eta_k)$, one has
\begin{align}\label{upbd-fxk+1-mem}
f(x^{k+1}) \le f(x^k) - \eta_k\|\nabla f(x^k)\| + 2\eta_k\|\nabla f(x^k) - m^k\| + \frac{L_1}{2}\eta_k^2.    
\end{align}
Combining this with \eqref{def:pot-pm} and \eqref{ineq:vr-mem}, we obtain that 
\begin{align}
&\mathbb{E}_{{\xi^{k+1}}}[{\mathcal P}_{k+1}] \overset{\eqref{def:pot-pm}}{=} \mathbb{E}_{{\xi^{k+1}}}[f(x^{k+1}) + p_{k+1}\|\nabla f(x^{k+1})-m^{k+1}\|^\alpha]\nonumber\\
&\overset{\eqref{ineq:vr-mem}\eqref{upbd-fxk+1-mem}}{\le} f(x^k) - \eta_k\|\nabla f(x^k)\| + 2\eta_k\|\nabla f(x^k) - m^k\| + \frac{L_1}{2}\eta_k^2 + \Big(1-\sum_{t=1}^{p-1}\theta_{k,t}\Big)p_{k+1}\|\nabla f(x^k)-m^k\|^\alpha\nonumber \\
&\qquad + \frac{3p^{\alpha-1} L_p^\alpha}{(p!)^\alpha} p_{k+1}\Big(\sum_{t=1}^{p-1}\theta_{k,t}\Big)^{1-\alpha}\eta_k^{\alpha p}\Big(1 + \sum_{t=1}^{p-1}\frac{|\theta_{k,t}|^\alpha}{\gamma_{k,t}^{\alpha p}}\Big) + 2(p-1)^{\alpha-1}\sigma^\alpha p_{k+1}\sum_{t=1}^{p-1}|\theta_{k,t}|^\alpha \nonumber\\
&\le f(x^k) - \eta_k\|\nabla f(x^k)\| + 2\eta_k\|\nabla f(x^k) - m^k\| + \frac{L_1}{2}\eta_k^2 + \Big(1-\sum_{t=1}^{p-1}\theta_{k,t}/10\Big)p_k\|\nabla f(x^k)-m^k\|^\alpha\nonumber\\
&\quad + \frac{3p^{\alpha-1} L_p^\alpha}{(p!)^\alpha} p_{k+1}\Big(\sum_{t=1}^{p-1}\theta_{k,t}\Big)^{1-\alpha}\eta_k^{\alpha p}\Big(1 + \sum_{t=1}^{p-1}\frac{|\theta_{k,t}|^\alpha}{\gamma_{k,t}^{\alpha p}}\Big) + 2(p-1)^{\alpha-1}\sigma^\alpha p_{k+1}\sum_{t=1}^{p-1}|\theta_{k,t}|^\alpha,\label{pot-desc-ineq-mem}
\end{align}
where the last inequality is due to $(1-\sum_{t=1}^{p-1}\theta_{k,t})p_{k+1}\le(1-\sum_{t=1}^{p-1}\theta_{k,t}/10)p_k$. In addition, letting $\alpha^\prime=\alpha/(\alpha-1)$ and using the Young's inequality, we have that
\begin{align*}
2\eta_k\|\nabla f(x^k) - m^k\| & \le \frac{((\alpha p_k\sum_{t=1}^{p-1}\theta_{k,t}/10)^{1/\alpha} \|\nabla f(x^k) - m^k\|)^\alpha}{\alpha} + \frac{(2\eta_k/(\alpha p_k\sum_{t=1}^{p-1}\theta_{k,t}/10)^{1/\alpha})^{\alpha^\prime}}{\alpha^\prime} \\
&= \frac{p_k\sum_{t=1}^{p-1}\theta_{k,t}}{10}\|\nabla f(x^k) - m^k\|^\alpha + \frac{(\alpha-1)(2\eta_k)^{\alpha/(\alpha-1)}}{\alpha^{\alpha/(\alpha-1)}(p_k\sum_{t=1}^{p-1}\theta_{k,t}/10)^{1/(\alpha-1)}}.
\end{align*}
This together with \eqref{pot-desc-ineq-mem} implies that
\begin{align*}
\mathbb{E}_{{\xi^{k+1}}}[{\mathcal P}_{k+1}] &\le f(x^k) + p_k\|m^k-\nabla f(x^k)\|^\alpha - \eta_k\|\nabla f(x^k)\| + \frac{L_1}{2}\eta_k^2 + \frac{(\alpha-1)(2\eta_k)^{\alpha/(\alpha-1)}}{\alpha^{\alpha/(\alpha-1)}(p_k\sum_{t=1}^{p-1}\theta_{k,t}/10)^{1/(\alpha-1)}}\\
&\quad + \frac{3p^{\alpha-1} L_p^\alpha}{(p!)^\alpha} p_{k+1}\Big(\sum_{t=1}^{p-1}\theta_{k,t}\Big)^{1-\alpha}\eta_k^{\alpha p}\Big(1 + \sum_{t=1}^{p-1}\frac{|\theta_{k,t}|^\alpha}{\gamma_{k,t}^{\alpha p}}\Big) + 2(p-1)^{\alpha-1}\sigma^\alpha p_{k+1}\sum_{t=1}^{p-1}|\theta_{k,t}|^\alpha.
\end{align*}
The conclusion \eqref{stat-bd-mem} then follows from this relation and \eqref{def:pot-pm}.
\end{proof}

We next establish some properties for a specific choice of $\{(\gamma_{k,t},\theta_{k,t})\}$ and $\{p_k\}$, which will be used to prove  Theorem \ref{thm:known-rate-mem} subsequently.

\begin{lemma}\label{lem:pk-mem}
Let $\{(\gamma_{k,t},\theta_{k,t})\}$ be defined in \eqref{def:theta-known-rate-mem} and \eqref{def:gmak-mem-know}, and let $\{p_k\}$ be defined as     
\begin{align}\label{def:pk-known-mem}
p_k=(k+4)^{(p(\alpha-1)^2-\alpha+1)/(p(2\alpha-1)+\alpha-1)} \qquad\forall k\ge0.
\end{align}
Then \eqref{pth-lss} and \eqref{sum-tht-bd} hold for such $\{(\gamma_{k,t},\theta_{k,t})\}$, and moreover, $(1-\sum_{t=1}^{p-1}\theta_{k,t})p_{k+1} \le (1-\sum_{t=1}^{p-1}\theta_{k,t}/10)p_k$ holds for all $k\ge0$.
\end{lemma}

\begin{proof}
Fix any $k\ge0$. Notice that $p\alpha/(p(2\alpha-1)+\alpha-1)\in(1/2,1)$ for all $p\ge2$ and $\alpha\in(1,2]$. It then follows from \eqref{def:gmak-mem-know} that $\gamma_k=1/(k+4)^{p\alpha/(p(2\alpha-1)+\alpha-1)}\in(0,1/2)$. By this,  \eqref{def:theta-known-rate-mem}, and Lemma \ref{lem:ppt-thetakt}, one can see that \eqref{pth-lss} and \eqref{sum-tht-bd} hold for $\{(\gamma_{k,t},\theta_{k,t})\}$ that is defined in \eqref{def:theta-known-rate-mem} and \eqref{def:gmak-mem-know}. In addition, observe that
\begin{align}
&\frac{1-\sum_{t=1}^{p-1}\theta_{k,t}/10}{1-\sum_{t=1}^{p-1}\theta_{k,t}} =  1 + \frac{9\sum_{t=1}^{p-1}\theta_{k,t}}{10(1-\sum_{t=1}^{p-1}\theta_{k,t})} \ge 1 + \frac{9\sum_{t=1}^{p-1}\theta_{k,t}}{10} \overset{\eqref{two-ipt-ppt-theta-2}}{\ge} 1 + \frac{9\gamma_k}{10(1 + \pi^2/6)}\nonumber\\
&\overset{\eqref{def:gmak-mem-know}}{=} 1+\frac{9}{10(1 + \pi^2/6)(k+4)^{p\alpha/(p(2\alpha-1)+\alpha-1)}} >1+ \frac{9}{10(1 + \pi^2/6)(k+4)}> 1+\frac{1}{3(k+4)},\label{lwbd-frac-sum-thetak}
\end{align}
where the first inequality follows from $\sum_{t=1}^{p-1}\theta_{k,t}\in(0,1)$, and the third inequality is due to $p\alpha/(p(2\alpha-1)+\alpha-1)<1$ for all $p\ge1$ and $\alpha\in(1,2]$. Also, note that 
\begin{align*}
\frac{p_{k+1}}{p_k} = \Big(1+\frac{1}{k+4}\Big)^{(p(\alpha-1)^2-\alpha+1)/(p(2\alpha-1)+\alpha-1)} \le \Big(1+\frac{1}{k+4}\Big)^{1/3}\le 1+\frac{1}{3(k+4)}, 
\end{align*}
where the first inequality follows from $(p(\alpha-1)^2-\alpha+1)/(p(2\alpha-1)+\alpha-1)\le1/3$ for all $p\ge1$ and $\alpha\in(1,2]$, and the last inequality is due to $(1+\tau)^\beta\le 1+\tau \beta$ for all $\tau>-1$ and $\beta\in[0,1]$. The above relation together with \eqref{lwbd-frac-sum-thetak} implies that $(1-\sum_{t=1}^{p-1}\theta_{k,t})p_{k+1} \le (1-\sum_{t=1}^{p-1}\theta_{k,t}/10)p_k$ holds.
\end{proof}

We are now ready to prove Theorem \ref{thm:known-rate-mem}.

\begin{proof}[\textbf{Proof of Theorem \ref{thm:known-rate-mem}}]
Let $\{(x^k,m^k)\}$ be generated by Algorithm \ref{alg:unf-sfom} with $\{(\eta_k,\gamma_{k,t},\theta_{k,t})\}$ defined in \eqref{def:eta-known-rate-mem} and \eqref{def:theta-known-rate-mem}, and let 
$\{{\mathcal P}_k\}$ be defined in \eqref{def:pot-pm} with such $\{(x^k,m^k)\}$ and $\{p_k\}$ given in \eqref{def:pk-known-mem}. By Lemma \ref{lem:pk-mem}, one can see that such $\{(\eta_k,\gamma_{k,t},\theta_{k,t},p_k)\}$ satisfies the assumptions in Lemma \ref{thm:rate-mem} and Algorithm \ref{alg:unf-sfom}. In addition, by \eqref{def:pot-pm} and \eqref{def:pk-known-mem}, one has that
\begin{align}
\mathbb{E}[{\mathcal P}_0] & = f(x^0)+p_0\mathbb{E}[\|m^0-\nabla f(x^0)\|^\alpha]\nonumber\\   
&\le f(x^0) + 4^{(p(\alpha-1)^2-\alpha+1)/(p(2\alpha-1)+\alpha-1)}\mathbb{E}[\|G(x^0;\xi^0)-\nabla f(x^0)\|^\alpha] \le f(x^0)+ 4^{1/3}\sigma^\alpha,\label{upbd-exp-P0-mem}\\
\mathbb{E}[{\mathcal P}_K] & = \mathbb{E}[f(x^K)+p_K\|m^K-\nabla f(x^K)\|^\alpha]\ge \mathbb{E}[f(x^K)] \ge f_{\mathrm{low}},\label{lwbd-exp-Pk-mem}
\end{align}
where the inequality in \eqref{upbd-exp-P0-mem} is due to $(p(\alpha-1)^2-\alpha+1)/(p(2\alpha-1)+\alpha-1)\le1/3$ for all $p\ge1$ and $\alpha\in(1,2]$, and $\mathbb{E}[\|G(x^0;\xi^{0})-\nabla f(x^0)\|^\alpha]\le\sigma^\alpha$ {for all $1\le t\le p-1$}. Taking expectation on both sides of \eqref{stat-bd-mem} with respect to {$\{\xi^{i}\}_{0\le i\le k+1}$}, we have
\begin{align*}
\mathbb{E}[{\mathcal P}_{k+1}] &\le  \mathbb{E}[{\mathcal P}_k]- \eta_k\mathbb{E}[\|\nabla f(x^k)\|] + \frac{L_1}{2}\eta_k^2 + \frac{(\alpha-1)(2\eta_k)^{\alpha/(\alpha-1)}}{\alpha^{\alpha/(\alpha-1)}(p_k\sum_{t=1}^{p-1}\theta_{k,t}/10)^{1/(\alpha-1)}}\\
&\quad + \frac{3p^{\alpha-1} L_p^\alpha}{(p!)^\alpha} p_{k+1}\Big(\sum_{t=1}^{p-1}\theta_{k,t}\Big)^{1-\alpha}\eta_k^{\alpha p}\Big(1 + \sum_{t=1}^{p-1}\frac{|\theta_{k,t}|^\alpha}{\gamma_{k,t}^{\alpha p}}\Big) + 2(p-1)^{\alpha-1}\sigma^\alpha p_{k+1}\sum_{t=1}^{p-1}|\theta_{k,t}|^\alpha{\quad\forall k\ge0.}
\end{align*}
Summing up this inequality over $k=0,\ldots,K-1$, and using \eqref{upbd-exp-P0-mem} and \eqref{lwbd-exp-Pk-mem}, we obtain that for all $K\ge1$,
\begin{align}
&f_{\mathrm{low}}\overset{\eqref{lwbd-exp-Pk-mem}}{\le} \mathbb{E}[{\mathcal P}_K] \le \E[{\mathcal P}_0] - \sum_{k=0}^{K-1}\eta_k\mathbb{E}[\|\nabla f(x^k)\|]+ \sum_{k=0}^{K-1}\Big(\frac{L_1}{2}\eta_k^2 + \frac{(\alpha-1)(2\eta_k)^{\alpha/(\alpha-1)}}{\alpha^{\alpha/(\alpha-1)}(p_k\sum_{t=1}^{p-1}\theta_{k,t}/10)^{1/(\alpha-1)}}\nonumber\\
&\quad + \frac{3p^{\alpha-1} L_p^\alpha}{(p!)^\alpha} p_{k+1}\Big(\sum_{t=1}^{p-1}\theta_{k,t}\Big)^{1-\alpha}\eta_k^{\alpha p}\Big(1 + \sum_{t=1}^{p-1}\frac{|\theta_{k,t}|^\alpha}{\gamma_{k,t}^{\alpha p}}\Big) + 2(p-1)^{\alpha-1}\sigma^\alpha p_{k+1}\sum_{t=1}^{p-1}|\theta_{k,t}|^\alpha\Big)\nonumber\\
&\overset{\eqref{upbd-exp-P0-mem}}{\le} f(x^0) + 4^{1/3}\sigma^\alpha - \eta_{K-1}\sum_{k=0}^{K-1}\mathbb{E}[\|\nabla f(x^k)\|] + \sum_{k=0}^{K-1}\Big(\frac{L_1}{2}\eta_k^2 + \frac{(\alpha-1)(2\eta_k)^{\alpha/(\alpha-1)}}{\alpha^{\alpha/(\alpha-1)}(p_k\sum_{t=1}^{p-1}\theta_{k,t}/10)^{1/(\alpha-1)}}\nonumber\\
&\quad + \frac{6p^{\alpha-1} L_p^\alpha}{(p!)^\alpha} p_k\Big(\sum_{t=1}^{p-1}\theta_{k,t}\Big)^{1-\alpha}\eta_k^{\alpha p}\Big(1 + \sum_{t=1}^{p-1}\frac{|\theta_{k,t}|^\alpha}{\gamma_{k,t}^{\alpha p}}\Big) + 4(p-1)^{\alpha-1}\sigma^\alpha p_k\sum_{t=1}^{p-1}|\theta_{k,t}|^\alpha\Big)\nonumber\\ 
&\le f(x^0) + 4^{1/3}\sigma^\alpha - \eta_{K-1}\sum_{k=0}^{K-1} \mathbb{E}[\|\nabla f(x^k)\|] + \sum_{k=0}^{K-1}\Big(\frac{L_1}{2}\eta_k^2 + \frac{30^{1/(\alpha-1)}(\alpha-1)(2/\alpha)^{\alpha/(\alpha-1)}\eta_k^{\alpha/(\alpha-1)}}{(p_k\gamma_k)^{1/(\alpha-1)}}\nonumber\\
&\quad + \frac{18 p^{\alpha-1} L_p^\alpha}{(p!)^\alpha} p_k \gamma_k^{1-\alpha}\eta_k^{\alpha p}\Big(1 + \sum_{t=1}^{p-1}\frac{|\theta_{k,t}|^\alpha}{\gamma_{k,t}^{\alpha p}}\Big) + 64(p-1)^\alpha\sigma^\alpha p_k\gamma_k^\alpha\Big)\nonumber\\
&\le f(x^0) + 4^{1/3}\sigma^\alpha - \eta_{K-1}\sum_{k=0}^{K-1} \mathbb{E}[\|\nabla f(x^k)\|] + \sum_{k=0}^{K-1}\Big(\frac{L_1}{2}\eta_k^2 + \frac{30^{1/(\alpha-1)}(\alpha-1)(2/\alpha)^{\alpha/(\alpha-1)}\eta_k^{\alpha/(\alpha-1)}}{(p_k\gamma_k)^{1/(\alpha-1)}}\nonumber\\
&\quad + \frac{18p^{\alpha-1} L_p^\alpha}{(p!)^\alpha} p_k\gamma_k^{1-\alpha}\eta_k^{\alpha p} \Big(1+16\gamma_k^{\alpha-\alpha p}\sum_{t=1}^{p-1}t^{2\alpha(p-1)}\Big) + 64(p-1)^\alpha\sigma^\alpha p_k\gamma_k^\alpha\Big)\nonumber\\
&\le f(x^0) + 4^{1/3}\sigma^\alpha - \eta_{K-1}\sum_{k=0}^{K-1} \mathbb{E}[\|\nabla f(x^k)\|] + \sum_{k=0}^{K-1}\Big(\frac{L_1}{2}\eta_k^2 + \frac{30^{1/(\alpha-1)}(\alpha-1)(2/\alpha)^{\alpha/(\alpha-1)}\eta_k^{\alpha/(\alpha-1)}}{(p_k\gamma_k)^{1/(\alpha-1)}}\nonumber\\
&\quad + \frac{306p^{2\alpha p} L_p^\alpha}{(p!)^\alpha} p_k\gamma_k^{1-\alpha p}\eta_k^{\alpha p} + 64(p-1)^\alpha\sigma^\alpha p_k\gamma_k^\alpha\Big),
\label{rear-pot-desc-rm}
\end{align}
where the third inequality follows from \eqref{upbd-exp-P0-mem}, $p_{k+1}\le 2p_k$ for all $k\ge0$ due to \eqref{def:pk-known-mem}, and the fact that $\{\eta_k\}$ is nonincreasing, the fourth inequality is due to $\sum_{t=1}^{p-1}\theta_{k,t}\in(\gamma_k/3,2\gamma_k)$ and $|\theta_{k,t}|\le 4\gamma_k/t^2\le 4\gamma_k$ for all $1\le t\le p-1$ because of \eqref{two-ipt-ppt-theta-2}, the fifth inequality follows from $|\theta_{k,t}|^\alpha/\gamma_{k,t}^{\alpha p} \le 4^\alpha\gamma_k^{\alpha-\alpha p}t^{2\alpha(p-1)}\le 16\gamma_k^{\alpha-\alpha p}t^{2\alpha(p-1)}$ because of \eqref{two-ipt-ppt-theta-2} and \eqref{def:theta-known-rate-mem}, and the last inequality is due to $1+16\gamma_k^{\alpha-\alpha p}\sum_{t=1}^{p-1}t^{2\alpha(p-1)}\le 17\gamma_k^{\alpha-\alpha p}\sum_{t=1}^{p-1}t^{2\alpha(p-1)}\le17p^{2\alpha p-2\alpha+1}\gamma_k^{\alpha-\alpha p}$. Rearranging the terms in \eqref{rear-pot-desc-rm}, and using \eqref{def:Mpa-known}, \eqref{def:eta-known-rate-mem}, \eqref{def:gmak-mem-know}, and \eqref{def:pk-known-mem}, we obtain that for all $K\ge5$,
\begin{align*}
&\frac{1}{K}\sum_{k=0}^{K-1}\mathbb{E}[\|\nabla f(x^k)\|]\\
&\overset{\eqref{rear-pot-desc-rm}}{\le} \frac{f(x^0) - f_{\mathrm{low}} + 4^{1/3}\sigma^\alpha}{K\eta_{K-1}} + \frac{1}{K\eta_{K-1}}\sum_{k=0}^{K-1}\Big(\frac{L_1}{2}\eta_k^2 + \frac{30^{1/(\alpha-1)}(\alpha-1)(2/\alpha)^{\alpha/(\alpha-1)}\eta_k^{\alpha/(\alpha-1)}}{(p_k\gamma_k)^{1/(\alpha-1)}}\\
&\qquad + \frac{306p^{2\alpha p} L_p^\alpha}{(p!)^\alpha} p_k\gamma_k^{1-\alpha p}\eta_k^{\alpha p} + 64(p-1)^\alpha\sigma^\alpha p_k\gamma_k^\alpha\Big)\\
&\overset{\eqref{def:eta-known-rate-mem}\eqref{def:gmak-mem-know}\eqref{def:pk-known-mem}}{=} \frac{(f(x^0) - f_{\mathrm{low}} + 4^{1/3}\sigma^\alpha)(K+3)^{(p\alpha+\alpha-1)/(p(2\alpha-1)+\alpha-1)}}{K} \\
&\qquad + \frac{(K+3)^{(p\alpha+\alpha-1)/(p(2\alpha-1)+\alpha-1)}}{K}\sum_{k=0}^{K-1}\Big(\frac{L_1}{2(k+4)^{2(p\alpha+\alpha-1)/(p(2\alpha-1)+\alpha-1)}}\\
&\qquad + \frac{30^{1/(\alpha-1)}(\alpha-1)(2/\alpha)^{\alpha/(\alpha-1)} + 306 p^{2\alpha p} L_p^\alpha/(p!)^\alpha +64(p-1)^\alpha\sigma^\alpha}{k+4}\Big)\\
&\le\frac{2(f(x^0) - f_{\mathrm{low}} + 4^{1/3}\sigma^\alpha)}{K^{p(\alpha-1)/(p(2\alpha-1)+\alpha-1)}}  + \frac{2}{K^{p(\alpha-1)/(p(2\alpha-1)+\alpha-1)}}\sum_{k=0}^{K-1}\Big(\frac{L_1}{2(k+4)^{2(p\alpha+\alpha-1)/(p(2\alpha-1)+\alpha-1)}}\\
&\qquad + \frac{30^{1/(\alpha-1)}(\alpha-1)(2/\alpha)^{\alpha/(\alpha-1)} + 306 p^{2\alpha p} L_p^\alpha/(p!)^\alpha +64(p-1)^\alpha\sigma^\alpha}{k+4}\Big)\\
&\le \frac{2(f(x^0) - f_{\mathrm{low}} + 4^{1/3}\sigma^\alpha)}{K^{p(\alpha-1)/(p(2\alpha-1)+\alpha-1)}} + \frac{2}{K^{p(\alpha-1)/(p(2\alpha-1)+\alpha-1)}} \\
&\qquad \times \sum_{k=0}^{K-1}\Big(\frac{L_1/2 + 30^{1/(\alpha-1)}(\alpha-1)(2/\alpha)^{\alpha/(\alpha-1)} + 306 p^{2\alpha p} L_p^\alpha/(p!)^\alpha +64(p-1)^\alpha\sigma^\alpha}{k+4}\Big)\\
% &\le \frac{2(f(x^0) - f_{\mathrm{low}} + 4^{1/3}\sigma^\alpha)}{K^{p(\alpha-1)/(p(2\alpha-1)+\alpha-1)}}\\
% &\qquad + \frac{2(L_1/2 + 30^{1/(\alpha-1)}(\alpha-1)(2/\alpha)^{\alpha/(\alpha-1)} + 306 p^{2\alpha p} L_p^\alpha/(p!)^\alpha +64(p-1)^\alpha\sigma^\alpha)\ln(2K/5+1)}{K^{p(\alpha-1)/(p(2\alpha-1)+\alpha-1)}}\\
&\le \frac{4\Big(f(x^0) - f_{\mathrm{low}} + 4^{1/3}\sigma^\alpha + \frac{L_1}{2} + 30^{1/(\alpha-1)}(\alpha-1)(\frac{2}{\alpha})^{\alpha/(\alpha-1)} + \frac{306 p^{2\alpha p} L_p^\alpha}{(p!)^\alpha} +64(p-1)^\alpha\sigma^\alpha\Big)\ln K}{K^{p(\alpha-1)/(p(2\alpha-1)+\alpha-1)}}\\
&\overset{\eqref{def:Mpa-known}}{=} \frac{M_{p,\alpha} \ln K}{K^{p(\alpha-1)/(p(2\alpha-1)+\alpha-1)}},
\end{align*}
where the second inequality follows from $(K+3)^{(p\alpha+\alpha-1)/(p(2\alpha-1)+\alpha-1)}\le 2K^{(p\alpha+\alpha-1)/(p(2\alpha-1)+\alpha-1)}$ for all $K\ge5$, the third inequality is due to $2(p\alpha+\alpha-1)/(p(2\alpha-1)+\alpha-1)\ge1$ for all $p\ge1$ and $\alpha\in(1,2]$, and the last inequality follows from $\sum_{k=0}^{K-1}1/(k+4)\le \ln(2K/5+1) \le 2\ln K$ for all $K\ge 5$ due to \eqref{upbd:series-ka}. Recall that $\iota_K$ is uniformly selected from $\{0,\ldots,K-1\}$. It then follows from this and the above inequality that 
\begin{align}\label{pre-upbd-mem}
\E[\|\nabla f(x^{\iota_K})\|] = \frac{1}{K}\sum_{k=0}^{K-1}\mathbb{E}[\|\nabla f(x^k)\|] \le \frac{M_{p,\alpha}\ln K}{K^{p(\alpha-1)/(p(2\alpha-1)+\alpha-1)}}\qquad\forall K\ge 5.
\end{align}
By Lemma \ref{lem:rate-complexity} with $(\beta,u,v)=(p(\alpha-1)/(p(2\alpha-1)+\alpha-1),p(\alpha-1)\epsilon/(2(p(2\alpha-1)+\alpha-1)M_{p,\alpha}),K)$, one can see that $K^{-p(\alpha-1)/(p(2\alpha-1)+\alpha-1)}\ln K\le \epsilon/M_{p,\alpha}$ for all $K$ satisfying 
\[
 K\ge\Big(\frac{2(p(2\alpha-1)+\alpha-1)M_{p,\alpha}}{p(\alpha-1)\epsilon}\ln\Big(\frac{2(p(2\alpha-1)+\alpha-1)M_{p,\alpha}}{p(\alpha-1)\epsilon}\Big)\Big)^{(p(2\alpha-1)+\alpha-1)/(p(\alpha-1))},    
\]
which together with \eqref{pre-upbd-mem} implies that Theorem \ref{thm:known-rate-mem}.
\end{proof}

The following lemma establishes some properties for a specific choice of $\{(\gamma_{k,t},\theta_{k,t})\}$ and $\{p_k\}$, which will be used to prove Theorem \ref{thm:unknown-rate-mem} subsequently.

\begin{lemma}\label{lem:pk-mem-unk}
Let $\{(\gamma_{k,t},\theta_{k,t})\}$ be defined in \eqref{def:theta-unknown-rate-mem} and \eqref{def:gmak-mem-unknow}, and let $\{p_k\}$ be defined as     
\begin{align}\label{def:pk-unknown}
p_k=(k+4)^{(2p(\alpha-1)^2-\alpha)/(3p\alpha+\alpha)} \qquad\forall k\ge0.  
\end{align}
Then \eqref{pth-lss} and \eqref{sum-tht-bd} hold for such $\{(\gamma_{k,t},\theta_{k,t})\}$, and moreover, $(1-\sum_{t=1}^{p-1}\theta_{k,t})p_{k+1} \le (1-\sum_{t=1}^{p-1}\theta_{k,t}/10)p_k$ holds for all $k\ge0$.
\end{lemma}

\begin{proof}
Fix any $k\ge0$. Notice that $2p/(3p+1)> 1/2$ for all $p\ge2$. It then follows from \eqref{def:gmak-mem-unknow} that $\gamma_k=1/(k+4)^{2p/(3p+1)}\in(0,1/2)$. By this, \eqref{def:theta-unknown-rate-mem} and Lemma \ref{lem:ppt-thetakt}, one can see that \eqref{pth-lss} and \eqref{sum-tht-bd} hold for $\{(\gamma_{k,t},\theta_{k,t})\}$ that is defined in \eqref{def:theta-unknown-rate-mem} and \eqref{def:gmak-mem-unknow}. In addition, observe that
\begin{align}
&\frac{1-\sum_{t=1}^{p-1}\theta_{k,t}/10}{1-\sum_{t=1}^{p-1}\theta_{k,t}} =1 + \frac{9\sum_{t=1}^{p-1}\theta_{k,t}}{10(1-\sum_{t=1}^{p-1}\theta_{k,t})}  \ge 1 + \frac{9\sum_{t=1}^{p-1}\theta_{k,t}}{10}\overset{\eqref{two-ipt-ppt-theta-2}}{\ge} 1 + \frac{9\gamma_k}{10(1 + \pi^2/6)}\nonumber\\
&\overset{\eqref{def:gmak-mem-unknow}}{=}1+ \frac{9}{10(1 + \pi^2/6)(k+4)^{2p/(3p+1)}} >1+ \frac{9}{10(1 + \pi^2/6)(k+4)}> 1+\frac{1}{3(k+4)},\label{lwbd-frac-sum-thetak-unk}
\end{align}
where the first inequality follows from $\sum_{t=1}^{p-1}\theta_{k,t}\in(0,1)$,  the third inequality follows from $2p/(3p+1)< 1$ for all $p\ge2$. Also, note that 
\begin{align*}
\frac{p_{k+1}}{p_k} = \Big(1+\frac{1}{k+4}\Big)^{(2p(\alpha-1)^2-\alpha)/(3p\alpha+\alpha)} \le \Big(1+\frac{1}{k+4}\Big)^{1/3} \le 1+\frac{1}{3(k+4)}, 
\end{align*}
where the first inequality follows from $(2p(\alpha-1)^2-\alpha)/(3p\alpha+\alpha)\le 1/3$ for all $p\ge2$ and $\alpha\in(1,2]$, and the last inequality is due to $(1+\tau)^\beta\le 1+\tau \beta$ for all $\tau>-1$ and $\beta\in[0,1]$. The above relation along with \eqref{lwbd-frac-sum-thetak-unk} implies that $(1-\sum_{t=1}^{p-1}\theta_{k,t})p_{k+1} \le (1-\sum_{t=1}^{p-1}\theta_{k,t}/10)p_k$ holds.
\end{proof}

We are now ready to prove Theorem \ref{thm:unknown-rate-mem}.

\begin{proof}[\textbf{Proof of Theorem \ref{thm:unknown-rate-mem}}]

Let $\{(x^k,m^k)\}$ be generated by Algorithm \ref{alg:unf-sfom} with $\{(\eta_k,\gamma_{k,t},\theta_{k,t})\}$ defined in \eqref{def:eta-unknown-rate-mem}, \eqref{def:theta-unknown-rate-mem}, and \eqref{def:gmak-mem-unknow}, and let 
$\{{\mathcal P}_k\}$ be defined in \eqref{def:pot-pm} with such $\{(x^k,m^k)\}$ and $\{p_k\}$ given in \eqref{def:pk-unknown}. By Lemma \ref{lem:pk-mem-unk}, one can see that such $\{(\eta_k,\gamma_{k,t},\theta_{k,t},p_k)\}$ satisfies the assumptions in Lemma \ref{thm:rate-mem} and Algorithm \ref{alg:unf-sfom}.  Using this and similar arguments as those for deriving \eqref{rear-pot-desc-rm}, we obtain that for all $K\ge1$,
\begin{align*}
&f_{\mathrm{low}}\le f(x^0) + 4^{1/3}\sigma^\alpha - \eta_{K-1}\sum_{k=0}^{K-1} \mathbb{E}[\|\nabla f(x^k)\|] + \sum_{k=0}^{K-1}\Big(\frac{L_1}{2}\eta_k^2 + \frac{30^{1/(\alpha-1)}(\alpha-1)(2/\alpha)^{\alpha/(\alpha-1)}\eta_k^{\alpha/(\alpha-1)}}{(p_k\gamma_k)^{1/(\alpha-1)}}\nonumber\\
&\qquad \ + \frac{306p^{2\alpha p} L_p^\alpha}{(p!)^\alpha} p_k\gamma_k^{1-\alpha p}\eta_k^{\alpha p} + 64(p-1)^\alpha\sigma^\alpha p_k\gamma_k^\alpha\Big).
\end{align*}
Rearranging the terms in this inequality, and using \eqref{def:t-Mpa-unk}, \eqref{def:h-Mpa-unk}, \eqref{def:eta-unknown-rate-mem}, \eqref{def:gmak-mem-unknow}, and \eqref{def:pk-unknown}, we obtain that for all $K\ge5$, 
\begin{align*}
&\frac{1}{K}\sum_{k=0}^{K-1}\mathbb{E}[\|\nabla f(x^k)\|] \le \frac{f(x^0) - f_{\mathrm{low}} + 4^{1/3}\sigma^\alpha}{K\eta_{K-1}} + \frac{1}{K\eta_{K-1}}\sum_{k=0}^{K-1}\Big(\frac{L_1}{2}\eta_k^2\nonumber\\
&\qquad  + \frac{30^{1/(\alpha-1)}(\alpha-1)(2/\alpha)^{\alpha/(\alpha-1)}\eta_k^{\alpha/(\alpha-1)}}{(p_k\gamma_k)^{1/(\alpha-1)}}+ \frac{306p^{2\alpha p} L_p^\alpha}{(p!)^\alpha} p_k\gamma_k^{1-\alpha p}\eta_k^{\alpha p} + 64(p-1)^\alpha\sigma^\alpha p_k\gamma_k^\alpha\Big)\\
&\overset{\eqref{def:eta-unknown-rate-mem}\eqref{def:gmak-mem-unknow}\eqref{def:pk-unknown}}{=} \frac{(f(x^0) - f_{\mathrm{low}} + 4^{1/3}\sigma^\alpha)(K+3)^{(2p+1)/(3p+1)}}{K}\nonumber\\
&\qquad + \frac{(K+3)^{(2p+1)/(3p+1)}}{K}\sum_{k=0}^{K-1}\Big(\frac{L_1}{2(k+4)^{2(2p+1)/(3p+1)}}\nonumber\\
&\qquad  + \frac{30^{1/(\alpha-1)}(\alpha-1)(2/\alpha)^{\alpha/(\alpha-1)}+ 64(p-1)^\alpha\sigma^\alpha}{(k+4)^{(2p(2\alpha-1)+\alpha)/(3p\alpha+\alpha)}}+ \frac{306p^{2p\alpha} L_p^\alpha/(p!)^\alpha}{(k+4)^{(p(6\alpha-\alpha^2-2)+\alpha)/(3p\alpha+\alpha)}}\Big)\\
&\le \frac{2(f(x^0) - f_{\mathrm{low}} + 4^{1/3}\sigma^\alpha)}{K^{p/(3p+1)}} + \frac{2}{K^{p/(3p+1)}}\sum_{k=0}^{K-1}\Big(\frac{L_1}{2(k+4)^{2(2p+1)/(3p+1)}}\nonumber\\
&\qquad  + \frac{30^{1/(\alpha-1)}(\alpha-1)(2/\alpha)^{\alpha/(\alpha-1)}+ 64(p-1)^\alpha\sigma^\alpha}{(k+4)^{(2p(2\alpha-1)+\alpha)/(3p\alpha+\alpha)}}+ \frac{306p^{2p\alpha} L_p^\alpha/(p!)^\alpha}{(k+4)^{(p(6\alpha-\alpha^2-2)+\alpha)/(3p\alpha+\alpha)}}\Big)\\
&\le \frac{2(f(x^0) - f_{\mathrm{low}} + 4^{1/3}\sigma^\alpha)}{K^{p/(3p+1)}} + \frac{2}{K^{p/(3p+1)}}\sum_{k=0}^{K-1}\frac{L_1/2 + 306 p^{2p\alpha} L_p^\alpha/(p!)^\alpha}{k+4}\\
&\qquad + \frac{2}{K^{p/(3p+1)}}\sum_{k=0}^{K-1}\frac{(30^{1/(\alpha-1)}(\alpha-1)(2/\alpha)^{\alpha/(\alpha-1)}+ 64(p-1)^\alpha\sigma^\alpha)(K+3)^{p(2-\alpha)/(3p\alpha+\alpha)}}{k+4}\\
&\le \frac{2(f(x^0) - f_{\mathrm{low}} + 4^{1/3}\sigma^\alpha)}{K^{p/(3p+1)}} + \frac{2}{K^{p/(3p+1)}}\sum_{k=0}^{K-1}\frac{L_1/2 + 306 p^{2p\alpha} L_p^\alpha/(p!)^\alpha}{k+4}\\
&\qquad + \frac{4}{K^{p/(3p+1)}}\sum_{k=0}^{K-1}\frac{(30^{1/(\alpha-1)}(\alpha-1)(2/\alpha)^{\alpha/(\alpha-1)}+ 64(p-1)^\alpha\sigma^\alpha)K^{p(2-\alpha)/(3p\alpha+\alpha)}}{k+4}\\
&\le \frac{4(f(x^0) - f_{\mathrm{low}} + 4^{1/3}\sigma^\alpha + L_1/2 + 306 p^{2p\alpha} L_p^\alpha/(p!)^\alpha)\ln K}{K^{p/(3p+1)}}\\
&\qquad +\frac{8(30^{1/(\alpha-1)}(\alpha-1)(2/\alpha)^{\alpha/(\alpha-1)}+ 64(p-1)^\alpha\sigma^\alpha)\ln K}{K^{2p(\alpha-1)/(3p\alpha+\alpha)}} \overset{\eqref{def:t-Mpa-unk}\eqref{def:h-Mpa-unk}}{=} \frac{\widetilde{M}_{p,\alpha}\ln K}{K^{p/(3p+1)}} + \frac{\widehat{M}_{p,\alpha}\ln K}{K^{2p(\alpha-1)/(3p\alpha+\alpha)}},
\end{align*}
where the second inequality is due to $(K+3)^{(2p+1)/(3p+1)}\le 2K^{(2p+1)/(3p+1)}$ for all $K\ge5$, the third inequality follows from $2(2p+1)/(3p+1)\ge1$, $(p(6\alpha-\alpha^2-2)+\alpha)/(3p\alpha+\alpha)\ge1$, and $(2p(2\alpha-1)+\alpha)/(3p\alpha+\alpha)\le1$ for all $p\ge1$ and $\alpha\in(1,2]$, the fourth inequality is due to $(K+3)^{p(2-\alpha)/(3p\alpha+\alpha)}\le 2K^{p(2-\alpha)/(3p\alpha+\alpha)}$ for all $K\ge5$, and the last inequality follows from $\sum_{k=0}^{K-1}1/(k+4)\le \ln(2K/5+1) \leq 2\ln K$ for all $K\ge 5$ due to \eqref{upbd:series-ka}. Recall that $\iota_K$ is uniformly selected from $\{0,\ldots,K-1\}$. It then follows from this and the above relation that
\begin{align}\label{pre-upbd-mem-unk}
\E[\|\nabla f(x^{\iota_K})\|] = \frac{1}{K}\sum_{k=0}^{K-1}\mathbb{E}[\|\nabla f(x^k)\|] \le \frac{\widetilde{M}_{p,\alpha}\ln K}{K^{p/(3p+1)}} + \frac{\widehat{M}_{p,\alpha}\ln K}{K^{2p(\alpha-1)/(3p\alpha+\alpha)}}\qquad\forall K\ge 5.
\end{align}
By Lemma \ref{lem:rate-complexity} with $(\beta,u,v)=(p/(3p+1),p\epsilon/(4(3p+1)\widetilde{M}_{p,\alpha}),K)$ and $(\beta,u,v)=(2p(\alpha-1)/(3p\alpha+\alpha),p(\alpha-1)\epsilon/(2(3p\alpha+\alpha)\widehat{M}_{p,\alpha}),K)$, one can see that 
\begin{align*}
&K^{-p/(3p+1)}\ln K\le \frac{\epsilon}{2\widetilde{M}_{p,\alpha}}\qquad\forall K\ge\Big(\frac{4(3p+1)\widetilde{M}_{p,\alpha}}{p\epsilon}\ln\Big(\frac{4(3p+1)\widetilde{M}_{p,\alpha}}{p\epsilon}\Big)\Big)^{(3p+1)/p},\\
&K^{-2p(\alpha-1)/(3p\alpha+\alpha)}\ln K\le \frac{\epsilon}{2\widehat{M}_{p,\alpha}}\qquad\forall K\ge\Big(\frac{2(3p\alpha+\alpha)\widehat{M}_{p,\alpha}}{p(\alpha-1)\epsilon}\ln\Big(\frac{2(3p\alpha+\alpha)\widehat{M}_{p,\alpha}}{p(\alpha-1)\epsilon}\Big)\Big)^{(3p\alpha+\alpha)/(2p(\alpha-1))},
\end{align*}
which together with \eqref{pre-upbd-mem-unk} implies that Theorem \ref{thm:unknown-rate-mem} holds.
\end{proof}

\subsection{Proof of the main results in Section \ref{subsec:nsgd-rm}}\label{subsec:pf-rm}

In this subsection, we first establish several technical lemmas and then use them to prove Theorems \ref{thm:known-rate-rm} and \ref{thm:unknown-rate-rm}.  

The next lemma presents a recurrence relation for the estimation error of the gradient estimators $\{m^k\}$ generated by Algorithm \ref{alg:unf-sfom-rm}. 

\begin{lemma}\label{lem:rec-rm}
Suppose that Assumptions \ref{asp:basic} and \ref{asp:gen-ave-smth} hold. Let $\{(x^k,m^k)\}$ be the sequence generated by Algorithm \ref{alg:unf-sfom-rm} with input parameters $\{(\eta_k,\theta_k)\}$. Then we have
\begin{align}
\mathbb{E}_{\xi^{k+1}}[\|m^{k+1}-\nabla f(x^{k+1})\|^\alpha] \le (1-\theta_k)\|m^k-\nabla f(x^k)\|^\alpha + 6(L_1^\alpha+L^\alpha)\eta_k^\alpha + 6\sigma^\alpha\theta_k^\alpha \qquad \forall k\ge0,\label{ineq:vr-rm}
\end{align}
where $L_1$, $\sigma$, $\alpha$, and $L$ are given in Assumptions \ref{asp:basic} and \ref{asp:gen-ave-smth}, respectively.
\end{lemma}

\begin{proof}
Fix any $k\ge0$. It follows from \eqref{update-mk-rm} that 
\begin{align}
m^{k+1} - \nabla f(x^{k+1}) &= (1-\theta_k)(m^k - \nabla f(x^k)) + G(x^{k+1};\xi^{k+1}) - \nabla f(x^{k+1}) \nonumber\\ 
&\quad + (1-\theta_k)(\nabla f(x^k) - G(x^k;\xi^{k+1})).\label{rm-idt-vr}
\end{align}
Observe from Algorithm \ref{alg:unf-sfom-rm} and Assumptions \ref{asp:basic} and \ref{asp:gen-ave-smth} that $\|x^{k+1}-x^k\|=\eta_k$, $\E_{\xi^{k+1}}[G(x^{k+1};\xi^{k+1})-\nabla f(x^{k+1})]=0$, $\E_{\xi^{k+1}}[G(x^k;\xi^{k+1})-\nabla f(x^k)]=0$, $\mathbb{E}_{\xi^{k+1}}[\|\nabla f(x^k) - G(x^k;\xi^{k+1}))\|^\alpha]\le\sigma^\alpha$, $\|\nabla f(x^k)-\nabla f(x^{k+1})\|\leq L_1 \eta_k$, and $\mathbb{E}_{\xi^{k+1}}[\|G(x^{k+1};\xi^{k+1}) - G(x^k;\xi^{k+1})\|^\alpha] \leq L^\alpha \eta_k^\alpha$. Using these, \eqref{open-alpha}, and \eqref{rm-idt-vr}, we obtain that 
\begin{align*}
&\mathbb{E}_{\xi^{k+1}}[\|m^{k+1} - \nabla f(x^{k+1})\|^\alpha] \\
&\overset{\eqref{rm-idt-vr}}{=}\mathbb{E}_{\xi^{k+1}}[\|(1-\theta_k)(m^k - \nabla f(x^k)) + G(x^{k+1};\xi^{k+1}) - \nabla f(x^{k+1}) + (1-\theta_k)(\nabla f(x^k) - G(x^k;\xi^{k+1}))\|^\alpha] \\
&\le(1-\theta_k)^\alpha\|m^k-\nabla f(x^k)\|^\alpha \\
&\quad + 2\mathbb{E}_{\xi^{k+1}}[\|G(x^{k+1};\xi^{k+1}) - \nabla f(x^{k+1}) + (1-\theta_k)(\nabla f(x^k) - G(x^k;\xi^{k+1}))\|^\alpha]\\
&=(1-\theta_k)^\alpha\|m^k-\nabla f(x^k)\|^\alpha +  2\mathbb{E}_{\xi^{k+1}}\big[\|G(x^{k+1};\xi^{k+1}) - G(x^k;\xi^{k+1}) + \nabla f(x^k) - \nabla f(x^{k+1}) \\ 
& \quad - \theta_k(\nabla f(x^k) - G(x^k;\xi^{k+1}))\|^\alpha\big]\\
&\le(1-\theta_k)^\alpha\|m^k-\nabla f(x^k)\|^\alpha + 6\mathbb{E}_{\xi^{k+1}}\big[\|G(x^{k+1};\xi^{k+1}) - G(x^k;\xi^{k+1})\|^\alpha]+ 6\|\nabla f(x^{k+1})-\nabla f(x^k)\|^\alpha \\
&\quad+ 6\theta_k^\alpha \mathbb{E}_{\xi^{k+1}}[\|\nabla f(x^k) - G(x^k;\xi^{k+1}))\|^\alpha\big]\\
&\le (1-\theta_k)^\alpha\|m^k-\nabla f(x^k)\|^\alpha + 6(L_1^\alpha+L^\alpha)\eta_k^\alpha + 6\sigma^\alpha\theta_k^\alpha,
\end{align*}
where the first inequality is due to \eqref{open-alpha}, $\E_{\xi^{k+1}}[G(x^{k+1};\xi^{k+1})-\nabla f(x^{k+1})]=0$, and $\E_{\xi^{k+1}}[G(x^k;\xi^{k+1})-\nabla f(x^k)]=0$, the second inequality follows from $\|a+b+c\|^\alpha\le3(\|a\|^\alpha+\|b\|^\alpha+\|c\|^\alpha)$ for all $a,b,c\in\mathbb{R}^n$ due to $\alpha\in(1,2]$ and the convexity of $\|\cdot\|^{\alpha}$, and the last inequality is due to $\mathbb{E}_{\xi^{k+1}}[\|\nabla f(x^k) - G(x^k;\xi^{k+1}))\|^\alpha]\le\sigma^\alpha$, $\|\nabla f(x^k)-\nabla f(x^{k+1})\|\leq L_1 \eta_k$, and $\mathbb{E}_{\xi^{k+1}}[\|G(x^{k+1};\xi^{k+1}) - G(x^k;\xi^{k+1})\|^\alpha] \leq L^\alpha \eta_k^\alpha$. The above relation together with $\theta_k\in(0,1]$ and $\alpha\in(1,2]$ implies that this lemma holds.
\end{proof}

The following lemma establishes a descent property for the potential sequence $\{{\mathcal P}_k\}$ defined below.

\begin{lemma}\label{thm:stat-bd-rm}
Suppose that Assumptions \ref{asp:basic} and \ref{asp:gen-ave-smth} hold. Let $\{(x^k,m^k)\}$ be the sequence generated by Algorithm \ref{alg:unf-sfom-rm} with input parameters $\{(\eta_k,\theta_k)\}$. Let $L_1$, $\alpha$, and $\sigma$ be given in Assumption \ref{asp:basic}, $L$ be given in Assumption \ref{asp:gen-ave-smth}, and $\{{\mathcal P}_k\}$ be defined in \eqref{def:pot-pm} for $\{(x^k,m^k)\}$ and any positive sequence $\{p_k\}$ that satisfies $(1-\theta_k)p_{k+1} \le (1-\theta_k/2)p_k$ for all $k\ge0$. Then it holds that for all $k\geq 0$,
\begin{align}
&\mathbb{E}_{\xi^{k+1}}[{\mathcal P}_{k+1}]  \le {\mathcal P}_k- \eta_k\|\nabla f(x^k)\| + \frac{L_1}{2}\eta_k^2+ \frac{(\alpha-1)(2\eta_k)^{\alpha/(\alpha-1)}}{\alpha^{\alpha/(\alpha-1)}(\theta_kp_k/2)^{1/(\alpha-1)}} + 6(L_1^\alpha+L^\alpha)\eta_k^\alpha  p_{k+1} + 6\sigma^\alpha\theta_k^\alpha p_{k+1}.\label{stat-bd-rm}
\end{align}
\end{lemma}

\begin{proof}
Fix any $k\ge0$. By Lemma \ref{lem:ppt-desc} with $(x^+,x,m,\eta)=(x^{k+1},x^k,m^k,\eta_k)$, one has 
\begin{align}\label{upbd-fxk+1-rm}
f(x^{k+1}) \le f(x^k) - \eta_k\|\nabla f(x^k)\| + 2\eta_k\|\nabla f(x^k) - m^k\| + \frac{L_1}{2}\eta_k^2.    
\end{align}
Combining this with \eqref{def:pot-pm} and \eqref{ineq:vr-rm}, we obtain that 
\begin{align}
&\mathbb{E}_{\xi^{k+1}}[{\mathcal P}_{k+1}] \overset{\eqref{def:pot-pm}}{=} \mathbb{E}_{\xi^{k+1}}[f(x^{k+1}) + p_{k+1}\|m^{k+1} - \nabla f(x^{k+1})\|^\alpha] \nonumber\\
&\overset{\eqref{ineq:vr-rm}\eqref{upbd-fxk+1-rm}}{\le} f(x^k) - \eta_k\|\nabla f(x^k)\| + 2\eta_k\|\nabla f(x^k) - m^k\| + \frac{L_1}{2}\eta_k^2  \nonumber\\
&\qquad \ + (1-\theta_k)p_{k+1}\|m^k-\nabla f(x^k)\|^\alpha + 6(L_1^\alpha+L^\alpha)\eta_k^\alpha  p_{k+1} + 6\sigma^\alpha\theta_k^\alpha p_{k+1}\nonumber\\
&\le f(x^k) - \eta_k\|\nabla f(x^k)\| + 2\eta_k\|\nabla f(x^k) - m^k\| + \frac{L_1}{2}\eta_k^2 \nonumber\\
&\quad + (1-\theta_k/2)p_k\|m^k-\nabla f(x^k)\|^\alpha + 6(L_1^\alpha+L^\alpha)\eta_k^\alpha  p_{k+1} + 6\sigma^\alpha\theta_k^\alpha p_{k+1},\label{pf-pot-desc-rm}
\end{align}
where the last inequality follows from $(1-\theta_k)p_{k+1}\le(1-\theta_k/2)p_k$. In addition, letting $\alpha^\prime=\alpha/(\alpha-1)$, and using the Young's inequality, one has that
\begin{align*}
2\eta_k\|\nabla f(x^k) - m^k\| & \le \frac{((\alpha\theta_k p_k/2)^{1/\alpha} \|\nabla f(x^k) - m^k\|)^\alpha}{\alpha} + \frac{(2\eta_k/(\alpha\theta_k p_k/2)^{1/\alpha})^{\alpha^\prime}}{\alpha^\prime}\\
&= \frac{\theta_k p_k}{2} \|\nabla f(x^k) - m^k\|^\alpha + \frac{(\alpha-1)(2\eta_k)^{\alpha/(\alpha-1)}}{\alpha^{\alpha/(\alpha-1)}(\theta_kp_k/2)^{1/(\alpha-1)}}.
\end{align*}
This together with \eqref{pf-pot-desc-rm} implies that
\begin{align*}
&\mathbb{E}_{\xi^{k+1}}[{\mathcal P}_{k+1}]  \le f(x^k) + p_k\|\nabla f(x^k) - m^k\|^\alpha -\eta_k\|\nabla f(x^k)\| + \frac{L_1}{2}\eta_k^2+\frac{(\alpha-1)(2\eta_k)^{\alpha/(\alpha-1)}}{\alpha^{\alpha/(\alpha-1)}(\theta_kp_k/2)^{1/(\alpha-1)}} \\
&\qquad\qquad\qquad+ 6(L_1^\alpha+L^\alpha)\eta_k^\alpha  p_{k+1} + 6\sigma^\alpha\theta_k^\alpha p_{k+1} \\
&\le {\mathcal P}_k- \eta_k\|\nabla f(x^k)\| + \frac{L_1}{2}\eta_k^2  + \frac{(\alpha-1)(2\eta_k)^{\alpha/(\alpha-1)}}{\alpha^{\alpha/(\alpha-1)}(\theta_kp_k/2)^{1/(\alpha-1)}} + 6(L_1^\alpha+L^\alpha)\eta_k^\alpha  p_{k+1} + 6\sigma^\alpha\theta_k^\alpha p_{k+1}.   
\end{align*}    
By this and \eqref{def:pot-pm}, one can see that \eqref{stat-bd-rm} holds.
\end{proof}

The following lemma establishes some property for a specific choice of $\{(\theta_k, p_k)\}$, which will be used to prove Theorem \ref{thm:known-rate-rm} subsequently.

\begin{lemma}\label{lem:pk-know-rm}
Let $\{\theta_k\}$ be given in \eqref{rm-eta-theta}, and $\{p_k\}$ be defined as 
\begin{align}\label{def:pot-known-rm}
p_k= (k+1)^{(\alpha-1)^2/(2\alpha-1)}\qquad\forall k\ge0.    
\end{align}
Then $(1-\theta_k)p_{k+1}\le (1-\theta_k/2)p_k$ holds for all $k\ge0$.
\end{lemma}

\begin{proof}
Fix any $k\ge0$. It follows from \eqref{rm-eta-theta} that
\begin{align}\label{frac-pktk-rm}
\frac{1-\theta_k/2}{1-\theta_k} =1 + \frac{\theta_k}{2(1-\theta_k)} \ge 1 + \frac{\theta_k}{2} \overset{\eqref{rm-eta-theta}}{=} 1 + \frac{1}{2(k+1)^{\alpha/(2\alpha-1)}} \ge 1 + \frac{1}{2(k+1)},
\end{align}
where the last inequality is due to $\alpha/(2\alpha-1)<1$ for all $\alpha\in(1,2]$. In addition, notice from \eqref{def:pot-known-rm} that
\begin{align*}
\frac{p_{k+1}}{p_k} = \Big(1+\frac{1}{k+1} \Big)^{(\alpha-1)^2/(2\alpha-1)} \le \Big(1+\frac{1}{k+1} \Big)^{1/3} \le 1+\frac{1}{3(k+1)},   
\end{align*}
where the first inequality is due to $(\alpha-1)^2/(2\alpha-1)\le1/3$ for all $\alpha\in(1,2]$, and the last inequality is due to $(1+\tau)^\beta\le 1+\tau \beta$ for all $\tau>-1$ and $\beta\in[0,1]$. The above relation along with \eqref{frac-pktk-rm} implies that $(1-\theta_k)p_{k+1}\le (1-\theta_k/2)p_k$ holds.
\end{proof}

We are now ready to prove Theorem \ref{thm:known-rate-rm}.

\begin{proof}[\textbf{Proof of Theorem \ref{thm:known-rate-rm}}]
Let $\{(x^k,m^k)\}$ be generated by Algorithm \ref{alg:unf-sfom-rm} with $\{(\eta_k,\theta_k)\}$ defined in \eqref{rm-eta-theta}, and let 
$\{{\mathcal P}_k\}$ be defined in \eqref{def:pot-pm} with such $\{(x^k,m^k)\}$ and $\{p_k\}$ given in \eqref{def:pot-known-rm}. By Lemma \ref{lem:pk-know-rm}, one can see that such $\{(\eta_k,\theta_k,p_k)\}$ satisfies the assumptions in Lemma \ref{thm:stat-bd-rm} and Algorithm \ref{alg:unf-sfom-rm}. In addition, by \eqref{def:pot-pm} and \eqref{def:pot-known-rm}, one has that
\begin{align}
&\mathbb{E}[{\mathcal P}_0]= f(x^0)+p_0\mathbb{E}[\|m^0-\nabla f(x^0)\|^\alpha]= f(x^0) + \mathbb{E}[\|G(x^0;\xi^0)-\nabla f(x^0)\|^\alpha]\le f(x^0)+\sigma^\alpha,\label{upbd-exp-P0-rm}\\ 
&\mathbb{E}[{\mathcal P}_K]= \mathbb{E}[f(x^K)+p_K\|m^K-\nabla f(x^K)\|^\alpha]\ge \mathbb{E}[f(x^K)] \ge f_{\mathrm{low}}.\label{lwbd-exp-Pk-rm}
\end{align}
Taking expectation on both sides of \eqref{stat-bd-rm} with respect to $\{\xi^i\}_{i=0}^{k+1}$, we have {that for all $k\ge0$,}
\[
\mathbb{E}[{\mathcal P}_{k+1}]  \le \mathbb{E}[{\mathcal P}_k]- \eta_k\mathbb{E}[\|\nabla f(x^k)\|] + \frac{L_1}{2}\eta_k^2+ \frac{(\alpha-1)(2\eta_k)^{\alpha/(\alpha-1)}}{\alpha^{\alpha/(\alpha-1)}(\theta_kp_k/2)^{1/(\alpha-1)}} + 6(L_1^\alpha+L^\alpha)\eta_k^\alpha  p_{k+1} + 6\sigma^\alpha\theta_k^\alpha p_{k+1}.
\]
Summing up this inequality over $k=0,\ldots,K-1$, and using \eqref{upbd-exp-P0-rm} and \eqref{lwbd-exp-Pk-rm}, we {obtain that for all $K\ge1$,}
\begin{align}
f_{\mathrm{low}} & \overset{\eqref{lwbd-exp-Pk-rm}}{\le} \mathbb{E}[{\mathcal P}_K]\le \E[{\mathcal P}_0] - \sum_{k=0}^{K-1}\eta_k\mathbb{E}[\|\nabla f(x^k)\|] + \sum_{k=0}^{K-1}\Big(\frac{L_1}{2}\eta_k^2 + \frac{(\alpha-1)(2\eta_k)^{\alpha/(\alpha-1)}}{\alpha^{\alpha/(\alpha-1)}(\theta_kp_k/2)^{1/(\alpha-1)}} \nonumber \\ 
&\qquad+ 6(L_1^\alpha+L^\alpha)\eta_k^\alpha  p_{k+1} + 6\sigma^\alpha\theta_k^\alpha p_{k+1}\Big)\nonumber\\
&\overset{\eqref{upbd-exp-P0-rm}}{\le} f(x^0) + \sigma^\alpha - \eta_{K-1}\sum_{k=0}^{K-1}\mathbb{E}[\|\nabla f(x^k)\|]  + \sum_{k=0}^{K-1}\Big(\frac{L_1}{2}\eta_k^2 + \frac{(\alpha-1)(2\eta_k)^{\alpha/(\alpha-1)}}{\alpha^{\alpha/(\alpha-1)}(\theta_kp_k/2)^{1/(\alpha-1)}} \nonumber \\
&\qquad + 12(L_1^\alpha+L^\alpha)\eta_k^\alpha p_k + 12\sigma^\alpha\theta_k^\alpha p_k\Big),\label{rear-pot-desc-rm-known}
\end{align}
where the last inequality follows from \eqref{upbd-exp-P0-rm} and the fact that $\{\eta_k\}$ is nonincreasing and $p_{k+1}\le 2p_k$ for all $k\ge0$. Rearranging the terms in \eqref{rear-pot-desc-rm-known}, and using \eqref{Ma-rm}, \eqref{rm-eta-theta}, and \eqref{def:pot-known-rm}, we obtain that for all $K\ge3$,
\begin{align*}
&\frac{1}{K}\sum_{k=0}^{K-1}\mathbb{E}[\|\nabla f(x^k)\|] \overset{\eqref{rear-pot-desc-rm-known}}{\le} \frac{f(x^0) - f_{\mathrm{low}}+\sigma^\alpha}{K\eta_{K-1}}\\
&\qquad + \frac{1}{K\eta_{K-1}}\sum_{k=0}^{K-1}\Big(\frac{L_1}{2}\eta_k^2 + \frac{(\alpha-1)(2\eta_k)^{\alpha/(\alpha-1)}}{\alpha^{\alpha/(\alpha-1)}(\theta_kp_k/2)^{1/(\alpha-1)}} + 12(L_1^\alpha+L^\alpha)\eta_k^\alpha p_k + 12\sigma^\alpha\theta_k^\alpha p_k\Big)\\
&\overset{\eqref{rm-eta-theta}\eqref{def:pot-known-rm}}{=} \frac{f(x^0) - f_{\mathrm{low}}+\sigma^\alpha}{K^{(\alpha-1)/(2\alpha-1)}} \\
&\qquad + \frac{1}{K^{(\alpha-1)/(2\alpha-1)}}\sum_{k=0}^{K-1}\Big(\frac{L_1}{2(k+1)^{2\alpha/(2\alpha-1)}} + \frac{2^{1/(\alpha-1)}(\alpha-1)(2/\alpha)^{\alpha/(\alpha-1)} + 12(L_1^\alpha+L^\alpha) + 12\sigma^\alpha}{k+1}\Big)\\
& \le \frac{f(x^0) - f_{\mathrm{low}}+\sigma^\alpha}{K^{(\alpha-1)/(2\alpha-1)}} + \frac{L_1/2 + 2^{1/(\alpha-1)}(\alpha-1)(2/\alpha)^{\alpha/(\alpha-1)} + 12(L_1^\alpha+L^\alpha) + 12\sigma^\alpha}{K^{(\alpha-1)/(2\alpha-1)}}\sum_{k=0}^{K-1}\frac{1}{k+1}\\
&\le \frac{2(f(x^0) - f_{\mathrm{low}}+\sigma^\alpha+ L_1/2 + 2^{1/(\alpha-1)}(\alpha-1)(2/\alpha)^{\alpha/(\alpha-1)} + 12(L_1^\alpha+L^\alpha) + 12\sigma^\alpha)\ln K}{K^{(\alpha-1)/(2\alpha-1)}}\\
&\overset{\eqref{Ma-rm}}{=} \frac{M_\alpha\ln K}{K^{(\alpha-1)/(2\alpha-1)}},
\end{align*}
where the second inequality follows from $2\alpha/(2\alpha-1)>1$, the third inequality follows from $\sum_{k=0}^{K-1}1/(k+1)\le \ln(2K+1)\le 2\ln K$ due to \eqref{upbd:series-ka} and $K\ge 3$. Recall that $\iota_K$ is uniformly selected from $\{0,\ldots,K-1\}$. It follows from this and the above relation that
\begin{align}\label{pre-upbd-rm}
\E[\|\nabla f(x^{\iota_K})\|] = \frac{1}{K}\sum_{k=0}^{K-1}\mathbb{E}[\|\nabla f(x^k)\|] \le \frac{M_{\alpha}\ln K}{K^{(\alpha-1)/(2\alpha-1)}}\qquad\forall K\ge 3.
\end{align}
By Lemma \ref{lem:rate-complexity} with $(\beta,u,v)=((\alpha-1)/(2\alpha-1),(\alpha-1)\epsilon/(2(2\alpha-1)M_{\alpha}),K)$, one can see that 
\begin{align*}
K^{-(\alpha-1)/(2\alpha-1)}\ln K\le \frac{\epsilon}{M_{\alpha}}\qquad\forall K\ge\Big(\frac{2(2\alpha-1)M_{\alpha}}{(\alpha-1)\epsilon}\ln\Big(\frac{2(2\alpha-1)M_{\alpha}}{(\alpha-1)\epsilon}\Big)\Big)^{(2\alpha-1)/(\alpha-1)},    
\end{align*}
which together with \eqref{pre-upbd-rm} implies that Theorem \ref{thm:known-rate-rm} holds.
\end{proof}

The next lemma establishes some property for a specific choice of $\{(\theta_k, p_k)\}$, which will be used to prove Theorem \ref{thm:unknown-rate-rm} subsequently.

\begin{lemma}\label{lem:pk-unknow-rm}
Let $\{\theta_k\}$ be given in \eqref{rm-eta-theta-un}, and $\{p_k\}$ be defined as  
\begin{align}\label{def:pk-unknown-rm}
p_k= (k+1)^{2(\alpha-1)^2/(3\alpha)}\qquad\forall k\ge0.   
\end{align}
Then $(1-\theta_k)p_{k+1}\le (1-\theta_k/2)p_k$ holds for all $k\ge0$.
\end{lemma}

\begin{proof}
Fix any $k\ge0$. Observe that
\begin{align}\label{frac-pktk-rm-u}
\frac{1-\theta_k/2}{1-\theta_k} = 1 + \frac{\theta_k}{2(1-\theta_k)} \ge 1 + \frac{\theta_k}{2} \overset{\eqref{rm-eta-theta-un}}{=} 1 + \frac{1}{2(k+1)^{2/3}} \ge 1 + \frac{1}{2(k+1)},
\end{align}
In addition, notice from \eqref{def:pk-unknown-rm} that
\begin{align*}
\frac{p_{k+1}}{p_k} = \Big(1+\frac{1}{k+1} \Big)^{2(\alpha-1)^2/(3\alpha)} \le \Big(1+\frac{1}{k+1} \Big)^{1/3} \le 1+\frac{1}{3(k+1)},   
\end{align*}
where the first inequality is due to $2(\alpha-1)^2/(3\alpha)\le1/3$ for all $\alpha\in(1,2]$, and the last inequality is due to $(1+\tau)^\beta\le 1+\tau \beta$ for all $\tau>-1$ and $\beta\in[0,1]$. The above relation together with \eqref{frac-pktk-rm-u} implies that $(1-\theta_k)p_{k+1}\le (1-\theta_k/2)p_k$ holds.
\end{proof}

We are now ready to prove Theorem \ref{thm:unknown-rate-rm}.
\begin{proof}[\textbf{Proof of Theorem \ref{thm:unknown-rate-rm}}]
Let $\{(x^k,m^k)\}$ be generated by Algorithm \ref{alg:unf-sfom-rm} with $\{(\eta_k,\theta_k)\}$ defined in \eqref{rm-eta-theta-un}, and let 
$\{{\mathcal P}_k\}$ be defined in \eqref{def:pot-pm} with such $\{(x^k,m^k)\}$ and $\{p_k\}$ given in \eqref{def:pk-unknown-rm}. By Lemma \ref{lem:pk-unknow-rm}, one can see that such $\{(\eta_k,\theta_k,p_k)\}$ satisfies the assumptions in Lemma \ref{thm:stat-bd-rm} and Algorithm \ref{alg:unf-sfom-rm}. Using this and similar arguments as those for deriving \eqref{rear-pot-desc-rm-known}, we have {that for all $K\ge1$,}
\begin{align}
&f_{\mathrm{low}}\le f(x^0) + \sigma^\alpha - \eta_{K-1}\sum_{k=0}^{K-1}\mathbb{E}[\|\nabla f(x^k)\|] \nonumber\\
&\qquad\ \ + \sum_{k=0}^{K-1}\Big(\frac{L_1}{2}\eta_k^2 + \frac{(\alpha-1)(2\eta_k)^{\alpha/(\alpha-1)}}{\alpha^{\alpha/(\alpha-1)}(\theta_kp_k/2)^{1/(\alpha-1)}} + 12(L_1^\alpha+L^\alpha)\eta_k^\alpha p_k + 12\sigma^\alpha\theta_k^\alpha p_k\Big). \label{bd-unk-rm}
\end{align}
Rearranging the terms in \eqref{bd-unk-rm},  and using \eqref{Ma-tilde-rm}, \eqref{Ma-hat-rm}, \eqref{rm-eta-theta-un}, and \eqref{def:pk-unknown-rm}, we obtain that for all $K\ge3$,
\begin{align*}
&\frac{1}{K}\sum_{k=0}^{K-1}\mathbb{E}[\|\nabla f(x^k)\|] \overset{\eqref{bd-unk-rm}}{\le} \frac{f(x^0) - f_{\mathrm{low}} + \sigma^\alpha}{K\eta_{K-1}}\\
&\qquad\qquad + \frac{1}{K\eta_{K-1}} \sum_{k=0}^{K-1}\Big(\frac{L_1}{2}\eta_k^2 + \frac{(\alpha-1)(2\eta_k)^{\alpha/(\alpha-1)}}{\alpha^{\alpha/(\alpha-1)}(\theta_kp_k/2)^{1/(\alpha-1)}} + 12(L_1^\alpha+L^\alpha)\eta_k^\alpha p_k + 12\sigma^\alpha\theta_k^\alpha p_k\Big)\\
&\overset{\eqref{rm-eta-theta-un}\eqref{def:pk-unknown-rm}}{=}\frac{f(x^0) - f_{\mathrm{low}} + \sigma^\alpha}{K^{1/3}} \\
&\quad + \frac{1}{K^{1/3}}\sum_{k=0}^{K-1}\Big(\frac{L_1}{2(k+1)^{4/3}} + \frac{2^{1/(\alpha-1)}(\alpha-1)(2/\alpha)^{\alpha/(\alpha-1)}+ 12(L_1^\alpha+L^\alpha) + 12\sigma^\alpha}{(k+1)^{2(2\alpha-1)/(3\alpha)}}\Big)\\
&\le \frac{f(x^0) - f_{\mathrm{low}} + \sigma^\alpha}{K^{1/3}}\\
&\quad + \frac{1}{K^{1/3}}\sum_{k=0}^{K-1}\Big(\frac{L_1}{2(k+1)} + \frac{(2^{1/(\alpha-1)}(\alpha-1)(2/\alpha)^{\alpha/(\alpha-1)}+ 12(L_1^\alpha+L^\alpha) + 12\sigma^\alpha)K^{(2-\alpha)/(3\alpha)}}{k+1}\Big)\\
&\le \frac{2(f(x^0) - f_{\mathrm{low}} + \sigma^\alpha + L_1/2)\ln K}{K^{1/3}} + \frac{2(2^{1/(\alpha-1)}(\alpha-1)(2/\alpha)^{\alpha/(\alpha-1)}+ 12(L_1^\alpha+L^\alpha) + 12\sigma^\alpha)\ln K}{K^{2(\alpha-1)/(3\alpha)}}\\
&\overset{\eqref{Ma-tilde-rm}\eqref{Ma-hat-rm}}{=}\frac{\widetilde{M}_\alpha \ln K}{K^{1/3}} + \frac{\widehat{M}_\alpha \ln K}{K^{2(\alpha-1)/(3\alpha)}},
\end{align*}
where the second inequality follows from $2(2\alpha-1)/(3\alpha)\le 1$ for all $\alpha\in(1,2]$, and the last inequality follows from $\sum_{k=0}^{K-1}1/(k+1)\le 2\ln K$ due to \eqref{upbd:series-ka} and $K\ge 3$. Recall that $\iota_K$ is uniformly selected from $\{0,\ldots,K-1\}$. It follows from this and the above relation that
\begin{align}\label{pre-upbd-rm-unk}
\E[\|\nabla f(x^{\iota_K})\|] = \frac{1}{K}\sum_{k=0}^{K-1}\mathbb{E}[\|\nabla f(x^k)\|] \le \frac{\widetilde{M}_{\alpha}\ln K}{K^{1/3}} + \frac{\widehat{M}_{\alpha}\ln K}{K^{2(\alpha-1)/(3\alpha)}}\qquad\forall K\ge 3.
\end{align}
By Lemma \ref{lem:rate-complexity} with $(\beta,u,v)=(1/3,\epsilon/(12\widetilde{M}_{\alpha}),K)$ and $(\beta,u,v)=(2(\alpha-1)/(3\alpha),(\alpha-1)\epsilon/(6\alpha\widehat{M}_{\alpha}),K)$, one can see that 
\begin{align*}
&K^{-1/3}\ln K\le \frac{\epsilon}{2\widetilde{M}_{\alpha}}\qquad\forall K\ge\Big(\frac{12\widetilde{M}_{\alpha}}{\epsilon}\ln\Big(\frac{12\widetilde{M}_{\alpha}}{\epsilon}\Big)\Big)^{3},\\
&K^{-2(\alpha-1)/(3\alpha)}\ln K\le \frac{\epsilon}{2\widehat{M}_{\alpha}}\qquad\forall K\ge\Big(\frac{6\alpha\widehat{M}_{\alpha}}{(\alpha-1)\epsilon}\ln\Big(\frac{6\alpha\widehat{M}_{\alpha}}{(\alpha-1)\epsilon}\Big)\Big)^{3\alpha/(2(\alpha-1))},
\end{align*}
which together with \eqref{pre-upbd-rm-unk} implies that Theorem \ref{thm:unknown-rate-rm} holds.
\end{proof}

\appendix

\bibliographystyle{abbrv}
\bibliography{ref}

@book{horn2012matrix,
  title={Matrix Analysis},
  author={Horn, Roger A and Johnson, Charles R},
  year={2012},
  publisher={Cambridge university press}
}

@article{arjevani2023lower,
  title={Lower bounds for non-convex stochastic optimization},
  author={Arjevani, Yossi and Carmon, Yair and Duchi, John C and Foster, Dylan J and Srebro, Nathan and Woodworth, Blake},
  journal={Mathematical Programming},
  volume={199},
  number={1},
  pages={165--214},
  year={2023}
}

@article{ghadimi2013stochastic,
  title={Stochastic first-and zeroth-order methods for nonconvex stochastic programming},
  author={Ghadimi, Saeed and Lan, Guanghui},
  journal={SIAM Journal on Optimization},
  volume={23},
  number={4},
  pages={2341--2368},
  year={2013}
}

@inproceedings{gao2024non,
  title={Non-Convex Stochastic Composite Optimization with {P}olyak Momentum},
  author={Gao, Yuan and Rodomanov, Anton and Stich, Sebastian U},
  booktitle={International Conference on Machine Learning},
  volume= {235},
  pages= {14826--14843},
  year={2024}
}

@inproceedings{cutkosky2020momentum,
  title={Momentum improves normalized {SGD}},
  author={Cutkosky, Ashok and Mehta, Harsh},
  booktitle={International Conference on Machine Learning},
  pages={2260--2268},
  year={2020}
}

@article{ghadimi2016accelerated,
  title={Accelerated gradient methods for nonconvex nonlinear and stochastic programming},
  author={Ghadimi, Saeed and Lan, Guanghui},
  journal={Mathematical Programming},
  volume={156},
  number={1},
  pages={59--99},
  year={2016}
}

@inproceedings{cutkosky2019momentum,
  title={Momentum-based variance reduction in non-convex {SGD}},
  author={Cutkosky, Ashok and Orabona, Francesco},
  booktitle={Advances in Neural Information Processing Systems},
  volume={32},
  year={2019}
}

@inproceedings{fang2018spider,
  title={Spider: Near-optimal non-convex optimization via stochastic path-integrated differential estimator},
  author={Fang, Cong and Li, Chris Junchi and Lin, Zhouchen and Zhang, Tong},
  booktitle={Advances in Neural Information Processing Systems},
  volume={31},
  year={2018}
}

@inproceedings{lei2017non,
  title={Non-convex finite-sum optimization via {SCSG} methods},
  author={Lei, Lihua and Ju, Cheng and Chen, Jianbo and Jordan, Michael I},
  booktitle={Advances in Neural Information Processing Systems},
  volume={30},
  year={2017}
}

@inproceedings{nguyen2017sarah,
  title={{SARAH}: A novel method for machine learning problems using stochastic recursive gradient},
  author={Nguyen, Lam M and Liu, Jie and Scheinberg, Katya and Tak{\'a}{\v{c}}, Martin},
  booktitle={International Conference on Machine Learning},
  pages={2613--2621},
  year={2017}
}

@inproceedings{carmon2017convex,
  title={“{C}onvex until proven guilty”: Dimension-free acceleration of gradient descent on non-convex functions},
  author={Carmon, Yair and Duchi, John C and Hinder, Oliver and Sidford, Aaron},
  booktitle={International Conference on Machine Learning},
  pages={654--663},
  year={2017}
}

@book{humpherys2020foundations,
  title={Foundations of Applied Mathematics Volume 2: Algorithms, Approximation, Optimization},
  author={Humpherys, Jeffrey and Jarvis, Tyler J},
  year={2020},
  publisher={SIAM}
}

@inproceedings{liu2025nonconvex,
title={Nonconvex Stochastic Optimization under Heavy-Tailed Noises: Optimal Convergence without Gradient Clipping},
author={Zijian Liu and Zhengyuan Zhou},
booktitle={International Conference on Learning Representations},
year={2025}
}

@article{rodomanov2020smoothness,
  title={Smoothness parameter of power of {E}uclidean norm},
  author={Rodomanov, Anton and Nesterov, Yurii},
  journal={Journal of Optimization Theory and Applications},
  volume={185},
  pages={303--326},
  year={2020}
}

@inproceedings{zhang2020adaptive,
  title={Why are adaptive methods good for attention models?},
  author={Zhang, Jingzhao and Karimireddy, Sai Praneeth and Veit, Andreas and Kim, Seungyeon and Reddi, Sashank and Kumar, Sanjiv and Sra, Suvrit},
  booktitle={Advances in Neural Information Processing Systems},
  volume={33},
  pages={15383--15393},
  year={2020}
}

@article{klamkin1970extensions,
  title={Extensions of the {W}eierstrass product inequalities},
  author={Klamkin, MS and Newman, D Jb},
  journal={Mathematics Magazine},
  volume={43},
  number={3},
  pages={137--141},
  year={1970},
  publisher={Taylor \& Francis}
}

@inproceedings{simsekli2019tail,
  title={A tail-index analysis of stochastic gradient noise in deep neural networks},
  author={Simsekli, Umut and Sagun, Levent and Gurbuzbalaban, Mert},
  booktitle={International Conference on Machine Learning},
  pages={5827--5837},
  year={2019}
}

@article{simsekli2019heavy,
  title={On the heavy-tailed theory of stochastic gradient descent for deep neural networks},
  author={Simsekli, Umut and G{\"u}rb{\"u}zbalaban, Mert and Nguyen, Thanh Huy and Richard, Ga{\"e}l and Sagun, Levent},
  journal={arXiv preprint arXiv:1912.00018},
  volume={222},
  year={2019},
  publisher={Nov}
}

@inproceedings{li2021page,
  title={{PAGE}: A simple and optimal probabilistic gradient estimator for nonconvex optimization},
  author={Li, Zhize and Bao, Hongyan and Zhang, Xiangliang and Richt{\'a}rik, Peter},
  booktitle={International Conference on Machine Learning},
  pages={6286--6295},
  year={2021}
}

@inproceedings{NEURIPS2023_4c454d34,
author = {Nguyen, Ta Duy and Nguyen, Thien H and Ene, Alina and Nguyen, Huy},
booktitle = {Advances in Neural Information Processing Systems},
pages = {24191--24222},
title = {Improved Convergence in High Probability of Clipped Gradient Methods with Heavy Tailed Noise},
volume = {36},
year = {2023}
}

@inproceedings{cutkosky2021high,
  title={High-probability bounds for non-convex stochastic optimization with heavy tails},
  author={Cutkosky, Ashok and Mehta, Harsh},
  booktitle={Advances in Neural Information Processing Systems},
  volume={34},
  pages={4883--4895},
  year={2021}
}

@inproceedings{sadiev2023high,
  title={High-probability bounds for stochastic optimization and variational inequalities: the case of unbounded variance},
  author={Sadiev, Abdurakhmon and Danilova, Marina and Gorbunov, Eduard and Horv{\'a}th, Samuel and Gidel, Gauthier and Dvurechensky, Pavel and Gasnikov, Alexander and Richt{\'a}rik, Peter},
  booktitle={International Conference on Machine Learning},
  pages={29563--29648},
  year={2023}
}

@inproceedings{liu2024high,
  title={High-probability bound for non-smooth non-convex stochastic optimization with heavy tails},
  author={Liu, Langqi and Wang, Yibo and Zhang, Lijun},
  booktitle={International Conference on Machine Learning},
  year={2024}
}

@inproceedings{liu2023breaking,
  title={Breaking the lower bound with (little) structure: Acceleration in non-convex stochastic optimization with heavy-tailed noise},
  author={Liu, Zijian and Zhang, Jiawei and Zhou, Zhengyuan},
  booktitle={Conference on Learning Theory},
  pages={2266--2290},
  year={2023}
}

@inproceedings{hubler2024gradient,
title={From Gradient Clipping to Normalization for Heavy Tailed {SGD}},
author={Florian H{\"u}bler and Ilyas Fatkhullin and Niao He},
booktitle={International Conference on Artificial Intelligence and Statistics},
year={2025}
}

@article{sun2024gradient,
  author  = {Tao Sun and Xinwang Liu and Kun Yuan},
  title   = {Revisiting Gradient Normalization and Clipping for Nonconvex SGD under Heavy-Tailed Noise: Necessity, Sufficiency, and Acceleration},
  journal = {Journal of Machine Learning Research},
  year    = {2025},
  volume  = {26},
  number  = {237},
  pages   = {1--42}
}

@article{touvron2023llama,
  title={{LLaMA}: Open and efficient foundation language models},
  author={Touvron, Hugo and Lavril, Thibaut and Izacard, Gautier and Martinet, Xavier and Lachaux, Marie-Anne and Lacroix, Timoth{\'e}e and Rozi{\`e}re, Baptiste and Goyal, Naman and Hambro, Eric and Azhar, Faisal and others},
  journal={arXiv preprint arXiv:2302.13971},
  year={2023}
}

@inproceedings{radford2021learning,
  title={Learning transferable visual models from natural language supervision},
  author={Radford, Alec and Kim, Jong Wook and Hallacy, Chris and Ramesh, Aditya and Goh, Gabriel and Agarwal, Sandhini and Sastry, Girish and Askell, Amanda and Mishkin, Pamela and Clark, Jack and others},
  booktitle={International Conference on Machine Learning},
  pages={8748--8763},
  year={2021}
}

@inproceedings{plummer2015flickr30k,
  title={Flickr30k entities: Collecting region-to-phrase correspondences for richer image-to-sentence models},
  author={Plummer, Bryan A and Wang, Liwei and Cervantes, Chris M and Caicedo, Juan C and Hockenmaier, Julia and Lazebnik, Svetlana},
  booktitle={International Conference on Computer Vision},
  pages={2641--2649},
  year={2015}
}

@inproceedings{lin2014microsoft,
  title={Microsoft {COCO}: Common objects in context},
  author={Lin, Tsung-Yi and Maire, Michael and Belongie, Serge and Hays, James and Perona, Pietro and Ramanan, Deva and Doll{\'a}r, Piotr and Zitnick, C Lawrence},
  booktitle={European Conference on Computer Vision},
  pages={740--755},
  year={2014}
}

@inproceedings{sharma2018conceptual,
  title={Conceptual captions: A cleaned, hypernymed, image alt-text dataset for automatic image captioning},
  author={Sharma, Piyush and Ding, Nan and Goodman, Sebastian and Soricut, Radu},
  booktitle={Annual Meeting of the Association for Computational Linguistics},
  pages={2556--2565},
  year={2018}
}

@inproceedings{he2016deep,
  title={Deep residual learning for image recognition},
  author={He, Kaiming and Zhang, Xiangyu and Ren, Shaoqing and Sun, Jian},
  booktitle={Computer Vision and Pattern Recognition},
  pages={770--778},
  year={2016}
}

@article{sanh2019distilbert,
  title={Distil{BERT}, a distilled version of {BERT}: smaller, faster, cheaper and lighter},
  author={Sanh, Victor and Debut, Lysandre and Chaumond, Julien and Wolf, Thomas},
  journal={arXiv preprint arXiv:1910.01108},
  year={2019}
}

@article{
khaled2022better,
title={Better Theory for {SGD} in the  Nonconvex World},
author={Ahmed Khaled and Peter Richt{\'a}rik},
journal={Transactions on Machine Learning Research},
issn={2835-8856},
year={2023},
note={Survey Certification}
}

@article{bottou2018optimization,
  title={Optimization methods for large-scale machine learning},
  author={Bottou, L{\'e}on and Curtis, Frank E and Nocedal, Jorge},
  journal={SIAM Review},
  volume={60},
  number={2},
  pages={223--311},
  year={2018},
  publisher={SIAM}
}

@inproceedings{gurbuzbalaban2021heavy,
  title={The heavy-tail phenomenon in {SGD}},
  author={Gurbuzbalaban, Mert and Simsekli, Umut and Zhu, Lingjiong},
  booktitle={International Conference on Machine Learning},
  pages={3964--3975},
  year={2021}
}

\end{document}